\documentclass[11pt,reqno]{amsart}
\numberwithin{equation}{section}
\usepackage{amsmath,amsthm}
\usepackage{mathptmx}  
\usepackage{amssymb}
\usepackage{multicol}
\DeclareFontFamily{U}{BOONDOX-calo}{\skewchar\font=45 }
\DeclareFontShape{U}{BOONDOX-calo}{m}{n}{
  <-> s*[1.05] BOONDOX-r-calo}{}
\DeclareFontShape{U}{BOONDOX-calo}{b}{n}{
  <-> s*[1.05] BOONDOX-b-calo}{}
\DeclareMathAlphabet{\mathcalboondox}{U}{BOONDOX-calo}{m}{n}
\SetMathAlphabet{\mathcalboondox}{bold}{U}{BOONDOX-calo}{b}{n}
\DeclareMathAlphabet{\mathbcalboondox}{U}{BOONDOX-calo}{b}{n}
\usepackage[utopia]{mathdesign}
\usepackage{mathtools}
\usepackage{mhchem}
\usepackage{enumitem}
\usepackage{xfrac}

\usepackage{tikz}
\usetikzlibrary{arrows.meta, positioning, shapes.geometric}

\newcommand{\mcb}[1]{{\mathcalboondox #1}}
\usepackage{amssymb}
\usepackage{amsmath}
\usepackage{amsthm}
\usepackage{diagbox}
\usepackage{booktabs}
\usepackage{enumitem}
\usepackage[sans]{dsfont} 
\usepackage{bbm}
\usepackage{changebar}
\usepackage[colorlinks]{hyperref}
\usepackage{a4wide}

\usepackage{tikz}
\usetikzlibrary{arrows,shapes,automata,petri,positioning,calc}
\tikzset{
    place/.style={
        circle,
        thick,
        draw=black,
        fill=gray!50,
        minimum size=20mm,
    },
        state/.style={
        circle,
        thick,
        draw=blue!75,
        fill=blue!20,
        minimum size=20mm,
    },
}

\tikzset{
    cross/.pic = {
    \draw[rotate = 45] (-0.2,0) -- (0.2,0);
    \draw[rotate = 45] (0,-0.2) -- (0, 0.2);
    }
}

\usepackage{ulem}
\usepackage{setspace}

\newtheorem{thm}{Theorem}[section]
\newtheorem{lem}[thm]{Lemma}
\newtheorem{cor}[thm]{Corollary}

\newtheorem{prop}[thm]{Proposition}

\newtheorem{definition}[thm]{Definition}

\newtheorem{hyp}{Hypothesis}

\newtheorem{rem}[thm]{Remark}

\newcommand\supp{\mathrm{supp}}

\usepackage{comment}

\title[{Regional Fractional Stochastic Burgers from random interactions}
]{Regional Fractional Stochastic Burgers\\ from random interactions}
\author{Pedro Cardoso, Patr\'icia   Gon\c calves}

\begin{document}
\begin{abstract}

The purpose of this article is to derive the crossover from the Ornstein-Uhlenbeck process  to energy solutions of the stochastic Burgers equation with characteristic operators given in terms of fractional operators, such as the regional fractional Laplacian. The approach is to consider a boundary driven exclusion process with long jumps and asymmetric jump rates. Depending on the strength of the asymmetry we prove the convergence to stationary solutions of either the Ornstein-Uhlenbeck equation, or the stochastic Burgers equation. In the later setting, the convergence in some regimes is guaranteed by the recent proof of uniqueness of energy solutions derived in \cite{GPP}.

\end{abstract}
\maketitle


\section{Introduction}

\subsection{The macroscopic level} \label{intmac}
Over the last years there has been a tremendous development in the field of singular stochastic partial differential equations (SPDEs). One remarkable example is the KPZ equation, which was proposed in 86' by three physicists, Kardar, Parisi and Zhang in \cite{KPZ} as a universal law ruling the evolution of the
profile of a randomly growing interface, and can be described as follows.
For $t \geq 0$ and $x \in \mathbb{R}$, if $ h(t,x)$ denotes the height of that interface at time $t$
and position $x$, then the KPZ equation reads as
\begin{equation}\label{eq:KPZ}
\partial_t h(t,x) = A\partial_x^2 h(t,x) dt + B[\partial_x h(t,x)]^2 dt + \sqrt C \mathcal W_t,
\end{equation}
 where  $A,B$ and $C$ are constants that depend on the thermodynamical quantities of the interface (see for example \cite{GJ4}) and $\mathcal W_t$ is a space-time white noise. There is another singular SPDE related to the KPZ equation, that is known as the Stochastic Burgers equation (SBE). Its solution can be obtained, at least formally, from the solution $h$ of the  KPZ equation by taking its space derivative, namely, defining 
 $Y_t=\partial_x h_t$. In this case, $Y_t$ solves the following equation
 \begin{equation} \label{SBEintro}
dY_t = A \partial_x^2Y_t dt + B \partial_x Y_t^2 dt + \sqrt C \partial_x \mathcal W_t.
 \end{equation}
We stress that both \eqref{eq:KPZ} and \eqref{SBEintro} are singular due to the fact that they contain a non-linear term, and their solutions are not functions, but rather random distributions. For instance, the second term on the right-hand side of \eqref{eq:KPZ} is quadratic, but defining the product of random distributions in a precise way is often nontrivial. Some advances in this direction were achieved by the pioneering work of Hairer in \cite{H1,H2,H3} with the theory of regularity structures, which opened a way to make sense of several singular SPDEs, such as the KPZ equation. 

Another approach is the later work of Gubinelli and Perkwoski, who developed the theory of paracontrolled in \cite{GP}. There, it was possible to show the well-posedness not only for \eqref{eq:KPZ} but also for a multitude of SPDEs with more general characteristic operators. We highlight that the notion of solutions to the KPZ equation   used by Gubinelli and Perkwoski is based on the definition of energy solutions, which was first cooked up by Gon\c calves and Jara in \cite{asymjara} (and afterwards refined in \cite{GuJ}); this was done by looking at the fluctuations of  a collection of weakly asymmetric exclusion processes in a one-dimensional lattice and at equilibrium. 

An energy solution of a SPDE such as the KPZ equation is a random distribution $Y_t$, which is continuous in time and  satisfies a martingale problem that contains an integral term corresponding to the non-linear term of the equation. Some energy estimates ensure that the aforementioned non-linear term is well-defined. The reader is invited to look at Definition \ref{defspdefbe} for more details. 

This notion of solution is well adapted to the derivation of fluctuations from microscopic random systems; they are governed by energy solutions to the KPZ equation for various models described in many recent works, see \cite{GHS} and references therein. This has provided advances for proving the weak KPZ universality conjecture, which states that the KPZ equation (or its companion, the SBE) is an universal law ruling the fluctuations of several random growth interfaces close to a stationary state.  

A natural question that arose afterwards is regarding the derivation of other singular SPDEs from scaling limits of random microscopic systems. This was done in \cite{jarafluc}, where it was obtained an energy solution to a fractional stochastic Burgers equation (or its companion the fractional KPZ equation) from an exclusion dynamics that allows long jumps. There, the characteristic operators of the equation replace the Laplacian and the derivative operators (which are present in the classical KPZ equation) by their fractional versions. Since the particles were evolving in the set of integers $\mathbb Z$, the SPDEs in \cite{jarafluc} were stated without boundary conditions. This was not the case in \cite{GPS}, where it was possible to produce energy solutions to the KPZ/SBE equation with Neumann/Dirichlet boundary conditions.

At this point it worth to note that by taking $B=0$ in \eqref{SBEintro}, one obtains the Ornstein-Uhlenbeck  equation. This corresponds to a Gaussian process that is typically obtained in the fluctuations of a variety of microscopic systems whose asymmetry is missing (such as  the symmetric simple exclusion process), or negligible. Some different notions of solutions for this process were also described in \cite{jarafluc, GPS}, among another aforementioned works. 

In this article we derive for the first time an energy solution to the KPZ/SBE equation with characteristic operators given in terms of fractional operators such as the regional fractional Laplacian, whose precise definition is given in \eqref{deflapfracreg}. As this operator is non-local and it may a priori not be well-defined at the boundary, some Neumann/Dirichlet boundary conditions are required for our test functions, depending on the values of some parameters ruling the dynamics; this is illustrated in Figure \ref{fig:test_functions}. And similarly as in other previous works, when the asymmetry of our model is not strong enough we obtain solutions to  the Ornstein-Uhlenbeck equation, as expected. All of this is done by applying a scaling limit to a particular interacting particle system, which is described in the following subsection.

\subsection{The microscopic level}
We consider the model introduced in \cite{byronsdif,HLstefano} which consists of an exclusion process evolving on the one-dimensional discrete set of points $\Lambda_n:=\{1,\cdots, n-1\}$ that we call \textit{bulk}. The transition probability is given by \eqref{prob}, so it  depends on the size of the  jump but, and contrarily to the setting of \cite{byronsdif,HLstefano},  it may not be symmetric; its moments are regulated by a parameter $\gamma \in (0,2)$. Moreover at each site $x,y\in\mathbb Z$ satisfying  $x\leq 0$ or $y\geq n$ we add  a reservoir that can inject or remove particles in the system. As in \cite{byronsdif,HLstefano}, the symmetric part of the dynamics is tuned by parameters $\alpha>0$ and $\beta \in \mathbb{R}$ in the boundary, but not in the bulk. On the other hand, the antisymmetric part of the dynamics is tuned by \textit{both} by parameters $\alpha_a>0$ and $\beta_a \geq 0$. All of this is illustrated in Figure \ref{fig:dyn}. Depending on the range of the aforementioned parameters, we {conclude} that the fluctuations at equilibrium of our {models} are either given in terms of the Ornstein-Uhlenbeck process or energy solutions to the SBE. We do not present the analogous results for the KPZ equation but we observe that one could follow the same strategy as in \cite{asymjara} and restate our results for the KPZ instead of the SBE. In order to do so, it is only necessary to replace the density fluctuation by the height density field, as an object of study. We leave this to the interested reader. 

\subsection{Strategy of the proof, most technical aspects and final comments}

 Now  we note that the notion of energy solutions that we obtain is reminiscent of the notions of energy solutions to the SBE/KPZ equation referred to in Section \ref{intmac}, see Definition \ref{defspdefbe}. 
  There are basically four important ingredients  to derive: the Dynkin's martingales, the characterization of the initial condition as a spatial white-noise, the study of the process reversed in time and the most important one, the energy estimate which guarantees the existence of the non-linear term in the martingale formulation corresponding to the non-linear term of the SPDE. 
  
  {At the microscopic level, the} major difficulty that we faced in our proof is the need of obtaining a closed formula for the martingale in terms of the solution to the SBE, requiring a second-order Boltzmann-Gibbs Principle to replace non-linear terms appearing in the Markovian evolution by terms written in terms of the solution of the SPDE. This is obtained by a development of a sophisticated multi-scale analysis that consists in replacing iteratively local functions (typically occupation variables in certain domains) by their averages in microscopic boxes whose sizes increase at each step of the iteration procedure. {Due to the presence of the boundary dynamics, an additional technical issue is that we either had to choose boxes that are either right-directed or left-directed, depending on whether the reference site is closer to $1$ or to $n-1$. At the final step the average is performed at a box of a large microscopic size, which is sufficient to close the expression of the martingale in terms of macroscopic objects. This procedure is much based on previous proofs of the second order Boltzmann-Gibbs principle and it has its roots in Theorem 2.6 of \cite{G2008}. 
 
 Though, since now we deal with non-local operators and since we had to ensure that the density fluctuation field can be defined by an approximation procedure, see Remark \ref{remapr} for details. This requires that our test functions, after being acted on by fractional operators, are elements of $L^{2}([0,1])$, which is not immediate due to the non-locality of the regional fractional Laplacian. This motivates the presence of various space of test functions different boundary conditions, see Figure \ref{fig:test_functions}.
 
 Going back to the general strategy, we prove that our sequence of density fluctuation fields is tight, which then by Prohorov's theorem gives the weak convergence through subsequences. Since the density fluctuation fields are linear functionals that act on the respective space of test functions (all of them being nuclear Fr\'echet spaces), we may apply  the Mitoma's criterion for the proof of tightness. It is not absolutely mandatory to make use of this criterion but we do so in order to simplify our arguments.  We are then capable to obtain  limiting
martingale problems in all the regimes that we analysed. Nevertheless, for some of them we do not know if the uniqueness of the solutions to our martingale problems holds.  This issue is present even in the simpler case of the Ornstein-Uhlenbeck process because, up to our knowledge, we do not have enough information  about the spectrum of the regional fractional Laplacian operator. Regarding the more challenging martingale problems associated to the SBE, in some regimes, our notion of energy solutions match those recently studied in \cite{GPP}, where the authors already proved the uniqueness; this is crucial to ensure the convergence of the sequence of density fluctuation fields.} We have highlighted along the article the exact values of the parameters where the uniqueness {is lacking}. 
 
We observe that the recent article \cite{GPP} not only proves the existence and uniqueness to  energy solutions of the equation that we derive, namely \eqref{spdefbe}, but also for a much more general version of the equation, in which the noise is more irregular and presents bilinear nonlinearities. Their approach is not based on Fourier series expansion  and this allows treating the uniqueness problem for a much larger class of characteristic operators, as the regional fractional Laplacian,  a wide range of domains and boundary conditions. We believe that the results in \cite{GPP} can be applied in many more contexts, aside from the ones of our work.
 
At this point it is worth to make a few comments about possible extensions of our model. We note that here we considered a Markov process whose transition rate is quite specific, but it should be possible to extend our results to more microscopic dynamics, with corresponding macroscopic operators given in terms of other kernels. A particular feature of our dynamics is the exclusion rule, which allows at most one particle per site. This restriction is absolutely not mandatory, but has the advantage of avoiding additional technical issues. And given that in some passages our arguments are already quite demanding to follow in their present form, we opted for the exclusion process as our object of study. Taking everything into account, our work is still a significant contribution to the literature, due to the fact that this is the first article regarding the derivation of SPDE written in terms of the regional fractional Laplacian in bounded domains from particle systems.

Lastly we remark that the Ornstein-Uhlenbeck equation and the SBE here obtained (or their analogous counterparts) should also arise as as the scaling limit of equilibrium fluctuations in several other dynamics, including the case of multi-component models. In this direction, we refer the reader to the articles of \cite{Spo2,SS,PSS,PSSS} where it is conjectured that the fluctuations for such models are described by a number of limit SPDEs. However, very little has been done in a rigorous way to prove those conjectures. Quite possibly the results in \cite{GPP} are an important step in that direction.

\vspace{0.3cm}
\textbf{Outline of the article:} 

In Section \ref{secstat} we state our main theorems, i.e. Theorems \ref{clt} and \ref{clt2}. In order to do so, in Subsection \ref{subsec:SPDE} we present all the relevant definitions at the \textit{macroscopic} level. This includes the characteristic operators that are present in our SPDEs, and the definitions of the notions of solutions that are referred to in our main results. On the other hand, in Subsection \ref{subsec:micro_sys} we define the dynamics of our model (which is illustrated in Figure \ref{fig:dyn}) and all the relevant objects at a \textit{microscopic} level. After  this, Theorems \ref{clt} and \ref{clt2} are presented in Subsection \ref{sec:mainres}.

Section \ref{tightfluc} is then devoted to the proof of the tightness of the sequence of density fluctuation fields, which is crucial to ensure at least a weak convergence through subsequences. 

In Section \ref{sec:nonlinear} we present our main estimates to deal with the microscopic counterpart of the non-linear term appearing in the SPDE, which is defined exactly in \eqref{defAtnG}. 

The characterization of the limit point is done in Section \ref{seccharac}; in some regimes of the tuning parameters we are able to prove that such limit point is unique, and therefore we have convergence for the sequence of density fields. In the remaining regimes, we lack the uniqueness of the solutions to our corresponding martingale problem, therefore we were able to prove only convergence through subsequences.  

Finally in Appendices \ref{propfrac}, \ref{secdiscconv} and \ref{useful}, we presented a series of technical results which are completely independent of our dynamics, in order to make our text as self contained as possible.

\section{Statement of results} \label{secstat}

The goal of this work is to obtain the Ornstein-Uhlenbeck equation and the  stochastic Burgers equation, both given in terms of the regional fractional Laplacian defined in $(0,1)$. 
These equations will be obtained by means of a scaling limit of an interacting particle system when looking at the equilibrium fluctuations around its mean.

\subsection{Stochastic Partial Differential Equations}\label{subsec:SPDE}

\subsubsection{The space of test functions}
For any $q \geq 1$, we denote by $L^q$ the space of measurable functions $G: (0,1) \rightarrow \mathbb{R}$ such that
$
\| G \|^q_{L^q} := \int_0^1 |G(u)|^q \, du < \infty.
$
Moreover, if $G: (0,1) \rightarrow \mathbb{R}$ is bounded, we denote $\| G \|_{\infty}:=\sup_{u \in (0,1)} |G(u)|$.  We denote by $C_c^{\infty}(\mathbb{R})$ the space of functions $H: \mathbb{R} \rightarrow \mathbb{R}$ which are infinitely differentiable and have compact support. Moreover, we denote by $C^{\infty}([0,1])$ the space of the infinitely differentiable functions $G:(0,1) \rightarrow \mathbb{R}$ that can be extended to some $\hat{G} \in C_c^{\infty}(\mathbb{R})$. For any $u \in (0,1)$ and $k \geq 1$, we denote the $k$-th derivative of $G$ at $u$ by $G^{(k)}(u)$ and for $k=0$, we define $G^{(0)}(u):=G(u)$. We also denote $G^{(1)}$ (resp. $G^{(2)}$) by $\nabla G$ (resp. $\Delta G$).
Furthermore, for any $G \in C^{\infty}([0,1])$ and $k \geq 0$, the following limits are well-defined:
$
G^{(k)}(0):= \lim_{u \rightarrow 0^+} G^{(k)}(u); \quad G^{(k)}(1):= \lim_{u \rightarrow1^-} G^{(k)}(u).
$
With this convention, for every $d \geq 1$, by performing Taylor expansions on $G$ of order $d$ around $0$ and $1$ we get
\begin{align*}
\forall u \in (0,1), \quad    \sum_{j=0}^{d-1} \frac{u^j}{j!}  G^{(j)}(0) + \frac{u^d}{d!}  G^{(d)}(\xi_u) = G(u) = \sum_{j=0}^{d-1} \frac{(1-u)^j}{j!}  G^{(j)}(1) + \frac{(1-u)^d}{d!}  G^{(d)}(\tilde{\xi}_u).
\end{align*}
Above, $\xi_u\in (0,u)$ and $\tilde{\xi}_u\in(u,1)$. In particular, if there exists $d \geq 1$ such that $G^{(j)}(0) = 0 = G^{(j)}(1)$ for every $j \in \{0, 1, \ldots, d-1\}$, then  doing Taylor expansions on $G$ of order $d$ around $0$ and $1$  we obtain
\begin{equation} \label{expGSd}
\forall u \in (0,1), \quad |G(u)| \leq \frac{\| G^{(d)} \|_{\infty}}{d!} u^d \quad \text{and} \quad  |G(u)| \leq \frac{\| G^{(d)} \|_{\infty}}{d!} (1-u)^d.
\end{equation}

\begin{rem} \label{remnuc}
It is well known that $C^{\infty}([0,1])$ is a nuclear space. Moreover, from Chapter 3, Section 7) of \cite{nuclear}  every arbitrary subspace of $C^{\infty}([0,1])$ is also a nuclear space). 
\end{rem}
Following Theorem V.5 in \cite{ReedSimon}, the natural topology in $C^{\infty}([0,1])$ is given by
\begin{align*}
\forall G_1, G_2 \in C^{\infty}([0,1]), \quad  \| G_2 - G_1 \|_{C^{\infty}([0,1])}:= \sum_{k \in \mathbb{N}} 2^{-k} \frac{\| G_2^{(k)} - G_1^{(k)}  \|_{ \infty }}{1+\| G_2^{(k)} - G_1^{(k)}  \|_{ \infty}}.
\end{align*}
Above and in what follows, $\mathbb{N}:=\{0, 1, 2, \ldots\}$. With this topology, it is well known that $C^{\infty}([0,1])$ is a Fr\'echet space (this can be proved with the same arguments applied to obtain Theorem V.9 in \cite{ReedSimon}). 

\begin{rem} \label{remfre}
From Chapter 2 Section 4 of \cite{nuclear}, every closed subspace of $C^{\infty}([0,1])$ is also a Fr\'echet space. 
\end{rem}
Given $\mathcal{H} \subset C^{\infty}([0,1])$, a metric space $(N, \| \cdot \|_N)$ and an operator $\mathcal{P}:   \mathcal{H} \rightarrow N$, we say that $\mathcal{P}$ is \textit{continuous} if for any sequence ($H_j)_{j \geq 1} \subset \mathcal{H}$ and any $G \in \mathcal{H}$, it holds
\begin{align*}
 \forall k \geq 0, \quad \lim_{j \rightarrow \infty} \| H_j^{(k)} - G^{(k)}  \|_{ \infty } =0 \Rightarrow   \lim_{j \rightarrow \infty} \| H_j - G  \|_{ N } =0.
\end{align*} 
Let $ \mcb S_k:= \{G \in C^{\infty}([0,1]): G^{(k)}(0) = 0 = G^{(k)}(1) \}$, then  $\mcb S:=\cap _{k\geq 0} \mcb S_k$, $\mcb S_{Dir}=\mcb S_0$ and $\mcb S_{Neu}=\mcb S_1$.

 For every $(\beta, \gamma) \in R_0:=\mathbb{R} \times (0, \;  2)$, we denote our space of test functions by $\mcb {S}_{\beta,\gamma}$, where
\begin{equation} \label{defsbetagamma}
\mcb {S}_{\beta,\gamma}:=
\begin{cases}
\mcb S, \quad & \gamma \in (0,  2), \beta < 0; \\
C^{\infty}([0,1]), \quad & \gamma \in (0, \,  1) , \beta \geq 0 \quad \textrm{or}\quad  \gamma \in [1, \,  3/2) , \beta > \gamma - 1 ; \\
 \mcb S_{Dir}, \quad & \gamma \in [1, \, 3/2) , \beta = 0  \quad \textrm{or}\quad  \gamma \in (1, \,  3/2) , \beta \in (0, \gamma - 1]; \\
\mcb S_{Dir} \cap \mcb S_{Neu}, \quad & \gamma \in [3/2, \,  2),   \beta \in [0, \gamma - 1]; \\
 \mcb S_{Neu}, \quad & \gamma \in [3/2, \,  2) ,  \beta > \gamma - 1.
\end{cases}
\end{equation}

\begin{figure}[htb!]
\centering
	\begin{tikzpicture}[scale=0.25]
	\fill[color=red!35] (-25,-5) rectangle (5,-20);
	\fill[color=cyan!20] (-25,-5) -- (-10,-5) -- (-2,0) -- (-2,10) -- (-25,10) -- cycle;
	\fill[color=magenta] (-2,0.4) -- (5,5.4) -- (5,10)-- (-2,10) -- cycle;
	\fill[color=violet!30] (-10,-5) -- (-2,-5) -- (-2,0.6)-- cycle;
	\fill[color=magenta!20] (-2,-5)--(5,-5) -- (5,5.4)-- (-2,0.6) -- cycle;

	\draw[-,=latex,cyan!20,ultra thick] (-25,-5) -- (-10, -5) node[midway, below, sloped] {\footnotesize{{\textbf{\textcolor{cyan!20}{$C^\infty$ }}}}};
		\draw[-,=latex,violet!30,ultra thick] (-10,-5) -- (-2, -5) node[midway, below, sloped] {\footnotesize{{\textbf{\textcolor{violet!80}{$\mcb S_{Dir}$ }}}}};
	\draw[dotted, ultra thick, white] (5,10) -- (5,-20);
	\draw[dotted, ultra thick, white] (-25,10) -- (-25,-20);
	\draw[dotted, ultra thick, white] (-10,10) -- (-10,-20);
	
	\node[right, black] at (-18,5) {\textbf{\small{\textcolor{black}{$C^\infty$}}}} ;
	\node[right, black] at (-2,-3) {\textbf{\small{\textcolor{black}{$\mcb S_{Dir}\cap\mcb S_{Neu}$}}}} ;

\node[right, black] at (-0.5,6) {\textbf{\small{\textcolor{black}{$\mcb S_{Neu}$}}}} ;

\node[right, black] at (-7.2,-3) {\textbf{\small{\textcolor{black}{$\mcb S_{Dir}$}}}} ;

	\node[right, black] at (-12.4,-12) {\textbf{\small{\textcolor{black}{$\mcb S$}}}} ;
	\node[rotate=270, above] at (5,1) {\textcolor{black}{$\gamma = 2$}};
	\node[rotate=270, above] at (-10,1) {\textcolor{black}{$\gamma = 1$}};
	\node[rotate=270, above] at (-25,1) {\textcolor{black}{$\gamma = 0$}};
	\node[left] at (-25,-5) {\textcolor{black}{$\beta=0$}};
	\draw[white,fill=white] (5,5) circle (1.5ex);
	\draw[white,fill=white] (5,-5) circle (1.5ex);
	\draw[white,fill=white] (-10,-5) circle (1.5ex);
	
	\draw[<-,black] (-3.3, -0.1) -- (-5,3) node[above] {\textcolor{black}{$\theta=\gamma-1$}};

	\end{tikzpicture}
	\label{fig:test_functions} 	\caption{The spaces of test functions}
\end{figure}
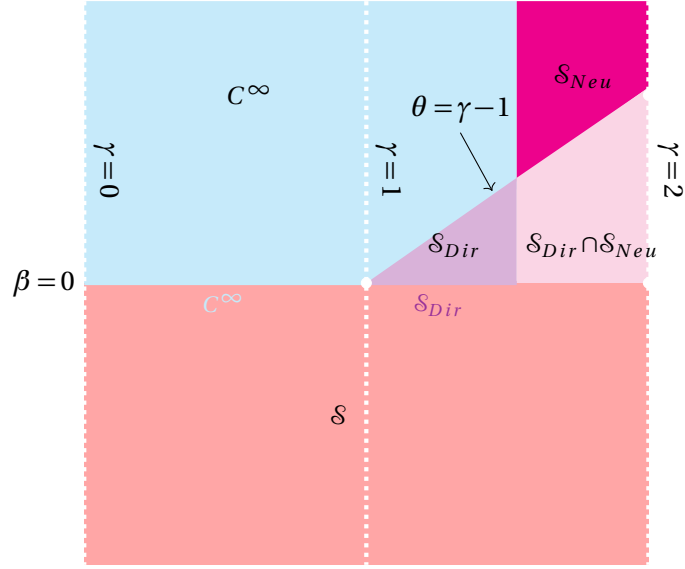

\begin{rem} \label{remnucfre}
For every $(\beta, \gamma) \in R_0$, $\mcb {S}_{\beta,\gamma}$ is a nuclear Fr\'echet space. The choices above are motivated by the goal of selecting the largest space of test functions (e.g., $ C^{\infty}([0,1])$) whenever it is possible. For $\gamma \in (0,  2)$ and $\beta <0$, the choice $\mcb {S}_{\beta,\gamma}= \mcb S$ is important to ensure that the density fluctuation defined below is well-defined, see Remark \ref{remsbetaneg} for more details. For $\beta=0$ and $\gamma \in [1, \, 2)$, the choice $\mcb {S}_{\beta,\gamma} \subset \mcb S_{Dir}$ is motivated by Proposition \ref{contlemA2stef}. For $\gamma \in (1,2)$ and $\beta \in (0, \; \gamma-1)$, the choice $\mcb {S}_{\beta,\gamma} \subset \mcb S_{Dir}$ is motivated by the third assumption in Proposition \ref{prop3assum}. Finally, for $\beta \geq 0$ and $\gamma \in [3/2, 2)$, the choice $\mcb {S}_{\beta,\gamma} \subset \mcb S_{Neu}$ is motivated by Proposition \ref{Lqreg}. 
\end{rem}

\subsubsection{Fractional operators and semi-inner products}

In this subsection we present $\mathbb{L}^{\gamma/2}$, which is equal (up to a multiplicative constant) to the regional fractional Laplacian  on $(0,1)$, see Definition 2.1 in \cite{reflected}. For every $G \in C^{\infty}([0,1])$, let
\begin{equation} \label{defdeltaepsilon}
	\forall u \in (0,1), \quad \big(\mathbb{L}_{\varepsilon}^{\gamma/2} G\big)(u):= \frac{c^{+}+c^{-}}{2}  \int_{0}^1 \mathbbm{1}_{\{ |u-v| \geq \varepsilon\} } \frac{G(v)-G(u)}{|v-u|^{1+\gamma}} \, dv.
\end{equation}
Next, for every $G \in C^{\infty}([0,1])$, we define $\mathbb{L}^{\gamma/2} G:(0,1) \rightarrow \mathbb{R}$ by 
\begin{align} \label{deflapfracreg}
\forall u \in (0,1), \quad  \big(\mathbb{L}^{\gamma/2} G\big)(u):= \lim_{\varepsilon \rightarrow 0^+} \big(\mathbb{L}_{\varepsilon}^{\gamma/2} G\big)(u).
\end{align} 
From Proposition 2.2 in \cite{reflected}, $\mathbb{L}^{\gamma/2} G$ is well-defined for any $G \in C^{\infty}([0,1])$, for every $\gamma \in (0,2)$. Next we define the semi inner-product $\langle\cdot,\cdot\rangle_{\gamma/2}$ by
\begin{equation*}
\forall f,g \in C^{\infty}([0,1]), \quad  \langle f, g \rangle_{\gamma/2} :=  \frac{c^{+}+c^{-}}{4} \iint_{(0,1)^2} \frac{[f(v) -f(u)] [g(v) -g(u)]}{|v-u|^{1+\gamma}} \; du \; dv.
\end{equation*}  
Finally, we define the corresponding semi-norm $\|\cdot \|_{\gamma/2}$ by
\begin{align} \label{seminorm}
\forall f \in C^{\infty}([0,1]), \quad  \| f^2 \|^2_{\gamma/2}:=  \langle f, f \rangle_{\gamma/2} =  \frac{c^{+}+c^{-}}{4} \iint_{(0,1)^2} \frac{[f(v) -f(u)]^2}{|v-u|^{1+\gamma}} \; du \; dv.
\end{align}

For  $(\alpha, \beta, \gamma) \in (0, \infty) \times \mathbb{R} \times (0,2)$, we define the operator $\mathbb{L}_{\alpha,\beta}^{\gamma/2}$ on $G $ by
\begin{equation} \label{deltaab}
\mathbb{L}_{\alpha,\beta}^{\gamma/2} G: = \mathbbm{1}_{ \{ \beta \geq 0 \} } \mathbb{L}^{\gamma/2} G + \mathbbm{1}_{ \{ \beta \leq 0 \} }  \alpha (r^{-} + r^{+}) G,  
\end{equation}
where $r^{-}:(0,1) \rightarrow \mathbb{R}$ and $r^{+}:(0,1) \rightarrow \mathbb{R}$ are given by
\begin{equation} \label{defrpm}
\forall u \in (0,1), \quad r^{-}(u):=  \frac{c^{+}+c^{-}}{2\gamma u^{\gamma}}  \quad \text{and} \quad r^{+}(u):= \frac{c^{+}+c^{-}}{2\gamma (1-u)^{\gamma}}.
\end{equation}
\begin{rem}\label{remsbetaneg}
Note that if $G\in \mcb S$ then  $\mathbb{L}_{\alpha,\beta}^{\gamma/2} G \in \mcb S$. Indeed, in this case 
 $\mathbb{L}_{\alpha,\beta}^{\gamma/2} G = \alpha (r^{-} + r^{+}) G$. Now, for every $k \in \mathbb{N}$, 
\begin{align*}
\forall u \in \mathbb{R}, \quad  (r^{-} G)^{(k)}(u) = \sum_{j=0}^k (r^{-})^{j}(u) G^{(k-j)}(u) \lesssim \sum_{j=0}^k u^{-\gamma-j} G^{(k-j)}(u).
\end{align*}
Therefore, in order to prove that $r^{-} G \in \mcb S$, it is enough to prove that given $k \in \mathbb{N}$ and $j \in \{0, 1, \ldots, k\}$, it holds  $\lim_{u \rightarrow 0^+} u^{-\gamma-j} G^{(k-j)}(u)=0$. In order to do so, from the fact that $G^{(i)}(0)=0$ for every $i \in \mathbb{N}$ (due to the assumption that $G \in \mcb S$), we can perform a Taylor expansion or order $j+2$ on $G^{(k-j)}$ around $0$. Thus, \eqref{expGSd} leads to $|G^{(k-j)}(u)| \leq \|G^{(k+2)} \|_{\infty} u^{j+2}$, for every $u \in (0,1)$, leading to $\lim_{u \rightarrow 0^+} u^{-\gamma-j} G^{(k-j)}(u)=0$.
Using the fact that $G^{(i)}(1)=0$ for every $i \in \mathbb{N}$ and \eqref{expGSd}, with an analogous argument we conclude that $r^{+} G \in \mcb S$, which ends the proof.
\end{rem}

However, when  $(\alpha,\beta,\gamma)\in (0,\infty)\times[0,\infty)\times (0,2)$ it is no longer true that $\mathbb{L}_{\alpha,\beta}^{ \gamma /2 } G\in \mcb S_{\alpha, \beta}$.  Nevertheless, in any space of test functions we can make an approximation in $L^2([0,1])$, see the corollary below, which is a direct consequence of Propositions \ref{Lqreg} and \ref{contlemA2stef}.
\begin{cor} \label{corL2ab}
	For every $G\in \mcb S_{\alpha,\beta}$  it holds $\mathbb{L}_{\alpha,\beta}^{ \gamma /2 } G \in  L^q( [0,1] )$, for $q \in \{1,2\}$. Moreover, the operator $\mathbb{L}_{\alpha,\beta}^{ \gamma /2 }: \mcb S_{\alpha,\beta} \rightarrow L^2(\mathbb{R})$ is continuous.
\end{cor}
Now we introduce a sequence of mollifiers which will regularize $\mathbb{L}_{\alpha,\beta}^{ \gamma /2 } G$. A classical example from \cite{brezis2010functional} is $(\phi_j)_{j \geq 1} \subset C_c^{\infty}(\mathbb{R})$ given by
\begin{equation} \label{defphij}
	\forall j \geq 1, \; \forall u \in \mathbb{R}, \quad 0 \leq \phi_j(u):= \mathbbm{1}_{ \{ |u|< j^{-1}  \} }  C_{\phi} \exp \big\{ (j^2u^2-1)^{-1} \big\} \leq C_{\phi},
\end{equation}
where $C_{\phi}>0$ is such that $\int_{\mathbb{R}} \phi_j(u) du = 1$, for any $j \geq 1$. Next, we fix a sequence $(\psi_j)_{j \geq 1} \subset C_c^{\infty}(\mathbb{R})$ such that
\begin{equation} \label{defpsij}
	\forall j \geq 1, \; \forall u \in \mathbb{R}, \quad 0 \leq \psi_j (u) \leq 1 \; \; \text{and} \; \; \psi_j(u):=
	\begin{cases} 
		0, \quad & u \in (-\infty,j^{-1}) \cup (1 - j^{-1}, \infty); \\
		1, \quad & 2j^{-1} \leq u \leq 1 - 2j^{-1}.
	\end{cases}
\end{equation}
Moreover, for every $ \tilde{G}: (0,1) \rightarrow \mathbb{R}$, we define the (generally not smooth) extension $A_{\tilde{G}}: \mathbb{R} \rightarrow \mathbb{R}$ by
\begin{equation} \label{extL1}
	\forall u \in \mathbb{R}, \quad A_{\tilde{G}}(u):=\tilde{G}(u)\textbf{1}_{(0,1)]}(u)
\end{equation}
Furthermore, for every $\widetilde{G} \in L^2 ( [0,1] )$ and for every $j \geq 1$, we define $H^{\tilde{G}}_j: (0,1) \rightarrow \mathbb{R}$ by
\begin{equation} \label{restraprox}
H^{\widetilde{G}}_j(u):= [\psi_{j} \cdot (\phi_{j}* A_{\widetilde{G}} ) ](u) = \psi_{j}(u) \int_{\mathbb{R}} A_{\widetilde{G}} (v) \phi_j(u-v) dv, \quad u \in (0,1). 
\end{equation}

Now we present a result whose proof is postponed to Appendix \ref{useful}.
\begin{prop} \label{aproxL2gen}
	Let $\widetilde{G} \in L^2 ( [0,1] )$, and for every $j \geq 1$, let $H^{\widetilde{G}}_j$ be given by \eqref{restraprox}. Then for every $j \geq 1$, it holds $H_j^{\widetilde{G}} \in \mcb S $. Moreover, $H^{\widetilde{G}}_j$ converges to $\widetilde{G}$ in $L^2([0,1])$ as $j \rightarrow \infty$. 
\end{prop}
For $(\alpha, \beta, \gamma) \in (0, \infty) \times \mathbb{R} \times (0, 2)$ and  $G\in \mcb S_{\alpha,\beta}$, we define  the operator  $ \mathcal{P}^{\gamma}_{\alpha, \beta}$ on $G$ by
\begin{align} \label{defPalbetgam}
	\mathcal{P}^{\gamma}_{\alpha, \beta} G: = \frac{c^{+}+c^{-}}{2} \mathbbm{1}_{  \{ \beta \geq 0 \} } \| G \|_{\gamma/2}^2 + \alpha \mathbbm{1}_{  \{ \beta \leq 0 \} } \int_0^1 [r^{-}(u) + r^{+}(u)]G^2 (u) \, du.
\end{align}
Now we present another result whose proof is postponed to Appendix \ref{useful}.

\begin{prop} \label{norcont}
	For every $(\alpha, \beta, \gamma) \in (0, \infty) \times \mathbb{R} \times (0, 2)$,  the operator $\mathcal{P}^{\gamma}_{\alpha, \beta}: \mcb S_{\alpha,\beta} \rightarrow \mathbb{R}$
is continuous.
\end{prop}

\subsubsection{Martingale problems for SPDEs}

In order to state our main results, we present a series of definitions in what follows.
\begin{definition} \label{defstat}

We say that a $\mcb S_{\alpha,\beta}'$-valued stochastic process $(\mathcal{Z}_t)_{ 0 \leq t \leq T}$ defined on some probability space $(X, \mathcal{F}, P)$ is \textit{stationary} if for any $t \in [0,T]$, the $\mcb S_{\alpha,\beta}'$-valued random variable $\mathcal{Z}_t$ is a white-noise of variance $\chi(b)$. In particular, there exists a constant $K_0>0$ such that 
\begin{align} \label{condUSC}
\forall G \in \mcb S_{\alpha,\beta}, \; \forall t \in [0, T], \quad \ \mathbb{E}_P \big[ \big( \mathcal{Z}_t(G) \big)^2 \big] \leq K_0 \| G \|^2_{L^2}.
\end{align}
\end{definition}

\begin{rem} \label{remapr}
Let $(\mathcal{Z}_t)_{ 0 \leq t \leq T}$ be a $\mcb S_{\alpha,\beta}'$-valued stochastic process. Then for every $G \in \mcb S_{\alpha,\beta}$ and every $(\alpha, \beta, \gamma) \in (0, \infty) \times \mathbb{R} \times (0, 2)$,  the process $\big(\mathcal{I}_t(G) \big)_{0 \leq t \leq T}$ given by
\begin{equation} \label{intprocess}
\mathcal{I}_t(G):= \lim_{j \rightarrow \infty}   \int_0^{t} \mathcal{Z}_s( H^{\mathbb{L}_{\alpha,\beta}^{ \gamma /2 } G}_j )ds
\end{equation}
is well defined, due to Corollary \ref{corL2ab} and Proposition \ref{aproxL2gen}. Above, $H^{\mathbb{L}_{\alpha,\beta}^{ \gamma /2 } G}_j$ is defined by replacing $\widetilde{G}$ in \eqref{restraprox} by $\mathbb{L}_{\alpha,\beta}^{ \gamma /2 } G$.
\end{rem}
Next, we present below the martingale problem for the Ornstein-Uhlenbeck equation that we will derive. 
\begin{definition} \label{defspde}
Let $(\alpha, \beta, \gamma) \in (0, \infty) \times \mathbb{R} \times (0,  2)$. We say that a process $(\mathcal{Z}_t)_{ 0 \leq t \leq T}$ is a stationary solution of {the Ornstein-Uhlenbeck equation}
\begin{equation} \label{spde}
d \mathcal{Z}_t = \mathbb{L}_{\alpha,\beta}^{ \gamma /2 } \mathcal{Z}_t dt + \sqrt{2 \chi(b) \mathcal{P}^{\gamma}_{\alpha, \beta} } d \mathcal{W}_t
\end{equation}
if the process $(\mathcal{Z}_t)_{ 0 \leq t \leq T}$ is stationary and
\begin{enumerate}
\item
 for any $G \in \mcb S_{\alpha,\beta}$, the processes $\mathcal{M}_t(G)$ and $\mathcal{N}_t(G)$ given by
\begin{equation}  \label{MartNou}
\mathcal{M}_t(G) := \mathcal{Z}_t(G) - \mathcal{Z}_0(G) - \mathcal{I}_t(G), \quad \mathcal{N}_t(G) :=[ \mathcal{M}_t(G)]^2  -  2 \chi(b) t    \mathcal{P}^{\gamma}_{\alpha, \beta} G,
\end{equation}
are $\mathcal{F}_t$-martingales, where for each $t \in [0,T]$, $\mathcal{F}_{t}:= \sigma ( \mathcal{Z}_s (G): (s,G) \in [0,t] \times \mcb {S}_{\beta,\gamma})$.
\item
$\mathcal{Y}_0$ is a mean zero Gaussian field with covariance given on $G_1,G_2 \in \mcb {S}_{\beta,\gamma}$ by
\begin{equation} \label{covinifie}
\mathbb{E}_{P} [ \mathcal{Z}_0(G_1) \mathcal{Z}_0(G_2) ] = \chi(b) \int_{0}^1  G_1(u) G_2(u) \, du. 
\end{equation}
\end{enumerate}
\end{definition}
The random element $\mathcal{Z}$ is a generalized Ornstein-Uhlenbeck process. From \eqref{MartNou} and Levy's Theorem on the martingale characterization of Brownian motion, for every $G \in \mcb {S}_{\beta,\gamma}$ fixed, the process
$
\mathcal M_t (G) \big[ \sqrt{2 \chi(b)}  \mathcal{P}^{\gamma}_{\alpha, \beta} G \big]^{-1}
$
is a standard Brownian motion. 

Next, we state a result regarding the uniqueness in law for stationary solutions of \eqref{spde}.
\begin{prop} \label{uniqlaw}
Let $\alpha >0$, $\gamma \in (0, 2)$ and $\beta \in (-\infty,0) \cup (0, \infty)$. Then two stationary solutions of \eqref{spde} have the same law.
\end{prop}
\begin{proof}
For $\alpha >0$, $\gamma \in (0,2)$ and $\beta <0$, the result follows by applying exactly the same arguments described in Section 5.2 of \cite{flucstefano}. On the other hand, for $\alpha >0$, $\gamma \in (0,2)$ and $\beta >0$, the result follows by applying exactly the same arguments described in Appendix C of \cite{jarafluc}.

On the other hand, for $\alpha >0$, $\gamma \in (0,2)$ and $\beta >0$, we get from \eqref{deltaab} that $\mathbb{L}_{\alpha,\beta}^{\gamma/2}=\mathbb{L}^{\gamma/2}$. Thus, combining the definition of $\mcb S_{\beta,\gamma}$ in \eqref{defsbetagamma} with Theorems 6.3 and 6.4 (for $p=2$) of \cite{reflected}, we have that $\mathbb{L}_{\alpha,\beta}^{\gamma/2}=\mathbb{L}^{\gamma/2}$ can be identified with the generator $\big(L, \mathcal{D}(L) \big)$ of a strongly continuous contraction semigroup $(P_t)_{t \geq 0}$ on $L^2([0,1])$. More exactly, $\mcb S_{\beta,\gamma} \subset \mathcal{D}(L)$ and $L G = \mathbb{L}_{\alpha,\beta}^{\gamma/2} G$ a.e. on $[0, 1]$, for every $G \in \mcb S_{\beta,\gamma}$. Moreover, we observe that $P_s G \in L^2([0,1])$ for any $s \geq 0$ and any $G \in \mcb S_{\beta,\gamma}$. 
Note that thanks to Proposition \ref{aproxL2gen}, for every $(\alpha, \beta, \gamma) \in (0, \infty) \times [0, \infty) \times (0,2)$, we can define
$\forall t \in [0, T]$ and $\forall \tilde{G} \in L^2([0, 1])$ $ \mathcal{Y}_t (\tilde{G}) := \lim_{j \rightarrow \infty} \mathcal{Y}_t ( H^{\tilde{G}}_j). 
$

Thus,   $\mathcal{Y}_t^n ( P_s G )$ is well-defined for every $s,t \in [0, T]$ and every $G \in \mcb S_{\beta,\gamma}$. The result follows by applying exactly the same arguments described in Appendix C of \cite{jarafluc}.
\end{proof}
\begin{rem}
In order to prove the uniqueness when $\alpha >0$, $\gamma \in (0,2)$ and $\beta=0$, we could follow the arguments described in  Section 5.3 of \cite{flucstefano}, nevertheless we do not have enough information about the spectrum of the  regional fractional Laplacian defined in $(0, 1)$.  The uniqueness in the other regimes is left as an open problem. 
\end{rem}

\begin{definition}\label{defEE}
Let $( \beta, \gamma) \in R_0$ and $(\mathcal{Z}_t)_{0 \leq t \leq T}$ be a given $\mcb {S}_{\beta,\gamma}'$-valued stochastic process $(\mathcal{Z}_t)_{ 0 \leq t \leq T}$ defined on some probability space $(X, \mathcal{F}, P)$. Define for $\varepsilon \in (0, 1/2)$, $s < t \in [0,T]$ and $G \in \mcb {S}_{\beta,\gamma}$,
\begin{equation} \label{defprocA}
\mathcal{A}_{s,t}^{\varepsilon}(G):= \int_s^t \Big[ \int_0^{1/2} [ \mathcal{Z}_r*\iota_{\varepsilon}^{+} (u) ]^2 \nabla G(u) du + \int_{1/2}^{1} [ \mathcal{Z}_r*\iota_{\varepsilon}^{-} (u) ]^2 \nabla G(u) du \Big] dr. 
\end{equation}
Above, $\iota_{\varepsilon}^{+}$ and $\iota_{\varepsilon}^{-}$ are given by
\begin{align} \label{defiota}
\iota_{\varepsilon}^{+}(u):= \frac{1}{\varepsilon} \mathbbm{1}_{ \{ (0, \varepsilon ] \} }(u); \quad \quad   \iota_{\varepsilon}^{-}(u):= \frac{1}{\varepsilon} \mathbbm{1}_{ \{ [- \varepsilon, 0) \} }(u).
\end{align} 
We say that $(\mathcal{Z}_t)_{0 \leq t \leq T}$ satisfies an \textit{Energy Estimate} (EE) if there exist $\kappa_0 >0$, $\omega \in (0,1)$ such that
\begin{equation} \label{eqdefEE}
\mathbb{E}_P \big[ \big( \mathcal{A}_{s,t}^{\varepsilon}(G) - \mathcal{A}_{s,t}^{\delta}(G)  \big)^2 \big] \leq \kappa_0 \varepsilon^{\omega} (t-s) \| \nabla G \|^2_{L^2} 
\end{equation}
for any $G \in \mcb {S}_{\beta,\gamma}$, any $0 < \delta < \varepsilon < 1$ and any $0 \leq s < t \leq T$. 
\end{definition}
We observe that the energy estimate tells us that the sequence  $(\mathcal{A}_{s,t}^{\varepsilon}(G) )_\varepsilon$
 is a Cauchy sequence in $L^2(P)$ and as a consequence of the completeness of $L^2(P)$ this implies  the existence of the limit
$
\mathcal{A}_{s,t}(G):= \lim_{\varepsilon \rightarrow 0^+} \mathcal{A}_{s,t}^{\varepsilon}(G)
$
in $L^2(P)$ for any $0 \leq s < t \leq T$ and any $G \in \mcb {S}_{\beta,\gamma}$. This process can be understood as an integrated version of $-( \nabla \mathcal{Z}_t)^2$:
\begin{align*}
\mathcal{A}_{s,t}(G)= - \int_s^t ( \nabla \mathcal{Z}_t)^2(G) ds.
\end{align*}
Moreover, if the process $(\mathcal{Z}_t)_{0 \leq t \leq T}$ is stationary and satisfies \eqref{eqdefEE}, the process $(\mathcal{A}_t)_{0 \leq t \leq T}$ given by
$
\mathcal{A}_t(G):= \lim_{\varepsilon \rightarrow 0^+} \mathcal{A}_{0,t}^{\varepsilon}(G)
$
is also well defined, see Proposition 2.10 in \cite{jarafluc} and Theorem 2.2 in \cite{asymjara} for more details. 

\begin{definition} \label{defspdefbe}
Let $(\alpha, \beta, \gamma) \in [0, \infty) \times(0,\infty)\times [\frac 32,  2)$ and $\kappa_1 \in \mathbb{R}$. We say that a process $(\mathcal{Z}_t)_{ 0 \leq t \leq T}$ is a stationary {energy} solution of the stochastic regional fractional Burgers equation
\begin{equation} \label{spdefbe}
d \mathcal{Z}_t = \mathbb{L}_{\alpha,\beta}^{ \gamma /2 } \mathcal{Z}_t dt +\kappa_1 (\nabla \mathcal{Z}_t)^2 +  \sqrt{2 \chi(b) \mathcal{P}^{\gamma}_{\alpha, \beta} } d \mathcal{W}_t
\end{equation}
if the process $(\mathcal{Z}_t)_{ 0 \leq t \leq T}$ is stationary, satisfies an energy estimate and
\begin{enumerate}
\item
 for any $G \in \mcb {S}_{\beta,\gamma}$, the processes $\mathcal{M}_t(G)$ and $\mathcal{N}_t(G)$ given by
\begin{align*}
& \mathcal{M}_t(G) := \mathcal{Z}_t(G) - \mathcal{Z}_0(G) - \mathcal{I}_t(G) + \kappa_1 \mathcal{A}_t(G),  \quad \mathcal{N}_t(G) :=[ \mathcal{M}_t(G)]^2  -  2 \chi(b) t    \mathcal{P}^{\gamma}_{\alpha, \beta} G,
\end{align*} 
 are {continuous} $\mathcal{F}_t$-martingales, where for each $t \in [0,T]$, $\mathcal{F}_{t}:= \sigma ( \mathcal{Z}_s (G): (s,G) \in [0,t] \times \mcb {S}_{\beta,\gamma})$.
\item
$\mathcal{Z}_0$ is a mean zero Gaussian field with covariance given on $G_1,G_2 \in \mcb {S}_{\beta,\gamma}$ by \eqref{covinifie}.
\item the reversed process $\{\widehat{Z}_{t}:= Z_{T-t}: t\in[0,T]\}$ satisfies item (1) with $\kappa_1$ replaced by $-\kappa_1$ and 
$\widehat{\mathcal{A}}_t:= \mathcal{A}_T-\mathcal{A}_{T-t}$. 
\end{enumerate}
\end{definition}

\begin{rem}\label{rem:uniq}
We observe that recently in Theorem 3.7 of \cite{GPP} the uniqueness of energy solutions in the sense as given above has been derived, if $(\alpha, \beta, \gamma) \in (0, \infty) \times (0,+\infty) \times [\frac 32,  2)$. We refer the interested reader to Lemma 2.19 of \cite{GPP} as well as Example 4.4 in that article, where the notion of energy solution in our specific case is analysed. We note that uniqueness of energy solutions for the regime $(\alpha,\beta,\gamma)\in(0,+\infty)\times \{0\}\times [\frac 32, 2)$ is left as an open problem, because there is an additional term in our SPDE that comes from the boundary dynamics and appear as a reaction term see \eqref{deltaab}. 
\end{rem}

\subsection{The microscopic models}
\label{subsec:micro_sys}

For any $n \geq 2$, we analyse the evolution of the exclusion process in $\Lambda_n:=\{1, \ldots, n-1\}$. This is an IPS which allows at most one particle per site, therefore our space state is $\Omega_n:=\{0,1\}^{\Lambda_n}$. The elements of $\Lambda_n$ are called \textit{sites} and will be denoted by Latin letters, such as $x$, $y$ and $z$. The elements of $\Omega_n$ are called \textit{configurations} and will be denoted by Greek letters, such as $\eta$. Moreover, we denote the number of particles at a site $x$ according to a configuration $\eta$ by $\eta(x)$; this means that the site $x$ is \textit{empty} if $\eta(x)=0$ and it is \textit{occupied} if $\eta(x)=1$. A particle jumps from $x$ to $y$ with rate $p(y-x)$, where $p: \mathbb{Z} \rightarrow [0,1]$ is a transition probability given,  for every $ x \in \mathbb{Z}$ by
\begin{equation} \label{prob}
 p( x ) := 
\frac{c^{+}}{|x|^{\gamma+1}}\textbf{1}_{x >0}+\frac{c^{-}}{|x|^{\gamma+1}}\textbf{1}_{x <0}
\end{equation}
where $\gamma \in (0, 2)$ and $c^-, c^+$ are two non-negative constants such that $c^-+c^+>0$. Since $\gamma >0$, we get $\sum_{x} p(x) = (c^-+c^+) \sum_{x=1}^{\infty} x^{-\gamma-1} < \infty$. Nevertheless, we do not impose last sum to be equal to one, therefore $p(\cdot)$ \textit{may not be} a probability measure. We also define the symmetric and antisymmetric parts of $p$ by
\begin{align} \label{defsa}
\forall x \in \mathbb{Z}, \quad s(x):=\frac{p(x)+p(-x)}{2}=s(-x), \quad \quad a(x)=:\frac{p(x)-p(-x)}{2}=-a(-x).
\end{align}
If $\gamma \in (1,2)$, we observe that $|m_a| < \infty$, where
\begin{align} \label{defma}
m_a:= \sum_{x=1}^{\infty} x a(x) = \sum_{x=1}^{\infty} (-x) a(-x) = \sum_{x=-\infty}^{-1} x a(x). 
\end{align}
Next, we fix the parameters $b \in (0,1)$, $\alpha >0$ and $\beta \in \mathbb{R}$, which only concern the symmetric part of the dynamics and we also fix two parameters $\alpha_a >0$ and $\beta_a \in \mathbb{R}$, that only concern the antisymmetric part of the dynamics. In this way, we can define 
the exclusion process with infinitesimal generator given by $\mcb L_n:=\mcb L_n^{0}+ \mcb L_n^{l} + \mcb L_n^{r}$, where
\begin{equation}  \label{generatorsasym}
\begin{split}
(\mcb L^{0}_n f)(\eta) :=& \sum_{x,y } [ s(y-x) + \alpha_a n^{-\beta_a} a(y-x) ] \eta(x)[1-\eta(y)] [ f(\eta^{x,y}) -f(\eta)],\\
(\mcb L_n^{l} f)(\eta) :=& \sum_x \sum_{y = - \infty}^0 [\alpha n^{-\beta} s(y-x) + \alpha_a n^{-\beta_a} a(y-x) ] \eta(x)[1-b] [f(\eta^x) - f(\eta)] \\
+& \sum_x \sum_{y = - \infty}^0 [\alpha n^{-\beta} s(x-y) + \alpha_a n^{-\beta_a} a(x-y) ] b [1 - \eta(x) ]  [f(\eta^x) - f(\eta)], \\
(\mcb L_n^{r} f)(\eta) :=& \sum_x \sum_{y = n}^{\infty} [\alpha n^{-\beta} s(y-x) + \alpha_a n^{-\beta_a} a(y-x) ] \eta(x)[1-b] [f(\eta^x) - f(\eta)] \\
+& \sum_x \sum_{y = n}^{\infty} [\alpha n^{-\beta} s(x-y) + \alpha_a n^{-\beta_a} a(x-y) ] b [1 - \eta(x) ]  [f(\eta^x) - f(\eta)],
\end{split}
\end{equation}
where  the exchange transformation is 
$(\eta^{x,y})(z) = \eta(z)\textbf{1}( z \ne x,y)+ 
		\eta(y)\textbf{1}(z=x)+ 
		\eta(x)\textbf{1}(z=y)$ and the Glauber transformation is given by $(\eta^x)(z)= 
 \eta(z)\textbf{1}( z \ne x)+
		(1-\eta(x))\textbf{1}(z=x).$
The dynamics can be described as follows: 
\begin{enumerate}
\item 
given  $x,y \in \Lambda_n$, a jump  from $x$ to $y$ occurs at rate $[ s(y-x) + \alpha_a n^{-\beta_a} a(y-x) ] \eta(x) [1 - \eta(y) ]$;
\item 
a reservoir at $y \leq 0$ creates (resp. destroys) a particle at  $x \in \Lambda_n$ with rate $[\alpha n^{-\beta} s(x-y) + \alpha_a n^{-\beta_a} a(x-y) ] b [1 - \eta(x) ]$ (resp.  $[\alpha n^{-\beta} s(y-x) + \alpha_a n^{-\beta_a} a(y-x) ] \eta(x)[1-b]$);
\item
a reservoir at $y \geq n$ creates (resp. destroys) a particle at $x$ with rate $[\alpha n^{-\beta} s(x-y) + \alpha_a n^{-\beta_a} a(x-y) ] b [1 - \eta(x) ]$ (resp.   $[\alpha n^{-\beta} s(y-x) + \alpha_a n^{-\beta_a} a(y-x) ] \eta(x)[1-b]$).
\end{enumerate}
\begin{figure}
\begin{tikzpicture}[scale=0.88][htb!]
\fill[cyan!20] (-1, -1) rectangle (3, 3);
\fill[cyan!20] (13, -1) rectangle (17, 3);
\fill[cyan!20, opacity=0.2] (3, -1) rectangle (13, 3);
\foreach \i in {0, 1, 2} {
    \foreach \j in {0, 0.4, 0.8, 1.2, 1.6} {
        \node[circle, draw, fill=violet!30] at (\i,\j) {};
    }
}
\foreach \i in {14, 15, 16} {
    \foreach \j in {0, 0.4, 0.8, 1.2} {
        \node[circle, draw, fill=violet!30] at (\i,\j) {};
    }
}
\foreach \i in {4, 5, 7, 9, 10, 12} {
    \foreach \j in {0} {
        \node[circle, draw, fill=violet!30] at (\i,\j) {};
    }
}
\foreach \i in { 3, 6, 8, 11,13} {
    \foreach \j in {0} {
        \node[circle, draw, fill=white] at (\i,\j) {};
    }
}

\draw[thick] (-1, 0) -- (17, 0);

\foreach \x in { 0, 1, 2, 4, 5, 7, 9, 10, 12,  14, 15, 16} {
    \draw[thick] (\x, 0) -- (\x, -0.3);
}

\node[below] at (8, -0.2) {$x_5$};
\node[below] at (0, -0.2) {$x_1$};
\node[below] at (4, -0.2) {$x_2$};
\node[below] at (7, -0.2) {$x_4$};
\node[below] at (13, -0.2) {\small{$N-1$}};
\node[below] at (3, -0.2) {\small{$1$}};
\node[below] at (5, -0.2) {$x_3$};
\node[below] at (16, -0.2) {$x_7$};
\node[below] at (9, -0.2) {$x_6$};
\draw[thick, ->, >=Stealth] (0,2) to [out=60, in=120] (4, 0.5);
\draw[thick, ->, >=Stealth] (5,0.5) to [out=60, in=120] (8, 0.5);
\draw[thick, ->, >=Stealth] (9,0.5) to [out=60, in=120] (16, 1.6);

\node at (4, 2) {\small{$[\alpha n^{-\beta} s(x_2-x_1) + \alpha_a n^{-\beta_a} a(x_2-x_1) ] b $}};
\node at (7, 1) {\small $[ s(x_5-x_3) + \alpha_a n^{-\beta_a} a(x_5-x_3) ] $};
\node at (13, 1.6) {\small $[\alpha n^{-\beta} s(x_6-x_7) + \alpha_a n^{-\beta_a} a(x_6-x_7) ] [1-b]$};
\end{tikzpicture}
	\label{fig:dyn}
	\caption{The dynamics }
\end{figure}
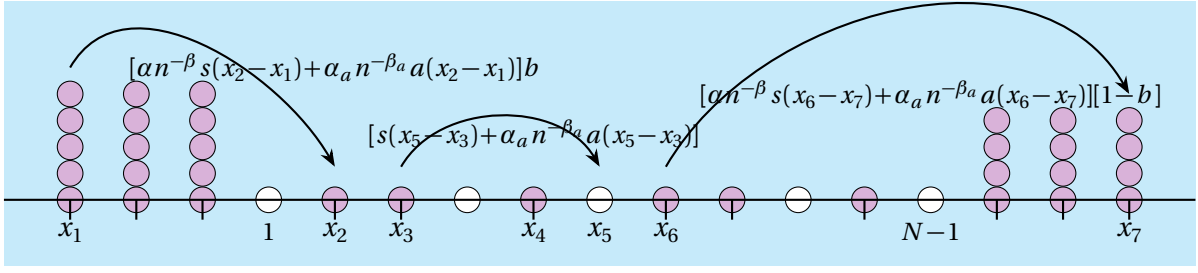

Above and in what follows, unless it is stated differently, we will assume that our discrete variables in a summation always range over $\Lambda_n$. Observe that $(\eta^{x,y})^{x,y}=\eta$ for any $x,y \in \Lambda_n$ and any $\eta \in \Omega_n$. Moreover, $(\eta^{x})^{x}=\eta$ for any $x \in \Lambda_n$ and any $\eta \in \Omega_n$. 
Since the rates must be non-negative, we will require $\alpha_a n^{-\beta_a} |a(y-x)| \leq \min\{1, \alpha n^{-\beta} \} s(y-x)$, for every $x,y \in \mathbb{Z}$ and every $n \geq 2$. Two sufficient conditions for this can be obtained by making the choice: 
\begin{equation} \label{condrate}
\beta_a \geq \max \{0, \; \beta\}; \quad \alpha_a \leq \mathbbm{1}_{ \{ c^{+} \neq c^{-} \}} \frac{ c^{+}+c^{-} }{ | c^{+} - c^{-} |} \min \{1, \; \alpha \}. 
\end{equation}

\subsubsection{The time scales}

In order to observe a non-trivial macroscopic limit, we will consider the Markov processes  accelerated in time by a factor $\Theta(n)= \Theta(n, \beta)$, given by
\begin{equation} \label{timescale}
\Theta(n, \beta):=
n^{\gamma + \beta}\textbf{1}_{\beta <0}+
n^{\gamma}\textbf{1}_{\beta \geq 0}.
\end{equation}
In particular, it holds
\begin{equation} \label{tiscbound}
\forall n \geq 2, \; \forall \beta \in \mathbb{R}, \; \forall \gamma \in (0,2), \quad \Theta(n) n^{-\beta} \leq n^{\gamma}.
\end{equation}
The role of the parameter $\beta $ is to tune the intensity of the reservoirs dynamics and to make it slower (resp. faster) when $\beta >0$ (resp. $\beta <0$ with respect to the bulk dynamics). To have a non trivial contribution we are also forced to speed up the process in the time scale $t\Theta(N)$. As a consequence, macroscopically, when the reservoirs are \textit{slow} we only see  a  flow of mass through the bulk, but when they are fast the contribution comes  purely from the action of the reservoirs' dynamics. When  $\beta=0$ both dynamics give a contribution at the macroscopic level.
Our goal is to analyse the evolution of the density for the process $\eta_t^n(\cdot):=\eta_{t \Theta(n)}(\cdot)$, whose infinitesimal generator is given by $\Theta(n) \mcb {L}_n$.

\subsubsection{Invariant measures}

Fix a parameter $b \in (0,1)$. We  denote by $\nu_b^n$ the Bernoulli product measure with marginals given by 
\begin{align} \label{defberprod}
\forall x \in  \Lambda_n, \quad  1 - \nu_b^n \big( \eta: \eta(x) = 0 \big) = \nu_b^n \big(\eta: \eta(x) = 1 \big) = b.
\end{align}
which is an invariant measure for the dynamics introduced above.

Fix $T>0$. Given a metric space $(N, \; \| \cdot \|_N)$ we denote by $ \mcb{D} ( [0,T], N)$ the space of c\`adl\`ag functions endowed with the Skorokod topology. Moreover, we denote by $\mathbb{P} _{\nu_b^n}$ be the probability measure on $\mcb D ([0,T],\Omega_n)$ induced by the Markov process $\{\eta_{t}^{n};{t\geq 0}\}$ and the initial distribution $\nu_b^n$. Furthermore, we denote the expectation with respect to $\mathbb{P}_{\nu_b^n}$ by $\mathbb{E}_{\nu_b^n}$.

\subsubsection{Density fluctuation field}
We denote by $\mcb {S}_{\beta,\gamma}'$  the space of real valued bounded linear functionals defined in $\mcb {S}_{\beta,\gamma}$. Moreover, $ \mcb{C} ( [0,T],    \mcb {S}_{\beta,\gamma}')$ is the space of continuous $\mcb {S}_{\beta,\gamma}'$- valued functions. Our object of study is the \textit{density fluctuation field} $\mathcal{Y}_t^n$, defined as the $\mcb {S}_{\beta,\gamma}'$-valued process $\{ \mathcal{Y}_t^n: t \in [0,T] \}$ given on $G \in \mcb {S}_{\beta, \gamma}$ by
\begin{equation}  \label{eqdensfie}
\mathcal{Y}_t^n (G) :=  \frac{1}{\sqrt{n-1}} \sum_{x} G ( \tfrac{x}{n} ) \bar{\eta}_t^n(x),
\end{equation}
for any $t \in [0,T]$ and any $n \geq 2$. Above $\bar{\eta}_t^n(x):= \eta_t^n(x)-b$. Note that since $\nu_b^n$ is Bernoulli product we have
\begin{equation} \label{densfield}
\begin{split} 
\forall G \in \mcb {S}_{\beta, \gamma} \subsetneq L^2 ( [0,1] ), \quad &\lim_{n \rightarrow \infty} \mathbb{E}_{\nu_b^n} \big[ \big(\mathcal{Y}_t^n (G)  \big)^2 \big] = \chi(b) \| G \|^2_{L^2} < \infty. 
\end{split}
\end{equation}
     
In order to state the main results of this article, we will make use of two hypotheses, which are motivated by Propositions \ref{convAtnG} and \ref{tightAtnG} below. We begin by defining
\begin{equation} \label{defrs}
r_a:=
(\gamma + \beta-\beta_a)\textbf{1}_{\beta < 0}+(
\gamma - \beta_a)\textbf{1}_{\beta \geq 0}.
\end{equation}
Recall from \eqref{condrate} that $\beta_a \geq \max \{0, \beta\}$. Thus, from \eqref{defrs} we get
\begin{equation} \label{comprsgam}
r_a \leq \gamma. 
\end{equation}
\begin{hyp}
We say that (H1) is satisfied if one of the next  assumptions holds:
\begin{itemize}
\item [(I)]
$c^{+} = c^{-}$ or $r_a < 1$;
\item [(II)]
$r_a \geq 1 = \gamma$;
\item [(III)]
 $\beta \geq 0$ and $2 r_a < 3$;
\item [(IV)]
 $\beta < 0$  and $2 r_a \leq 3 + \beta$.
\end{itemize}
\end{hyp}

\begin{hyp}
We say that (H2) is satisfied if $\beta \geq 0$, $r_a =3/2$ and $c^{+} \neq c^{-}$.
\end{hyp}

\begin{rem} \label{remH2}
Under (H2), from \eqref{defsbetagamma} and \eqref{comprsgam}, we conclude that $\mcb {S}_{\beta,\gamma}= \mcb S_{Dir} \cap \mcb S_{Neu}$.
\end{rem}
Finally we can state the main results of this work.

\subsection{Main results}\label{sec:mainres}

\begin{thm}  \label{clt}
Let $(\alpha, \beta, \gamma) \in (0, \infty) \times \mathbb{R} \times [\frac{3}{2},  2)$. Assume that either $c^{+}=c^{-}$, or $b=1/2$. Moreover, assume (H1). Consider the Markov process $\{ \eta_{t}^n : t \in [0,T] \}:=\{ \eta_{t \Theta(n)} : t \in [0,T] \}$  with generator given by \eqref{generatorsasym} with $\Theta(n)$  given by \eqref{timescale} and  suppose that it starts from the invariant state $\nu_b^n$. 

With respect to the Skorohod topology of $ \mcb{D} ( [0,T],  \mcb {S}_{\beta,\gamma}' )$, the sequence $( \mathcal{Y}_t^n)_{n \geq 1}$

 \begin{itemize}
\item 
converges in distribution, as $n \rightarrow \infty$,  to the unique stationary solution of \eqref{spde} if  $(\alpha, \beta, \gamma) \in (0, \infty) \times (\mathbb{R}-\{0\}) \times (0,  2)$;
\item  
 is tight  and any of its limit points is a stationary solution of \eqref{spde},  if $(\alpha, \beta, \gamma) \in (0, \infty) \times \{0\} \times (0,  2)$. \end{itemize}
\end{thm}

 We note that once the uniqueness of stationary solutions of \eqref{spde} is proved we obtain then convergence of  $( \mathcal{Y}_t^n)_{n \geq 1}$  when   $\alpha >0$, $\beta =0$ and $\gamma \in (0,2)$.

\begin{thm}  \label{clt2}
Let $(\alpha, \beta, \gamma) \in (0, \infty) \times [0,+\infty)\times [\frac 3 2,  2)$. Assume $b=1/2$ and (H2). Consider the Markov process $\{ \eta_{t}^n : t \in [0,T] \}:=\{ \eta_{t \Theta(n)} : t \in [0,T] \}$  with generator given by \eqref{generatorsasym} with $\Theta(n)$  given by \eqref{timescale} and starting from  $\nu_{\frac 12}^n$. With respect to the Skorohod topology of $ \mcb{D} ( [0,T],  \mcb {S}_{\beta,\gamma}' )$ the sequence $( \mathcal{Y}_t^n)_{n \geq 1}$ 

 \begin{itemize}
\item 
converges in distribution, as $n \rightarrow \infty$,  to the unique stationary energy solution of \eqref{spdefbe} with $\kappa_1=2 \alpha_a m_a$, given by \eqref{defma}, if $(\alpha, \beta, \gamma) \in (0, \infty) \times (0,+\infty) \times [\frac 32,  2)$;
\item  
 is tight  and any of its limit points is a stationary energy solution of \eqref{spdefbe} with $\kappa_1=2 \alpha_a m_a$,  if $(\alpha, \beta, \gamma) \in (0, \infty) \times \{0\} \times [\frac 32,  2)$. \end{itemize}

\end{thm}
We note that the last bullet point in  Theorem \ref{clt2} is only in the sense of convergence through subsequences since we lack the proof of uniqueness of energy solutions in that regimes, see Remark \ref{rem:uniq}.

\medskip

In Sections \ref{tightfluc}, \ref{secconvAtnG} and \ref{sectightAtnG}, we prove tightness of $( \mathcal{Y}_t^n)_{n \geq 1}$, which implies that  $( \mathcal{Y}_t^n)_{n \geq 1}$ converges (weakly) along subsequences. Finally, in Section \ref{seccharac}, we characterize the limit points of $( \mathcal{Y}_t^n)_{n \geq 1}$ as random elements satisfying the conditions conditions stated in Definition \ref{spde} or Definition \ref{spdefbe}, depending on whether (H1) or  (H2) is satisfied. Finally, we present some auxiliary results that are independent of the dynamics in Appendices \ref{propfrac}, \ref{secdiscconv} and \ref{useful}.

\section{Tightness of the  sequence of  density fluctuation fields} \label{tightfluc}

From Dynkin's formula (see Appendix 1.5 of \cite{kipnis1998scaling}), for functions $F:  \Omega_n \rightarrow \mathbb{R}$, the process
\begin{align*}
F( \eta_t^n) - F( \eta_0^n) - \int_0^t  \Theta(n) \mcb L_n  F( \eta_s^n) ds 
\end{align*}
is a martingale with respect to  $\mcb F_t^n := \sigma ( \eta_s^n: 0 \leq s \leq t)$, with quadratic variation given by
\begin{align} \label{quadvar0}
\Theta(n) \int_0^t \big[ \mcb L_n \big(  F^2 ( \eta_s^n) \big) - 2 F ( \eta_s^n) \mcb L_n F ( \eta_s^n) \big] ds.
\end{align}
Recall  \eqref{eqdensfie} and fix $G \in \mcb {S}_{\beta,\gamma}$. The process $\{ \mathcal{M}_t^n(G); \; t \in [0,T] \}$ defined by
\begin{equation} \label{defMtngfluc}
 \mathcal{M}_t^n(G) := \mathcal{Y}_t^n (G) - \mathcal{Y}_0^n (G) - \int_{0}^t  \Theta(n)  \mcb{L}_n  \mathcal{Y}_s^n (G) ds
\end{equation}
is a martingale with respect to $\mcb F_t^n$. The rightmost term on the right-hand side of \eqref{defMtngfluc} is known in the literature of IPS as the \textit{integral term}.  This is the first time to analyse.

To that end, let $r_n^{-}$ and $r_n^{+}$ be given by
\begin{align} \label{defrnpm}
\forall n \geq 2, \;  \forall x \in \Lambda_n, \quad & r_n^{-}(\tfrac{x}{n}):= \sum_{y=x}^{\infty} s(y), \quad  r_n^{+}(\tfrac{x}{n}):= \sum_{y=-\infty}^{x-n} s(y) =  \sum_{y=n-x}^{\infty} s(y). 
\end{align}
A simple computation shows that  for every $1 \leq k  \in \mathbb{N}$, it holds
\begin{equation} \label{Lbeta}
\begin{split}
\forall z \in \Lambda_n, \quad (\mcb L_n )  \big(\eta^k(z) \big) =& \sum_{y \in \Lambda_n}  s(y-z) [ \eta(y) -  \eta(z)] - \alpha_a n^{-\beta_a} \sum_{y \in \Lambda_n}  a(y-z)[ \eta(y) -  \eta(z)  ]^2 \\
+ & \alpha n^{-\beta} [ b  - \eta(z) ] [r_n^{-} \left( \tfrac{z}{n}\right) + r_n^{+} \left( \tfrac{z}{n}\right)] \\
-&  \alpha_a n^{-\beta_a} [ b   + \eta(z) (1 - 2 b) ] \Big[ \sum_{y \leq 0}   a(y-z)  + \sum_{y \geq n} a(y-z) \Big].
\end{split}
\end{equation}
Above we used \eqref{generatorsasym} and \eqref{defsa}.
As a consequence, for every $s \in [0, T]$, it holds
\begin{align}
\sqrt{n-1} \mcb{L}_n \big( \mathcal{Y}_s^n (G) \big)
= & \sum_{x,y} G( \tfrac{x}{n} ) s(y-x) [ \bar{\eta}_t^n(y) - \bar{\eta}_t^n(x) ] - \alpha n^{-\beta} \sum_{x} G( \tfrac{x}{n} ) \bar{\eta}_t^n(x) ] [ r_n^{+} \left( \tfrac{x}{n}\right) + r_n^{-} \left( \tfrac{x}{n}\right)] \nonumber \\
-& \frac{\alpha_a}{n^{\beta_a}}  \sum_{x,y }    G \left( \tfrac{x}{n}  \right)     a(y-x)   \bar{\eta}_s^n(y)(1 - 2 b) + 2 \frac{\alpha_a}{n^{\beta_a}} \sum_{x,y }    G \left( \tfrac{x}{n}  \right)     a(y-x)   \bar{\eta}_s^n(x) \bar{\eta}_s^n(y)  . \label{nullterm}
\end{align}

In the setting of both Theorems \ref{clt} and \ref{clt2}, we have that either $c^{-}=c^{+}$ or $b=1/2$, thus the first double sum in \eqref{nullterm} is equal to zero. Multiplying the last display by $\Theta(n)/\sqrt{n-1}$ and integrating from $0$ to $t$, we get from \eqref{defsa} and \eqref{timescale} that
\begin{align} \label{intrewr}
\int_{0}^t  \Theta(n)  \mcb{L}_n  \mathcal{Y}_s^n (G) ds = - A_t^n(G) + \int_0^t \frac{\Theta(n)}{\sqrt{n-1}} \sum_x \big( \mcb{K}_{n} G ( \tfrac{ x }{n} ) - \alpha n^{-\beta}  [ r_n^{+} \left( \tfrac{x}{n}\right) + r_n^{-} \left( \tfrac{x}{n}\right)] \big) \bar{\eta}_s^n(x) \; ds,
\end{align}
where the operator $\mcb{K}_{n}$ is defined on $G \in C^{\infty}([0,1])$ by 
\begin{align} \label{op_Kn} 
\forall n \geq 2, \;  \forall x \in \Lambda_n, \quad & \mcb{K}_{n} G ( \tfrac{ x }{n} ): = \sum_{  y} [ G( \tfrac{y }{n}) -G( \tfrac{ x}{n}) ] s( y -x)=\frac{c^{+}+c^{-}}{2} \sum_{  y} \mathbbm{1}_{ \{y \neq x \} } \frac{G( \tfrac{y }{n}) -G( \tfrac{ x}{n})}{|y-x|^{\gamma+1}} ,  
\end{align}
and $A_t^n(G)$ comes from the asymmetric part of $p(\cdot)$, being given by
\begin{equation} \label{defAtnG}
A_t^n(G):=  \alpha_a \int_{0}^t \frac{n^{r_a}}{ \sqrt{ n-1}} \sum_{x, y} [G \left( \tfrac{y}{n}  \right) - G \left( \tfrac{x}{n}  \right)]    a(y-x)  \bar{\eta}_s^n(x) \bar{\eta}_s^n(y) ds.
\end{equation}
In the last line, $r_a$ is given in \eqref{defrs}. 

The integral on the right-hand side of \eqref{intrewr} can be rewritten as $\mathcal{I}_t^n(G)+\mathcal{E}_t^n(G)$, where
\begin{align}
\mathcal{I}_t^n(G): =&   \int_0^t  \mathcal{Y}_s^n ( \mathbb{L}_{\alpha,\beta}^{ \gamma /2 }   G )ds, \label{princnbfluc} \\
 \mathcal{E}_t^n(G):=&  \int_{0}^{t} \frac{1}{\sqrt{n-1}} \sum _{x } \big[ \Theta(n)  \mcb {R}_{n, \alpha, \beta} G \left( \tfrac{x}{n} \right) + \Theta(n)  \mcb {K}_{n, \alpha, \beta} G \left( \tfrac{x}{n} \right) - \mathbb{L}_{\alpha,\beta}^{ \gamma /2 } G\left( \tfrac{x}{n} \right) \big]   \bar{\eta}_{s}^{n}(x) ds, \label{extranbfluc} 
\end{align}
and   $\mcb {K}_{n, \alpha, \beta}$ and $\mcb {R}_{n, \alpha, \beta}$ are given  by
\begin{align}
\mcb{K}_{n, \alpha, \beta} G : =&  \mathbbm{1}_{ \{ \beta \geq 0 \} } \mcb{K}_{n} G+ \alpha n^{-\beta} \mathbbm{1}_{ \{ \beta \leq 0 \} } \sum_{x} [r_n^{-} (\tfrac{x}{n})  + r_n^{+} (\tfrac{x}{n}) ] G (\tfrac{x}{n}),  \label{op_Knb} \\
\mcb{R}_{n, \alpha, \beta} G : = & \alpha n^{-\beta} \mathbbm{1}_{ \{ \beta > 0 \} } \sum_{x} [r_n^{-} (\tfrac{x}{n})  + r_n^{+} (\tfrac{x}{n}) ] G (\tfrac{x}{n})  + \mathbbm{1}_{ \{ \beta < 0 \} } \mcb{K}_{n} G.  \label{op_Rnb}
\end{align}
Note that  \eqref{princnbfluc} is well defined, due to Corollary \ref{corL2ab}. 

We expect the process $\mathcal{I}_t^n(G)$ to converge in law to \eqref{intprocess}. On the other hand the remainder  $\mathcal{E}_t^n(G)$  converges to zero as a consequence of the next  result.
\begin{prop} \label{convL2error}
Let $(\alpha, \beta, \gamma) \in (0, \infty) \times \mathbb{R} \times (0,  2)$ and $G \in \mcb {S}_{\beta,\gamma}$. Then, 
\begin{align*}
\lim_{n \rightarrow \infty} \mathbb{E}_{\nu_b^n} \big[ \sup_{t \in [0,T]} \big(  \mathcal{E}_t^n(G) \big)^2 \big] =0.
\end{align*}
\end{prop} 
We postpone the proof of Proposition \ref{convL2error} to Section \ref{tighterror}.
 
Now, combining \eqref{defMtngfluc} with \eqref{intrewr}, we conclude that 
\begin{equation} \label{decomp}
\forall G \in \mcb {S}_{\beta,\gamma}, \; \forall n \geq 2, \; \forall t \in [0, T], \quad \mathcal{Y}_t^n (G) = \mathcal{Y}_0^n (G) + \mathcal{M}_t^n(G) + \mathcal{I}_t^{n}(G) + \mathcal{E}_t^{n}(G) + A_t^{n}(G). 
\end{equation}
From Remark \ref{remnucfre}, we can use Mitoma's criterion (see \cite{mitoma}) in the same way as  in \cite{tertufluc} since for every $(\beta,\gamma) \in R_0$, our space of test functions  $\mcb {S}_{\beta,\gamma}$,  are nuclear  Fr\'echet spaces, see Remark \ref{remnucfre}. For the convenience of the reader, we state Mitoma's criterion in the next result. 
\begin{prop} \textbf{(Mitoma's criterion)} \label{propmit} Let $\mcb{N}$ be a nuclear Fr\'echet space. A sequence $\{ x_t^n; t \in [0,T] \}_{n \geq 2}$ in $\mcb{D} ([0,T], \mcb{N}' )$ of stochastic processes is tight with respect to the Skorohod topology if, and only if, the sequence of real-valued processes $\{ x_t^n (G); t \in [0,T] \}_{n \geq 2}$  is tight with respect to the Skorohod topology of $\mcb{D} ([0,T], \mathbb{R})$, for every $G \in  \mcb{N}$ fixed.
\end{prop}
Thus, we apply Proposition \ref{propmit} to obtain the tightness for $\{ \mathcal{Y}_t^n; t \in [0,T] \}_{n \geq 2} \}$ in $\mcb{D} ([0,T], \mcb {S}_{\beta,\gamma}' )$. In order to do so, according to \eqref{decomp}, we fix $G \in \mcb {S}_{\beta,\gamma}$ and obtain the tightness for the initial fields $\big( \mathcal{Y}_0^n (G) \big)_{n \geq 2}$, the martingale terms $ \big( \mathcal{M}_t^n(G) \big)_{n \geq 2}$, integral terms $\big(\mathcal{I}_t^n(G) \big)_{n \geq 2}$, error terms $\big(\mathcal{E}_t^n(G) \big)_{n \geq 2}$ and asymmetric terms $\big(A_t^n(G) \big)_{n \geq 2}$ in $\mcb{D} ([0,T], \mathbb{R})$, separately. 

\subsection{Tightness of the sequence of initial fields}

Next result is analogous to Proposition 3 of \cite{franco2017equilibrium}, therefore we leave its proof to the reader.
\begin{prop} \textbf{(Convergence of the initial field)} \label{tightfluc1}
Let $(\beta, \gamma) \in R_0$. For every $G \in \mathcal{S}_{\beta,\gamma}$, for every $t \geq 0$ and every $\lambda \in \mathbb{R}$, it holds
\begin{align*}
\lim_{n \rightarrow \infty}  \mathbb{E}_{\nu_b^n} [ \exp \{  i\lambda  \mathcal{Y}_t^n (G) \} ] = \exp \Big\{ - \frac{ \lambda^2 \chi(b)}{2} \int_{0}^1 G^2(u) du  \Big\}.
\end{align*}
Thus, the sequence $\big( \mathcal{Y}_t^n (G) \big)_{n \geq 1}$ converges in distribution to a mean zero Gaussian variable with variance $\lambda^2 \chi(b) \| G \|_{L^2}^2 $ and in particular it is tight with respect to the Skorohod topology of $\mcb{D} ([0,T], \mathbb{R})$. Furthermore, $( \mathcal{Y}_0^n)_{n \geq 2}$ converges in distribution to $ \mathcal{Y}_0$, where $\mathcal{Y}_0$ is a mean zero Gaussian field with covariance given by \eqref{covinifie}.
\end{prop}

\begin{rem} \label{remstat}
For any $t \in [0,T]$, the sequence $(\mathcal{Y}_t^n)_{n \geq 2}$ converges to a white noise of variance $\chi(b)$. Combining this with \eqref{densfield}, we have that if $(\mathcal{Y}_t)_{0 \leq t \leq T}$ is a limit point of the sequence $(\mathcal{Y}_t)_{0 \leq t \leq T}$, $(\mathcal{Y}_t)_{0 \leq t \leq T}$ is stationary and satisfies \eqref{condUSC}, see Definition \ref{defstat}. 
\end{rem}

\subsection{Tightness of the sequence of martingale terms}

Our goal in this subsection is to prove the following result.
\begin{prop} \label{tightmartterm}
Let $(\alpha, \beta, \gamma) \in (0, \infty) \times \mathbb{R} \times (0,  2)$. Then the sequence $\{ \mathcal{M}_t^n(G); \; t \in [0,T] \}_{n \geq 2}$ is tight with respect to the Skorohod topology of $\mcb{D}( [0,T], \mathbb{R})$, for every $G \in \mcb {S}_{\beta,\gamma}$.
\end{prop}
We will use Aldous' criterion to prove last result.

\begin{prop} (Aldous' criterion) \label{tightmarterm0}\\
Let $(\alpha, \beta, \gamma) \in (0, \infty) \times \mathbb{R} \times (0,  2)$ and $G\in \mcb {S}_{\beta,\gamma}$. The sequence $\{ \mathcal{M}_t^n(G); \; t \in [0,T] \}_{n \geq 2}$ is tight with respect to the Skorohod topology of $\mcb{D}( [0,T], \mathbb{R})$ if
\begin{enumerate}
\item
$\lim_{A \rightarrow \infty} \varlimsup_{n \rightarrow \infty}  \mathbb{P}_{\nu_b^n} \big( \sup_{t \in [0,T]} |\mathcal{M}_t^n(G) | > A \big)=0;$
\item
For every $\varepsilon >0$, we have
\begin{align*}
\lim _{\omega \rightarrow 0^+} \varlimsup_{n \rightarrow\infty} \sup_{\tau  \in \mathcal{T}_{T}, \; 0 \leq t \leq \omega} \mathbb{P}_{\nu_b^n} \big(  |\mathcal{M}_{T \wedge (\tau+ t)}^n(G)- \mathcal{M}_{\tau}^n(G)  |> \varepsilon \big)  =0,
\end{align*}
where $\mathcal{T}_T$ denotes the set of stopping times $\tau$ such that $0 \leq \tau \leq T$.
\end{enumerate}
\end{prop}
In order to apply last proposition, we state the following lemma, whose proof is postponed to the end of this subsection. We observe that Lemma \ref{lemconvmartterm1} is analogous to Lemma 3.3 in \cite{flucstefano}.
\begin{lem} \label{lemconvmartterm1}
Let $(\alpha, \beta, \gamma) \in (0, \infty) \times \mathbb{R} \times (0,  2)$ and recall the definition of $\mathcal{P}^{\gamma}_{\alpha, \beta}$ in \eqref{defPalbetgam}. Then
\begin{align} \label{limexpquadvar}
\forall 0 \leq s \leq t \leq T, \quad \lim_{n \rightarrow \infty} \mathbb{E}_{\nu_b^n} [ \langle \mathcal{M}^n(G)  \rangle_t - \langle \mathcal{M}^n(G)  \rangle_s ]  =  2 \chi(b) (t-s)   \mathcal{P}^{\gamma}_{\alpha, \beta} G,
\end{align}
for any $G \in \mcb S_{\beta,\gamma}$. Moreover, for all $ 0 \leq s \leq t \leq T$ and  $n \geq 2$ it holds 
\begin{align} \label{limexpquadvarb}
 \mathbb{E}_{\nu_b^n} [ \langle \mathcal{M}^n(G)  \rangle_t - \langle \mathcal{M}^n(G)  \rangle_s ] \lesssim \chi(b) (t-s). 
\end{align}
\end{lem}
From Lemma \ref{lemconvmartterm1}, for every $G \in \mcb {S}_{\beta,\gamma}$ fixed, it holds
\begin{equation} \label{MthnL2}
\forall n \geq 2, \quad \sup_{r \in [0,T]} \mathbb{E}_{\nu_b^n} \big[ \big( \mathcal{M}_r^n(G) \big)^2 \big] = \sup_{r \in [0,T]} \mathbb{E}_{\nu_b^n} [  \langle  \mathcal{M}^n(G)\rangle_r  ]  \lesssim \chi(b) T .
\end{equation}
Next, we will obtain Proposition \ref{tightmartterm} as a consequence of Proposition \ref{tightmarterm0}, \eqref{MthnL2} and \eqref{limexpquadvarb}. 
\begin{proof} [Proof of Proposition \ref{tightmartterm}]
Observe that the sequence of martingales $\{ \mathcal{M}_t^n(G); t \in [0,T] \}_{n \geq 1}$ satisfies the first condition of Proposition \ref{tightmarterm0}, due to Doob's inequality and \eqref{MthnL2}.
 In order to obtain the second condition of Proposition \ref{tightmarterm0}, we apply Chebyshev's inequality, which leads to
\begin{align*}
 \mathbb{P}_{\nu_b^n} \big(  |\mathcal{M}_{T \wedge (\tau+ t)}^n(G)- \mathcal{M}_{\tau}^n(G)  |> \varepsilon \big) \leq & \frac{1}{\varepsilon^2} \mathbb{E}_{\nu_b^n} \big[  \big(\mathcal{M}_{T \wedge (\tau+ t)}^n(G)- \mathcal{M}_{\tau}^n(G)  \big)^2 \big]\\
  \leq & \frac{1}{\varepsilon^2} \mathbb{E}_{\nu_b^n} \big[  \langle  \mathcal{M}^n(G)\rangle_{T \wedge (\tau+ t)} - \langle  \mathcal{M}^n(G)\rangle_{\tau} \big].
\end{align*}
The proof ends by combining last display with \eqref{limexpquadvarb}. 
\end{proof}
We end this subsection by presenting the proof of Lemma \ref{lemconvmartterm1}.

\begin{proof} [Proof for Lemma \ref{lemconvmartterm1}]

Combining \eqref{Lbeta} with \eqref{generatorsasym} and \eqref{quadvar0}, for every $0 \leq s \leq t \leq T$,  it holds
\begin{equation} \label{quadvar1}
\begin{split}
\langle \mathcal{M}^n(G) \rangle_t - \langle \mathcal{M}^n(G) \rangle_s=& \frac{\Theta(n)}{2(n-1)} \int_s^t \sum_{x, y } [ G\left( \tfrac{y}{n}\right) -  G\left( \tfrac{x}{n}\right) ]^2  s(y-x) [\eta_{s}^{n}(y) - \eta_{s}^{n}(x)]^2  dr \\
-& \frac{\Theta(n)}{2(n-1)} \alpha_a n^{-\beta_a} \int_s^t \sum_{x, y } [ G\left( \tfrac{y}{n}\right) -  G\left( \tfrac{x}{n}\right) ]^2   a(y-x) [\eta_{s}^{n}(y) - \eta_{s}^{n}(x)]   dr \\
+ & \frac{\alpha \Theta(n)}{n^{\beta} (n-1)}  \int_s^t \sum_{x } G^2\left( \tfrac{x}{n}\right) [ b  +  \eta_{s}^{n}(x) (1 - 2b) ] [ r_n^{+} \left( \tfrac{x}{n}\right) + r_n^{-} \left( \tfrac{x}{n}\right)] dr \\
+ & \frac{\alpha_a \Theta(n)}{n^{\beta_a} (n-1)} \frac{c_{+} - c_{-}}{c_{+} + c_{-}} \int_s^t  \sum_{x } G^2 \left( \tfrac{x}{n}\right) \bar{\eta}_s^n(x)  [ r_n^{+} \left( \tfrac{x}{n}\right) - r_n^{-} \left( \tfrac{x}{n}\right)] dr.
\end{split}
\end{equation}
Combining Fubini's Theorem with \eqref{quadvar1}, we get
\begin{align*}
\mathbb{E}_{\nu_b^n} [\langle \mathcal M^n (G)  \rangle_t - \langle \mathcal M^n (G)  \rangle_s ] = 2  \chi(b) (t-s)[ \hat{\mathcal{A}}_{n,\beta} (G) +   \alpha  \mathcal{B}_{n,\beta} (G) ],
\end{align*}
where for every $n \geq 2$,  $\beta \in \mathbb{R}$ and $G \in C^{\infty}([0,1])$, $\hat{\mathcal{A}}_{n,\beta} (G)$ and $\mathcal{B}_{n,\beta} (G)$ are given by 
\begin{align}
	&\hat{\mathcal{A}}_{n,\beta} (G): = \frac{\Theta(n)}{n-1}    \sum_{  x,y } s(y-x) [G( \tfrac{y}{n}) - G( \tfrac{x}{n})]^2, \label{op_Anb} \\
	& \mathcal{B}_{n,\beta} (G): = \frac{\Theta(n)}{n^{\beta}(n-1)} \sum_{ x } [r_n^{-}(\tfrac{x}{n})+r_n^{+}(\tfrac{x}{n})]  G^2( \tfrac{x}{n}). \label{op_Bnb}
\end{align}
Therefore the proof ends as a consequence of Propositions \ref{prop1lem1convmart} and \ref{prop2lem1convmart}.
\end{proof}
\subsection{Tightness of the sequence of integral terms}
The tightness of the integral terms is a consequence of next result. 
\begin{prop}  \label{tightfluc2a}
	The sequence $ \big\{ \mathcal{I}_t^n(G) ; t \in [0,T] \big\}_{n \geq 2}$ is tight with respect to the Skorohod topology of $\mcb{D} ([0,T], \mathbb{R})$, for every $(\alpha, \beta, \gamma) \in (0, \infty) \times \mathbb{R} \times (0,  2)$ and every $G \in \mcb {S}_{\beta,\gamma}$.
\end{prop}
\begin{proof}
Let $G \in \mcb {S}_{\beta,\gamma}$. From Corollary \ref{corL2ab} we get $\mathbb{L}_{\alpha,\beta}^{ \gamma /2 } G \in L^2([0,1])$, thus
	\begin{align*} 
		K_1:=& \sup_{t \in [0,T], n \geq 2} \mathbb{E}_{\nu_b^n} \big[  \big( \mathcal{Y}_t^n ( \mathbb{L}_{\alpha,\beta}^{ \gamma /2 }  G ) \big)^2 \big]  = \sup_{ n \geq 2} \Big\{ \frac{\chi(b)}{n-1} \sum_{x} [\mathbb{L}_{\alpha,\beta}^{ \gamma /2 }  G (\tfrac{x}{n})]^2  \Big\}< \infty.
	\end{align*}
	Therefore, Cauchy-Schwarz's inequality and Fubini's Theorem lead to
	\begin{align*}
		\mathbb{E}_{\nu_b^n} [  |  \mathcal{I}_t^n(G) -  \mathcal{I}_r^n(G) |^2   ] = \mathbb{E}_{\nu_b^n}\Big[  \Big| \ \int_r^t  \mathcal{Y}_s^n ( \mathbb{L}_{\alpha,\beta}^{ \gamma /2 }   G )ds \Big|^2  \Big] \leq K_1   (t-r)^2, 
	\end{align*}
for every $r, t \in [0,T]$. Since $\mathcal{I}_0^n(G)=0$, combining last display with Kolmogorov-Centsov's criterion (see Proposition 4.3 of \cite{jarafluc}), we obtain the tightness of $ \big\{  \mathcal{I}_t^n(G) ; t \in [0,T] \big\}_{n \geq 2}$.
\end{proof}

\subsection{Tightness of the sequence of error terms} \label{tighterror}

Now we state a result that is useful in what follows.
\begin{prop} \label{conv0tight}
Let $\{x_t^n: t \in [0,T]\}_{n \geq 2}$ be a sequence of stochastic process such that
\begin{align*} 
		\lim_{n \rightarrow \infty} \mathbb{E}_{\nu_b^n} \big[ \sup_{t \in [0,T]} (x_t^n )^2 \big] =0. 
	\end{align*}
Then $\{x_t^n: t \in [0,T]\}_{n \geq 2}$ converges in distribution to zero, as $n \rightarrow \infty$. In particular, $\{x_t^n: t \in [0,T]\}_{n \geq 2}$  is tight with respect to the Skorohod topology of $\mcb{D}( [0,T], \mathbb{R})$. 
\end{prop}
\begin{proof}
The proof is a direct consequence of Chebyshev's inequality.
\end{proof}
Thus, the tightness for $\big(\mathcal{E}_t^n(G) \big)_{n \geq 2}$ is a direct consequence of Proposition \ref{convL2error}. 

Observe that that from Cauchy-Schwarz's inequality, for every sequence $(a_j)_{j \geq 1}$,
\begin{equation} \label{CSdisc}
\forall m \in \mathbb{N}, \quad  \sum_{j=1}^m a_j \leq m \sum_{j=1}^m (a_j)^2, 
\end{equation}
from where it follows that
\begin{equation} \label{CSdiscexp}
\forall t \in [0, T], \quad \mathbb{E}_{\nu_b^n} \Big[ \sup_{s \in [0,t]} \Big( \int_{0}^s \sum_{j=1}^m g_j(\eta_r^n)\; dr  \Big)^2  \Big] \leq  m  \sum_{j=1}^m\mathbb{E}_{\nu_b^n} \Big[ \sup_{s \in [0,t]} \Big( \int_{0}^s  g_j(\eta_r^n)\; dr  \Big)^2  \Big] ,
\end{equation}
for every $m \in \mathbb{N}$ and every $g_1, \ldots, g_m: \Omega_n \rightarrow \mathbb{R}$.
Combining the definition of $\mathcal{E}_t^n(G)$ given in \eqref{extranbfluc} with an application of \eqref{CSdiscexp} for $m=2$ and Proposition \ref{conv0tight}, the result of  Proposition \ref{convL2error} is a direct consequence of the next two results.
\begin{prop} \label{convrem1}
	Assume $(\alpha, \beta, \gamma) \in (0, \infty) \times \mathbb{R} \times (0,  2)$. For every $G \in \mcb {S}_{\beta,\gamma}$, it holds
	\begin{align*} 
		\lim_{n \rightarrow \infty} \mathbb{E}_{\nu_b^n} \Big[ \sup_{t \in [0,T]} \Big( \int_{0}^{t} \frac{1}{\sqrt{n-1}} \sum _{x } \big[  \Theta(n)  \mcb {K}_{n, \alpha, \beta} G \left( \tfrac{x}{n} \right) - \mathbb{L}_{\alpha,\beta}^{ \gamma / 2 } G\left( \tfrac{x}{n} \right) \big]   \bar{\eta}_{s}^{n}(x) ds \Big)^2 \Big] =0. 
	\end{align*}
\end{prop}
\begin{prop} \label{convrem2}
	Assume $(\alpha, \beta, \gamma) \in (0, \infty) \times \mathbb{R} \times (0, 2)$. For every $G \in \mcb {S}_{\beta,\gamma}$, it holds
	\begin{align*} 
		\lim_{n \rightarrow \infty} \mathbb{E}_{\nu_b^n} \Big[ \sup_{t \in [0,T]} \Big(\int_{0}^{t} \frac{1}{\sqrt{n-1}} \sum _{x } [ \Theta(n)  \mcb {R}_{n, \alpha, \beta} G \left( \tfrac{x}{n} \right)   ]   \bar{\eta}_{s}^{n}(x) ds \Big)^2 \Big] =0. 
	\end{align*}
\end{prop}
In order to obtain the aforementioned results, it will be useful the following inequality, which is a direct consequence of Cauchy-Schwarz's inequality and Fubini's Theorem:
\begin{equation} \label{CSFub}
\forall t \in [0, T], \; \forall g: \Omega_n \rightarrow \mathbb{R}, \quad \mathbb{E}_{\nu_b^n} \Big[ \sup_{s \in [0,t]} \Big(  \int_0^s g(\eta_{r}^n) dr \Big)^2 \Big] \leq t \int_0^t \mathbb{E}_{\nu_b^n} \big[ g^2( \eta_{r}^n) \big] dr.
\end{equation}
Now we present the proof for Proposition \ref{convrem1}.
\begin{proof}[Proof of Proposition \ref{convrem1}]
The proof is a direct consequence of \eqref{CSFub} and Corollary \ref{convknbeta}.
\end{proof}
In order to obtain Proposition \ref{convrem2}, we will apply the following result. Since it can be obtained from Lemma 6.1 of \cite{flucstefano} for the choice $a_j=G(j/n)$ for every $j \in \Lambda_n$, we leave its proof to the reader. 
\begin{prop} \label{lem61stefano}
	Let $(c_n)_{n \geq 2} \subset \mathbb{R}$ and $q \in \{1, 2\}$. It holds
	\begin{align*}
		\mathbb{E}_{\nu_b} \Big[ \sup_{t \in [0,T]} \Big( \int_0^t c_n \sum_x [G(x/n)]^{q} r_n^{\pm}\left( \tfrac{x}{n} \right) \bar{\eta}_s^n(x) ds \Big)^2  \Big] \lesssim  \frac{(c_n)^2 n^{\beta}}{\alpha \Theta(n)} \sum_x  r_n^{\pm}\left( \tfrac{x}{n} \right) [G(x/n)]^{2q}.
	\end{align*}
\end{prop}
The remaining ingredient which leads to Proposition \ref{convrem2} is Proposition \ref{prop3assum} below.
\begin{prop} \label{prop3assum}
	Assume that at least one of the next affirmations:
	\begin{itemize}
		\item 
		$(\alpha, \beta, \gamma) \in (0, \infty) \times (0, \infty) \times (0, 1]$ and $G \in C^{\infty}([0,1])$; 
		\item 
		$\alpha >0$, $\gamma \in (1,2)$, $\beta > \gamma - 1$ and $G \in  C^{\infty}([0,1])$;
		\item 
		$(\alpha, \beta, \gamma) \in (0, \infty) \times (0, \infty) \times (1, 2)$ and $G \in \mcb S_{Dir}$. 
	\end{itemize}
	Then
	\begin{align*}
		\lim_{n \rightarrow \infty} \alpha \frac{n^{\gamma}}{n^{\beta+1}} \sum_x [r^{-} \left( \tfrac{x}{n} \right) + r^{+} \left( \tfrac{x}{n} \right) ]  G^2 \left( \tfrac{x}{n} \right) =0.
	\end{align*}
\end{prop}
\begin{proof}
	Applying \eqref{boundrnpm}, last display is bounded from above by a constant (independent of $n$) times
	\begin{equation} \label{L2notrob}
		\lim_{n \rightarrow \infty} \alpha \frac{n^{\gamma}}{n^{\beta+1}} \sum_x  [x^{-\gamma} + (n-x)^{-\gamma} ] G^2 \left( \tfrac{x}{n} \right).
	\end{equation}
For any $G \in C^{\infty}([0,1])$, last expression is bounded from above by
	\begin{equation} \label{L2neu}
		\lim_{n \rightarrow \infty} \alpha \| G \|_{\infty}^2  \frac{n^{\gamma}}{n^{\beta+1}} \sum_x  [x^{-\gamma} + (n-x)^{-\gamma} ] \lesssim \lim_{n \rightarrow \infty} \frac{n^{\gamma}}{n^{\beta+1}} \sum_x x^{-\gamma}.
	\end{equation}
\begin{itemize}
\item
Now we treat the case $(\alpha, \beta, \gamma) \in (0, \infty) \times (0, \infty) \times (0, 1]$ and $G \in C^{\infty}([0,1])$. For $\gamma =1$, the sum in \eqref{L2neu} is of order $\log(n)$, thus the term in \eqref{L2neu} is of order $n^{-\beta} \log(n)$, which vanishes as $n \rightarrow \infty$, due to $\beta >0$. Alternatively, if $\gamma \in (0,1)$, last limit in \eqref{L2neu} can be rewritten as
	\begin{align*}
		\lim_{n \rightarrow \infty} \frac{1}{n^{\beta}} \frac{1}{n}  \sum_x \Big( \frac{x}{n} \Big)^{-\gamma} \lesssim \lim_{n \rightarrow \infty} \frac{1}{n^{\beta}} \int_0^1 u^{-\gamma} du = \lim_{n \rightarrow \infty} \frac{1}{n^{\beta}(1-\gamma)} =0.
	\end{align*} 
	Again, last limit holds due to $\beta >0$. 
\item
Now we analyse the regime $\alpha >0$, $\gamma \in (1,2)$, $\beta > \gamma - 1$ for $G \in C^{\infty}([0,1])$. In this case, the sum in \eqref{L2neu} is convergent and the display in \eqref{L2neu} is of order $n^{\gamma - 1 -\beta}$, which goes to zero as $n \rightarrow \infty$, since $\beta > \gamma - 1$.
\item
Finally, we study the regime $(\alpha, \beta, \gamma) \in (0, \infty) \times (0, \infty) \times (1, 2)$ for $G \in \mcb S_{Dir}$. In this case, applying \eqref{expGSd} for $d=1$, the display in \eqref{L2notrob} is bounded from above by
	\begin{align*}
		& \lim_{n \rightarrow \infty} \alpha \| G^{(1)} \|_{\infty}^2  \frac{n^{\gamma}}{n^{\beta+1}} \Big\{ \sum_x x^{-\gamma} \Big( \frac{x}{n} \Big)^{2} + \sum_x (n-x)^{-\gamma} \Big( \frac{n-x}{n} \Big)^{2} \Big \} \\
		\lesssim & \lim_{n \rightarrow \infty} \frac{n^{\gamma}}{n^{\beta+1}} \sum_x x^{-\gamma} \Big( \frac{x}{n} \Big)^{2} = \lim_{n \rightarrow \infty} \frac{1}{n^{\beta}} \frac{1}{n}  \sum_x \Big( \frac{x}{n} \Big)^{2-\gamma} \\&\lesssim \lim_{n \rightarrow \infty} \frac{1}{n^{\beta}} \int_0^1 u^{2-\gamma} du = \lim_{n \rightarrow \infty} \frac{1}{n^{\beta}(3-\gamma)} =0.
	\end{align*}
Last limit holds due to $\beta >0$. This ends the proof. 
\end{itemize}		
\end{proof}
Finally, we present the proof for Proposition \ref{convrem2}.
\begin{proof} [Proof of Proposition \ref{convrem2}]
Let $(\alpha, \beta, \gamma) \in (0, \infty) \times \mathbb{R} \times (0, 2)$ and $G \in \mcb {S}_{\beta,\gamma}$.
\begin{itemize}
\item
If $(\alpha, \beta, \gamma) \in (0, \infty) \times \{ 0 \} \times (0, 2)$, the proof is a direct consequence of \eqref{op_Rnb}.

\item
If $(\alpha, \beta, \gamma) \in (0, \infty) \times (-\infty, 0) \times (0,  2)$, the proof ends by combining \eqref{timescale}, \eqref{defsbetagamma}, \eqref{op_Rnb}, \eqref{CSFub}, \eqref{CSdisc} with Propositions \ref{convfrac} and \ref{Lqreg}.

\item
If $(\alpha, \beta, \gamma) \in (0, \infty) \times (0, \infty) \times (0, 2)$, the proof follows by combining \eqref{op_Rnb} with \eqref{timescale}, Proposition \ref{prop3assum} and an application of Proposition \ref{lem61stefano} for $c_n = \alpha \Theta(n) n^{-\beta} (\sqrt{n-1})^{-1}$ and $q=1$.
\end{itemize}
\end{proof}

\subsection{Conclusion}

From \eqref{decomp} and Propositions \ref{propmit},  \ref{tightfluc1}, \ref{tightmartterm}, \ref{tightfluc2a}, \ref{convL2error} and \ref{conv0tight}, the tightness for $\{ \mathcal{Y}_t^n; t \in [0,T] \}_{n \geq 2} \}$ in $\mcb{D} ([0,T], \mcb {S}_{\beta,\gamma}' )$ follows from the next two results.
\begin{prop} \label{convAtnG}
Assume that $G \in C^{\infty}([0,1])$ and at least one of the cases \textit{(I), (II), (III) or (IV)} of (H1) holds. 
Then 
$
\lim_{n \rightarrow \infty} \mathbb{E}_{\nu_b^n} \big[ \sup_{t \in [0,T]} \big(  A_t^n(G) \big)^2 \big] =0.
$
\end{prop}

\begin{prop} \label{tightAtnG}
Assume that $G \in C^{\infty}([0,1])$, $\beta \geq 0$, $r_a =3/2$ and $c^{+} \neq c^{-}$. Then the sequence $\{ \mathcal{A}_t^n(G); t \in [0,T] \}_{n \geq 1}$  is tight with respect to the Skorohod topology of $\mcb{D}( [0,T], \mathbb{R})$.
\end{prop}
Since the arguments required to obtain the last two results are particularly technical, we prove Proposition \ref{convAtnG} and  \ref{tightAtnG} in Section \ref{secconvAtnG}. 

\section{Proof of Propositions \ref{convAtnG} and  \ref{tightAtnG}} \label{sec:nonlinear}

\subsection{Proof of Proposition \ref{convAtnG}} 
\label{secconvAtnG}
If $c^{+}=c^{-}$, we get from \eqref{defsa} that $a(z)=0$, for any $z \in \mathbb{Z}$. Thus, from \eqref{defAtnG} we get that $A_t^n(G)=0$ for every $G \in C^{\infty}([0,1])$, every $t \in [0, T]$ and every $n \geq 2$, and Proposition \ref{convAtnG} holds trivially. Therefore, in the remainder of this section we will assume that $c^{+} \neq c^{-}$.
We observe that Proposition \ref{convAtnG} is a consequence of Propositions \ref{propraless1}, \ref{propraeq1} and \ref{propgammales32} below. We begin by obtaining Proposition \ref{propraless1}, which corresponds to the case $r_a < 1$ in Proposition \ref{convAtnG}.

\subsubsection{The case $r_a < 1$}
\begin{prop} \label{propraless1}
Assume that $G \in C^{\infty}([0,1])$ and $r_a < 1$. Then 
\begin{align*}
\lim_{n \rightarrow \infty} \mathbb{E}_{\nu_b^n} \Big[ \sup_{t \in [0,T]} \Big(  \int_{0}^t \frac{n^{r_a}}{ \sqrt{ n-1}} \sum_{x, y} [G \left( \tfrac{y}{n}  \right) - G \left( \tfrac{x}{n}  \right)]    a(y-x)  \bar{\eta}_s^n(x) \bar{\eta}_s^n(y) ds \Big)^2  \Big]=0.
\end{align*}
\end{prop}
\begin{proof}
Combining \eqref{CSFub}, \eqref{defsa} and the Mean Value`s Theorem, last expectation is bounded from above by a constant, times
\begin{align*}
n^{2 r_a - 3} \sum_{x,z} z^{-2\gamma} \leq n^{2 r_a - 2} \sum_{z} z^{-2\gamma}.
\end{align*}
If $\gamma > 1/2$, last sum is convergent and last display is of order $n^{2 r_a - 2}$, which goes to zero as $n \rightarrow \infty$, due to the assumption that $r_a < 1$. If $\gamma = 1/2$, last sum is of order $\log(n)$ and last display is of order $n^{2 r_a - 2} \log(n)$, which goes to zero as $n \rightarrow \infty$, due to the assumption that $r_a < 1$. Finally, if $\gamma \in (0, \, 1/2)$, last display can be rewritten as
\begin{align*}
n^{2 r_a - 1} \frac{1}{n} \sum_{z} z^{-2\gamma} \leq \frac{1}{n} n^{2 r_a - 2 \gamma -1} \frac{1}{n} \sum_{z=1}^n \Big( \frac{z}{n} \Big)^{-2\gamma} \lesssim n^{2 r_a - 2 \gamma -1} \int_0^1 u^{- 2 \gamma} du \lesssim n^{2 r_a - 2 \gamma -1} \lesssim n^{-1},
\end{align*}
which goes to zero as $n \rightarrow \infty$. In last bound we applied \eqref{comprsgam}.
\end{proof}
\subsubsection{The case $r_a \geq 1$}

We note that the proof of \textit{(II)} in (H1) is given below in Proposition \ref{propraeq1}, but in order to prove it we need to derive intermediate results which are given below.

For the ase of notation, we denote $b_n:=n^{r_a-1} (\sqrt{ n-1})^{-1}$, for every $n \geq 2$. Performing a change of variables and a second order Taylor expansion on $G$ around $x/n$, we observe that
\begin{align}
A_t^n(G)=& \alpha_a \int_{0}^t \frac{b_n}{n} \sum_{x=1}^{n/2} \sum_{y=x+1}^{n-1} \Delta G \left(  \xi^n_{x,y}  \right)     a(y-x) (y-x)^2  \bar{\eta}_s^n(x) \bar{\eta}_s^n(y) ds \label{Atngrad0a} \\
+& \alpha_a \int_{0}^t \frac{b_n}{n} \sum_{x=1+n/2}^{n-1} \; \sum_{y=1+n/2}^{x-1} \Delta G \left(  \xi^n_{x,y}  \right)     a(y-x) (y-x)^2  \bar{\eta}_s^n(x) \bar{\eta}_s^n(y) ds \label{Atngrad0b}  \\
+ & 2 \alpha_a \int_{0}^t b_n \sum_{x=1}^{n/2} \nabla G \left(  \tfrac{x}{n}  \right) \sum_{y=x+1}^{n-1}      a(y-x)  (y-x) \bar{\eta}_s^n(x) \bar{\eta}_s^n(y) ds \label{Atngrad1} \\
+ & 2 \alpha_a \int_{0}^t b_n \sum_{x=1+n/2}^{n-1} \nabla G \left(  \tfrac{x}{n}  \right) \sum_{y=1+n/2}^{x-1}      a(y-x)  (y-x) \bar{\eta}_s^n(x) \bar{\eta}_s^n(y) ds, \label{Atngrad2}
\end{align}
where $\xi^n_{x,y}\in(x/n,y/n)$.

Next we estimate the terms in \eqref{Atngrad0a} and \eqref{Atngrad0b} by applying the following result.
\begin{lem} \label{lem46}
Assume that $G \in C^{\infty}([0,1])$. Then it holds
\begin{align*}
&   \mathbb{E}_{\nu_b^n} \Big[ \sup_{t \in [0,T]} \Big(  \int_{0}^t \frac{b_n}{n} \sum_{x=1}^{n/2} \sum_{y=x+1}^{n-1} \Delta G \left(  \xi^n_{x,y}  \right)     a(y-x) (y-x)^2 \bar{\eta}_s^n(x) \bar{\eta}_s^n(y) ds \Big)^2  \Big] \lesssim \frac{1}{n}; \\
&   \mathbb{E}_{\nu_b^n} \Big[ \sup_{t \in [0,T]} \Big(  \int_{0}^t \frac{b_n}{n} \sum_{x=1+n/2}^{n-1} \; \sum_{y=1+n/2}^{x-1} \Delta G \left(  \xi^n_{x,y}  \right)     a(y-x) (y-x)^2 \bar{\eta}_s^n(x) \bar{\eta}_s^n(y)ds \Big)^2  \Big] \lesssim \frac{1}{n}.
\end{align*}
\end{lem} 
\begin{proof}
Combining \eqref{CSFub} and \eqref{defsa}, both expectations in last display are bounded from above by a positive constant depending only on $T$ and $G$, times
\begin{align*}
&  \frac{(b_n)^2}{n^2}  \sum_{x, z} z^{2 - 2 \gamma} \lesssim  \frac{1}{n} n^{2 r_a - 2 \gamma -1} \frac{1}{n} \sum_{z=1}^n \Big( \frac{z}{n} \Big)^{2-2\gamma} \lesssim n^{2 r_a - 2 \gamma -1} \int_0^1 u^{2- 2 \gamma} du \lesssim n^{2 r_a - 2 \gamma -1} \lesssim \frac{1}{n},
\end{align*}
ending the proof. In last line, we applied \eqref{comprsgam} and the fact that $n \leq 2(n-1)$, for every $n \geq 2$.
\end{proof}
In order to treat \eqref{Atngrad1}, given $K_n \in \{1, \ldots, n/3\}$ such that $\lim_{n \rightarrow \infty} K_n = \infty$, we rewrite \eqref{Atngrad1} as
\begin{align}
& 2 \alpha_a \int_{0}^t b_n \sum_{x=1}^{n/2} \nabla G \left(  \tfrac{x}{n}  \right) \sum_{y=x+K_n}^{n-1}       a(y-x)  (y-x)\bar{\eta}_s^n(x) \bar{\eta}_s^n(y) ds \label{Atngrad1a} \\
+ & 2 \alpha_a \int_{0}^t b_n \sum_{x=1}^{n/2} \nabla G \left(  \tfrac{x}{n}  \right)  \sum_{y=x+1}^{x+K_n-1}       a(y-x)  (y-x) \bar{\eta}_s^n(x) \bar{\eta}_s^n(y) ds. \label{Atngrad1b}
\end{align}
Analogously, \eqref{Atngrad2} can be rewritten as
\begin{align}
& 2 \alpha_a \int_{0}^t b_n \sum_{x=K_n+1+n/2}^{n-1} \nabla G \left(  \tfrac{x}{n}  \right) \sum_{y=1+n/2}^{x-K_n}      a(y-x)  (y-x) \bar{\eta}_s^n(x) \bar{\eta}_s^n(y) ds \label{Atngrad2a} \\
+ & 2 \alpha_a \int_{0}^t b_n \sum_{x=K_n+1+n/2}^{n-1} \nabla G \left(  \tfrac{x}{n}  \right) \sum_{y=x-K_n+1}^{x-1}      a(y-x)  (y-x) \bar{\eta}_s^n(x) \bar{\eta}_s^n(y) ds \label{Atngrad2b} \\
+& 2 \alpha_a \int_{0}^t b_n \sum_{x=1+n/2}^{K_n+n/2} \nabla G \left(  \tfrac{x}{n}  \right) \sum_{y=1+n/2}^{x-1}      a(y-x)  (y-x) \bar{\eta}_s^n(x) \bar{\eta}_s^n(y) ds. \label{Atngrad2c}
\end{align}
Due to \eqref{comprsgam}, a necessary condition for $r_a \geq 1$ is $\gamma \geq 1$, thus we will state this assumption in the following lemmas. In order to treat \eqref{Atngrad1a} and \eqref{Atngrad2a} (which are the terms for which $|x-y| \geq K_n)$, we apply the following result, which is analogous to Lemma 4.5 in \cite{jarafluc}.
\begin{lem} \label{lem45}
Assume that $\gamma \geq 1$, $1 \leq K_n \leq n/3$ and $G \in C^{\infty}([0,1])$. Then
\begin{align*}
& \lim_{n \rightarrow \infty}  \mathbb{E}_{\nu_b^n} \Big[ \sup_{t \in [0,T]} \Big(  \int_{0}^t b_{n} \sum_{x=1}^{n/2} \nabla G \left(  \tfrac{x}{n}  \right) \sum_{y=x+K_n}^{n-1}     a(y-x)  (y-x) \bar{\eta}_s^n(x) \bar{\eta}_s^n(y) ds \Big)^2  \Big] \lesssim \frac{n^{2r_a-2}}{K_n^{2 \gamma - 1}}; \\
& \lim_{n \rightarrow \infty}  \mathbb{E}_{\nu_b^n} \Big[ \sup_{t \in [0,T]} \Big(   \int_{0}^t b_{n} \sum_{x=K_n+1+n/2}^{n-1} \nabla G \left(  \tfrac{x}{n}  \right) \sum_{y=1+n/2}^{x-K_n}      a(y-x)  (y-x) \bar{\eta}_s^n(x) \bar{\eta}_s^n(y) ds \Big)^2  \Big] \lesssim \frac{n^{2r_a-2}}{K_n^{2 \gamma - 1}}.
\end{align*}
\end{lem} 
\begin{proof}
From \eqref{CSFub}, both expectations in last display are bounded from above by a constant times
\begin{align*}
 n^{2r_a-2} \sum_{z = K_n}^{n-1} z^{- 2 \gamma}  = &  n^{2r_a-2} \Big[ K_n^{- 2 \gamma} + \sum_{z = K_n+1}^{n} z^{- 2 \gamma} \Big] \leq n^{2r_a-2} \Big[ K_n^{- 2 \gamma} + \int_{K_n}^{n} z^{- 2 \gamma} \Big] \lesssim \frac{n^{2r_a-2}}{K_n^{2 \gamma - 1}},
\end{align*}
ending the proof. To get the last upper bound we applied the assumption that $\gamma \geq 1$.
\end{proof}
In \eqref{Atngrad1b}, \eqref{Atngrad2b} and \eqref{Atngrad2c}, we do not have a lower bound for controlling terms depending on the distance $|x-y|$, thus we introduce some extra notation in order to do that. Given $L \in \{1, \ldots, n/3\}$ we define 
\begin{equation*}
\overrightarrow{\eta}^{L}(x):=\frac{1}{L} \sum_{y=0}^{L-1} \bar{\eta}(x+y), \; 1 \leq x \leq n/2;   \quad \quad \overleftarrow{\eta}^{L}(x):=\frac{1}{L} \sum_{y=0}^{L-1} \bar{\eta}(x-y), \;  1 + n/2 \leq x \leq n-1.
\end{equation*}
For every $L \in \{2, \ldots, n/3\}$ we define
\begin{equation}    \label{medemp} 
\begin{split}
& \overrightarrow{\psi}_x^{L}(\eta)= \frac{L}{L-1} \Big[ \big( \overrightarrow{\eta}^{L}(x) \big)^2 - \frac{b(1-b)}{L} \Big], \quad 1 \leq x \leq n/2, \; \eta \in \Omega_n ;   \\
& \overleftarrow{\psi}_x^{L}(\eta)= \frac{L}{L-1} \Big[ \big( \overleftarrow{\eta}^{L}(x) \big)^2 - \frac{b(1-b)}{L} \Big], \quad 1 + n/2 \leq x \leq n-1, \; \eta \in \Omega_n .
\end{split}
\end{equation}
From \eqref{medemp}, it is not hard to prove that for every $L \in \{2, \ldots, n/3\}$ and  for every $x \in \{1, \ldots,  n/2 \}$  it holds
\begin{equation} \label{boundpsi}
\begin{split}
 \int_{\Omega_n}  \overrightarrow{\psi}_x^{L}(\eta)  d \nu_b^n=0, \quad  \int_{\Omega_n} \big( \overrightarrow{\psi}_x^{L}(\eta) \big)^2 d \nu_b^n \leq \frac{6}{L^2}, 
\end{split}
\end{equation}
and the same holds for  the left averages  for $x \in \{1 + n/2, \ldots, n-1\}.$
Next, we state a lemma that will be useful in what follows.

\begin{lem} \label{lemauxtight}
Assume that $\gamma \in (0,2)$ and $G \in C^{\infty}([0,1])$. Let $n \geq 2$, $J_n \subset \{1+n/2, \ldots, n-1\}$, $h_n: \Lambda_n \rightarrow \mathbb{R}$,  $n/3 \geq \hat{K} \in \mathbb{N}$ and $t \in [0, T]$. Then, 
\begin{align*}
 \mathbb{E}_{\nu_b^n} \Big[ \sup_{s \in [0,t]} \Big(  \int_{0}^s
\sum_{x=1}^{n/2}  \nabla G \left(  \tfrac{x}{n}  \right) h_{n}(x)  \overrightarrow{\psi}_x^{\hat{K}}(\eta_r^n)  dr  \Big)^2  \Big] \leq \frac{6}{\hat{K}} t^2 \sum_{x=1}^{n/2}[\nabla G \left(  \tfrac{x}{n}  \right) h_{n}(x)]^2, \\
 \mathbb{E}_{\nu_b^n} \Big[ \sup_{s \in [0,t]} \Big(  \int_{0}^s
\sum_{x \in J_n}  \nabla G \left(  \tfrac{x}{n}  \right) h_{n}(x)  \overrightarrow{\psi}_x^{\hat{K}}(\eta_r^n)  dr  \Big)^2  \Big] \leq \frac{6}{\hat{K}} t^2 \sum_{x \in J_n}[\nabla G \left(  \tfrac{x}{n}  \right) h_{n}(x)]^2.
\end{align*}
\end{lem}

\begin{proof}
We only prove the first bound in last display, the proof of the second one being analogous. From \eqref{CSFub}, the first expectation in last display is bounded from above by 
\begin{align*}
t \int_0^t \mathbb{E}_{\nu_b^n}  \Big[ \Big(  \sum_{x=1}^{n/2}  \nabla G \left(  \tfrac{x}{n}  \right) h_n(x)  \overrightarrow{\psi}_x^{\hat{K}}(\eta_r^n)  \Big)^2 \Big] dr
\end{align*}
Note that for any $x,y \in \Lambda_n$, $\overrightarrow{\psi}_x^{\hat{K}}(\eta_s^n)$ and $\overrightarrow{\psi}_y^{\hat{K}}(\eta_s^n)$ are independent for every $s \in [0, \; T]$ whenever the supports of $\overrightarrow{\psi}_x^{\hat{K}}(\eta_s^n)$ and $\overrightarrow{\psi}_y^{\hat{K}}(\eta_s^n)$ are disjoint. From \eqref{medemp}, a sufficient condition for this is $|x-y| \geq \hat{K}$. Keeping this in mind, we rewrite the expression in last display as
\begin{align} \label{explem410} 
t \int_0^t \mathbb{E}_{\nu_b^n}  \Big[ \Big( \sum_{j=1}^{\hat{K}} \sum_{x \in \mathcal{X}_j} \nabla G \left(  \tfrac{x}{n}  \right) h_n(x)  \overrightarrow{\psi}_x^{\hat{K}}(\eta_r^n) \Big)^2 \Big] dr,
\end{align}
where for any $j \in \{1, \ldots, \hat{K} \}$, $\mathcal{X}_j$ is given by 
$
\mathcal{X}_j:=\{1 \leq x \leq n/2: \exists q \in \mathbb{Z} \, \text{such that} \, x = q \hat{K} +j \}.
$
In words, $\mathcal{X}_j$ is the subset of $\{1, \ldots, n/2\}$ whose remainder in the Euclidean division by $\hat{K}$ is $j$ either $j$, if $j < \hat{K}$; or $0$, if $j= \hat{K}$. Applying \eqref{CSdisc} for $m=\hat{K}$, \eqref{explem410} is bounded from above by
\begin{align*}
\hat{K}  t    \int_0^t   \sum_{j=1}^{\hat{K}}   \sum_{x,y \in \mathcal{X}_j} \nabla G \left(  \tfrac{x}{n}  \right)  h_n(x) \nabla G \left(  \tfrac{y}{n}  \right)  \tilde{h}_{n,\varepsilon}(y) \mathbb{E}_{\nu_b^n}[ \overrightarrow{\psi}_x^{K_n}(\eta_s^n)  \overrightarrow{\psi}_y^{K_n}(\eta_r^n) ] dr.
\end{align*}
For every $j \in \{1, \ldots, \hat{K}\}$ \textit{fixed}, $\overrightarrow{\psi}_x^{\hat{K}}(\eta)$ and $\overrightarrow{\psi}_y^{\hat{K}}(\eta)$ have zero mean and are independent when $x$ and $y$ are different elements of $\mathcal{X}_j$, since $|x-y| \geq \hat{K}$. Therefore, last display can be rewritten as
\begin{align*}
\hat{K}  t  \int_0^t   \sum_{j=1}^{\hat{K}}   \sum_{x \in \mathcal{X}_j} [ \nabla G \left(  \tfrac{x}{n}  \right)  h_n(x)]^2  \mathbb{E}_{\nu_b^n}\big[ \big( \overrightarrow{\psi}_x^{\hat{K}}(\eta_r^n) \big)^2  \big] dr.
\end{align*}
The proof ends by combining last display with \eqref{medemp} and \eqref{boundpsi}.
\end{proof}
Making the change $y=x+z$, for every $K_n \in \{2, \ldots, n/3\}$, \eqref{Atngrad1b} can be rewritten as
\begin{align}
& 2 \alpha_a \int_{0}^t b_n \sum_{x=1}^{n/2} \nabla G \left(  \tfrac{x}{n}  \right)  \sum_{z=1}^{K_n-1}      z a(z)   \overrightarrow{\psi}_x^{K_n}(\eta_s^n) ds \label{Atngrad1b1} \\
+ & 2 \alpha_a \int_{0}^t b_n \sum_{x=1}^{n/2} \nabla G \left(  \tfrac{x}{n}  \right) \sum_{z=1}^{K_n-1}       z a(z) \big[  \bar{\eta}_s^n(x) \bar{\eta}_s^n(x+z) - \overrightarrow{\psi}_x^{K_n}(\eta_s^n) \big] ds \label{Atngrad1b2}.
\end{align}
Analogously, making the change $y=x-z$,  \eqref{Atngrad2b} can be rewritten as
\begin{align}
& 2 \alpha_a \int_{0}^t b_n \sum_{x=K_n+1+n/2}^{n-1} \nabla G \left(  \tfrac{x}{n}  \right) \sum_{z=1}^{K_n-1}     z a(z)  \overleftarrow{\psi}_x^{K_n}(\eta_s^n) ds \label{Atngrad2b1} \\ 
+& 2 \alpha_a \int_{0}^t b_n \sum_{x=K_n+1+n/2}^{n-1} \nabla G \left(  \tfrac{x}{n}  \right) \sum_{z=1}^{K_n-1}     z a(z) \big[ \bar{\eta}_s^n(x) \bar{\eta}_s^n(x-z) - \overleftarrow{\psi}_x^{K_n}(\eta_s^n) \big] ds. \label{Atngrad2b2}
\end{align}
In last display, we used  $-z a(-z) = z a(z)$, due to \eqref{defsa}. Similarly, \eqref{Atngrad2c} can be rewritten as
\begin{align}
& 2 \alpha_a \int_{0}^t b_n \sum_{x=1+n/2}^{K_n+n/2} \nabla G \left(  \tfrac{x}{n}  \right) \sum_{z=1}^{x-1-n/2}     z a(z)  \overleftarrow{\psi}_x^{K_n}(\eta_s^n) ds \label{Atngrad2c1} \\ 
+& 2 \alpha_a \int_{0}^t b_n \sum_{x=1+n/2}^{K_n+n/2} \nabla G \left(  \tfrac{x}{n}  \right) \sum_{z=1}^{x-1-n/2}     z a(z) \big[ \bar{\eta}_s^n(x) \bar{\eta}_s^n(x-z) - \overleftarrow{\psi}_x^{K_n}(\eta_s^n) \big] ds. \label{Atngrad2c2}
\end{align}
In order to treat \eqref{Atngrad1b1}, \eqref{Atngrad2b1} and \eqref{Atngrad2c1}, we apply a result which is analogous to Lemma 4.10 in \cite{jarafluc}.
\begin{lem} \label{lem410}
Let $G \in C^{\infty}([0,1])$, $\gamma \geq 1$, $n \geq 2$ and $K_n \in \mathbb{N} \cap [2, n/3]$. Then it holds
\begin{align*}
& \mathbb{E}_{\nu_b^n} \Big[ \sup_{s \in [0, T]} \Big( \int_{0}^s b_{n} \sum_{x=1}^{n/2} \nabla G \left(  \tfrac{x}{n}  \right)  \sum_{z=1}^{K_n-1}      z a(z)   \overrightarrow{\psi}_x^{K_n}(\eta_r^n) dr  \Big)^2  \Big] \lesssim T^2 \frac{n^{2 r_a}}{n^2} \frac{B_{\gamma}(K_n)}{K_n} \|\nabla G\|^2_{L^2}; \\
& \mathbb{E}_{\nu_b^n} \Big[ \sup_{s \in [0, T]} \Big(  \int_{0}^s b_{ n} \sum_{x=K_n+1+n/2}^{n-1} \nabla G \left(  \tfrac{x}{n}  \right)  \sum_{z=1}^{K_n-1}     z a(z)  \overleftarrow{\psi}_x^{K_n}(\eta_r^n) dr  \Big)^2  \Big] \lesssim T^2 \frac{n^{2 r_a}}{n^2} \frac{B_{\gamma}(K_n)}{K_n} \|\nabla G\|^2_{L^2}; \\
& \mathbb{E}_{\nu_b^n} \Big[ \sup_{s \in [0, T]} \Big(  \int_{0}^s b_{ n} \sum_{x=1+n/2}^{K_n+n/2} \nabla G \left(  \tfrac{x}{n}  \right) \sum_{z=1}^{x-1-n/2}       z a(z)  \overleftarrow{\psi}_x^{K_n}(\eta_r^n) dr \Big)^2  \Big] \lesssim T^2 \frac{n^{2 r_a}}{n^2} \frac{B_{\gamma}(K_n)}{K_n} \|\nabla G\|^2_{L^2}.
\end{align*}
Above, $B_{\gamma}(K_n)=1$ for $\gamma >1$ and $B_{\gamma}(K_n)=[\log(K_n
)]^2$ for $\gamma =1$. 
\end{lem}
\begin{proof}
The proof is a direct consequence of an application Lemma \ref{lemauxtight} for $t=T$.
\end{proof}
For every $F: \Omega_n \rightarrow \mathbb{R}$, we denote by $\supp(F)$ the minimal set $\tilde{\Lambda}_n \subset \Lambda_n$ such that 
for all $ y \in \tilde{\Lambda}_n$ and for all $ \eta_1, \eta_2 \in \Omega_n$ such that $ \eta_1(y)= \eta_2(y)$ then $ F(\eta_1)=F(\eta_2)$.
Next we observe that Proposition A.1 and Corollary A.3 of \cite{jarafluc} can be adapted for our case by using the same arguments of \cite{jarafluc}. Therefore, we can treat \eqref{Atngrad1b2}, \eqref{Atngrad2b2} and \eqref{Atngrad2c2} by invoking the following result, which is a stronger version of Proposition 4.9 in \cite{jarafluc}. 
\begin{prop} \label{prop48}
Let $m \in \mathbb{N}$ and $1 \leq k_0 < k_1 < \ldots < k_m \leq n-1$ be elements of $\mathbb{N}$. Let $F_1, \ldots, F_m: \Omega_n \rightarrow \mathbb{R}$ be such that $\supp(F_i) \subset \{k_{i-1}+1, \ldots ,  k_i\}$ for any $i \in \{1, \ldots, m\}$. Assume that $\int_{\Omega_n} F_i d \nu_{\sigma}^n =0$ for any $\sigma \in [0,1]$ and any $i \in \{1, \ldots, m\}$. Then there exists a positive universal constant $\kappa$ such that for $F:=F_1 + \ldots + F_m$, we get 
\begin{align*}
\forall t \in [0, T], \quad \mathbb{E}_{\nu_b^n} \Big[ \sup_{s \in [0,t] } \Big(  \int_{0}^s F( \eta_r^n) dr \Big)^2  \Big] \leq  \kappa t  \sum_{i=1}^m   \frac{(k_i - k_{i-1})^{\gamma}}{\Theta(n)} \int_{\Omega_n} [F_i( \eta)]^2 d \nu_b^n  .
\end{align*}
\end{prop} 
\begin{proof}
The proof is a direct consequence of Lemma 4.3 in \cite{CLO} and Proposition 4.8 in \cite{jarafluc}.
\end{proof}

Keeping last proposition in mind, we will treat \eqref{Atngrad1b2}, \eqref{Atngrad2b2} and \eqref{Atngrad2c2}. In order to do that, we apply a result which is analogous to Lemma 4.11 in \cite{jarafluc}.
\begin{lem} \label{lem411}
Let $G \in C^{\infty}([0,1])$, $\gamma \geq 1$, $n \geq 2$ and $K_n \in \mathbb{N} \cap [2, n/3]$. Then it holds
\begin{align*}
& \mathbb{E}_{\nu_b^n} \Big[ \sup_{t \in [0,T]} \Big( \int_{0}^t b_{n} \sum_{x=1}^{n/2} \nabla G \left(  \tfrac{x}{n}  \right) \sum_{z=1}^{K_n-1}       z a(z) \big[  \bar{\eta}_s^n(x) \bar{\eta}_s^n(x+z) - \overrightarrow{\psi}_x^{K_n}(\eta_s^n) \big] ds  \Big)^2  \Big] \lesssim  \frac{(K_n)^{1+\gamma} n^{2 r_a }}{\Theta(n) n^2} ; \\
& \mathbb{E}_{\nu_b^n} \Big[ \sup_{t \in [0,T]} \Big(  \int_{0}^t b_{n} \sum_{x=K_n+1+n/2}^{n-1} \nabla G \left(  \tfrac{x}{n}  \right) \sum_{z=1}^{K_n-1}     z a(z) \big[ \bar{\eta}_s^n(x) \bar{\eta}_s^n(x-z) - \overleftarrow{\psi}_x^{K_n}(\eta_s^n) \big] ds  \Big)^2  \Big] \lesssim \frac{(K_n)^{1+\gamma} n^{2 r_a }}{\Theta(n) n^2}; \\
& \mathbb{E}_{\nu_b^n} \Big[ \sup_{t \in [0,T]} \Big(  \int_{0}^t b_{n} \sum_{x=1+n/2}^{K_n+n/2} \nabla G \left(  \tfrac{x}{n}  \right) \sum_{z=1}^{x-1-n/2}       z a(z) \big[ \bar{\eta}_s^n(x) \bar{\eta}_s^n(x-z) - \overleftarrow{\psi}_x^{K_n}(\eta_s^n) \big] ds \Big)^2  \Big] \lesssim \frac{(K_n)^{1+\gamma} n^{2 r_a }}{\Theta(n) n^2}.
\end{align*}
\end{lem}
\begin{proof}
We only prove the first bound in last display, the proof of the second and third ones being analogous. From \eqref{CSdiscexp}, the first expectation in last display is bounded from above by
\begin{align} \label{boundlem411}
K_n \sum_{j=1}^{K_n} \mathbb{E}_{\nu_b^n} \Big[    \sup_{t \in [0,T]}  \Big( \int_{0}^t  \sum_{x \in \mathcal{X}_j}  b_{n}    \nabla G \left(  \tfrac{x}{n}  \right) \sum_{z=1}^{K_n-1}       z a(z) \big[  \bar{\eta}_s^n(x) \bar{\eta}_s^n(x+z) - \overrightarrow{\psi}_x^{K_n}(\eta_s^n) \big] ds  \Big)^2  \Big],
\end{align}
where for any $j \in \{1, \ldots, K_n \}$, $\mathcal{X}_j$ is given by 
\begin{equation*} 
\mathcal{X}_j:=\{1 \leq x \leq n/2: \exists q \in \mathbb{Z} \quad \text{such that} \quad x = q K_n \}.
\end{equation*}
In order to treat the expression in last line, we apply Proposition \ref{prop48}. Keeping this in mind, for every $j \in \{1, \ldots, K_n\}$, let $\{j, K_n+j, \ldots, q_j K_n+j\}$ be an enumeration of the elements of $\mathcal{X}_j$. In particular, we get $K_n q_j \leq n/2$. Next, for every $i \in \{0, 1, \ldots, q_j\}$, define $k_i:=j+i K_n -1$. Thus, $k_0 < k_1 < \ldots k_{q_j}$. Moreover, for every $i \in \{1, \ldots, q_j\}$, define $y_i:=j+i K_n$ and $F_i:  \Omega_n \rightarrow \mathbb{R}$ by
\begin{align*}
F_i( \eta):= b_{n}    \nabla G \left(  \tfrac{y_i}{n}  \right) \sum_{z=1}^{K_n-1}       z a(z) \big\{ [\eta ( y_i ) - b] [\eta ( y_i +z ) - b] - \overrightarrow{\psi}_{y_i}^{K_n}(\eta) \big\}, \quad \eta \in  \Omega_n.
\end{align*}
From \eqref{defberprod} and \eqref{medemp}, we get $\int_{\Omega_n} F_i d \nu_{\sigma}^n =0$ for any $\sigma \in [0,1]$ and any $i \in \{1, \ldots, q_j\}$. Moreover, due to \eqref{medemp}, $\supp(F_i) \subset \{k_{i-1}+1, \ldots, k_i \}$, for any $i \in \{1, \ldots, q_j\}$. From Proposition \ref{prop48}, we get
\begin{align*}
\mathbb{E}_{\nu_b^n} \Big[    \sup_{t \in [0,T]}  &\Big( \int_{0}^t   \sum_{x \in \mathcal{X}_j} b_{n}    \nabla G \left(  \tfrac{x}{n}  \right) \sum_{z=1}^{K_n-1}       z a(z) \big[  \bar{\eta}_s^n(x) \bar{\eta}_s^n(x+z) - \overrightarrow{\psi}_x^{K_n}(\eta_s^n) \big] ds  \Big)^2  \Big] \\
= & \mathbb{E}_{\nu_b^n} \Big[ \sup_{t \in [0,T] } \Big(  \int_{0}^t \sum_{i=1}^{q_j} F_i ( \eta_s^n) ds \Big)^2  \Big] \lesssim  T  \sum_{i=1}^{q_j}   \frac{(k_i - k_{i-1})^{\gamma}}{\Theta(n)} \int_{\Omega_n} [F_i( \eta)]^2 d \nu_b^n  \\
= & \frac{T (K_n)^{\gamma}}{\Theta(n)} \sum_{i=1}^{q_j} \int_{\Omega_n} [F_i( \eta)]^2 d \nu_b^n \lesssim  \frac{ (K_n)^{\gamma}}{\Theta(n)} (b_n)^2 \sum_{i=1}^{q_j} [ \nabla G \left(  \tfrac{y_i}{n}  \right) ]^2. 
\end{align*} 
The upper bound in last line is due to \eqref{defberprod}, \eqref{medemp}, \eqref{boundpsi}, \eqref{defsa} and the assumption $\gamma \geq 1$.
The proof ends by combining the final upper bound in last display with \eqref{boundlem411}.
\end{proof}

Now we obtain Proposition \ref{propraeq1}, which corresponds to \textit{(II)} in Proposition \ref{convAtnG}.

\subsubsection{The case $r_a \geq 1$ and $\gamma=1$}

\begin{prop} \label{propraeq1}
Assume that $G \in C^{\infty}([0,1])$, $r_a \geq 1$ and $\gamma = 1$. Then
\begin{align*}
\lim_{n \rightarrow \infty} \mathbb{E}_{\nu_b^n} \big[ \sup_{t \in [0,T]} \big(  A_t^n(G) \big)^2 \big] =0.
\end{align*}
\end{prop}
\begin{proof}
Due to \eqref{defrs} and \eqref{comprsgam}, the conditions if $r_a \geq 1$ and $\gamma = 1$ only occur simultaneously iff $\beta=\beta_a=0$ and $\gamma=r_a=1$. Since $A_t^n(G)$ is equal to the sum of the terms in \eqref{Atngrad0a}, \eqref{Atngrad0b}, \eqref{Atngrad1} and \eqref{Atngrad2}, we get from Lemmas \ref{lem46}, \ref{lem45}, \ref{lem410} and \ref{lem411} and an application of \eqref{CSdiscexp} for $m=4$ that the expectation in last display is bounded from above by a constant times
\begin{align*}
\frac{1}{n}+ \frac{1}{K_n} + \frac{[\log(K_n)]^2}{K_n} + \frac{(K_n)^2}{\Theta(n)} = \frac{1}{n}+ \frac{1}{K_n} + \frac{[\log(K_n)]^2}{K_n} + \frac{(K_n)^2}{n}. 
\end{align*} 
In last equality we applied \eqref{timescale}. Thus, the proof ends by choosing $K_n=n^{u}$, for any $0 < u < 1/2$. 
\end{proof}

\subsubsection{The case $r_a \geq 1$ and $\gamma>1$}

Our next goal is to prove the following result, which corresponds to the remaining cases in Proposition \ref{convAtnG}.
\begin{prop} \label{propgammales32}
Let $G \in C^{\infty}([0,1])$. Assume that $\gamma >1$ and $2 r_a < 3$. Moreover assume that either $\beta \geq 0$; or  $\beta <0$ and $2 r_a \leq 3 + \beta$. Then 
\begin{align*}
\lim_{n \rightarrow \infty} \mathbb{E}_{\nu_b^n} \big[ \sup_{t \in [0,T]} \big(  A_t^n(G) \big)^2 \big] =0.
\end{align*}
\end{prop}
\begin{rem}
Last conditions are less restrictive that the condition $\gamma < 3/2$, which was necessary to state an analogous result in \cite{jarafluc}. Indeed, if $\gamma < 3/2$, then from \eqref{comprsgam} we get $2 r_a \leq 2 \gamma < 3$. Moreover, if $\beta < 0$ and $\gamma < 3/2$, then 
$
2 r_a = 2 \gamma + 2 \beta - 2 \beta_a < 3 + 2 \beta < 3 + \beta.
$
\end{rem}
Since Proposition \ref{convAtnG} is a direct consequence of Propositions \ref{propraless1}, \ref{propraeq1} and \ref{propgammales32}, we are done if we can prove the latter result. Therefore, for the remainder of this section, we will assume that $2r_a < 3$ and fix $\delta$ such that $0 < \delta < 3 - 2 r_a \leq 1$. In this setting, the upper bound provided by Lemma \ref{lem411} may be not good enough for obtaining Proposition \ref{propgammales32}. Alternatively, we apply a multiscale analysis, similarly as it is done in Section 4.4.3 of \cite{jarafluc}.

Our first step is to state the following result.
\begin{lem} \label{lemeq413}
Assume that $G \in C^{\infty}([0,1])$ and $\gamma > 1$. Let $0 \leq \lambda < \tilde{\lambda} <1$, $L_n=n^{\lambda}$ and $\tilde{L}_n =n^{\tilde{\lambda}}$. Then
\begin{align*}
& \mathbb{E}_{\nu_b^n} \Big[ \sup_{t \in [0,T]} \Big( \int_{0}^t b_{n} \sum_{x=1}^{n/2} \nabla G \left(  \tfrac{x}{n}  \right) \sum_{z=L_n}^{\tilde{L}_n-1}       z a(z) \big[  \bar{\eta}_s^n(x) \bar{\eta}_s^n(x+z) - \overrightarrow{\psi}_x^{\tilde{L}_n}(\eta_s^n) \big] ds  \Big)^2  \Big]  \lesssim \frac{\tilde{L}_n^{1+\gamma}}{L_n^{2 \gamma - 1}}  \frac{n^{2 r_a - 2}}{\Theta(n)} ; \\
& \mathbb{E}_{\nu_b^n} \Big[ \sup_{t \in [0,T]} \Big(  \int_{0}^t b_{n} \sum_{x=\tilde{L}_n+1+n/2}^{n-1} \nabla G \left(  \tfrac{x}{n}  \right) \sum_{z=L_n}^{\tilde{L}_n-1}     z a(z) \big[ \bar{\eta}_s^n(x) \bar{\eta}_s^n(x-z) - \overleftarrow{\psi}_x^{\tilde{L}_n}(\eta_s^n) \big] ds  \Big)^2  \Big] \lesssim \frac{\tilde{L}_n^{1+\gamma}}{L_n^{2 \gamma - 1}}  \frac{n^{2 r_a - 2}}{\Theta(n)}.
\end{align*}
\end{lem}
\begin{proof}
It is enough to apply arguments analogous to the ones presented in the proof for Lemma \ref{lem411}.
\end{proof}
Keeping Lemma \ref{lemeq413} in mind, assume that $ \tilde{\lambda} (\gamma+1) - (2 \gamma - 1) \lambda - (2 - \gamma)=-\hat{\delta}$ for some $\hat{\delta} >0$ and some $0 \leq \lambda < \tilde{\lambda} <1$.  In this case, due to \eqref{timescale} \eqref{defrs} and \eqref{condrate}, the upper bound in Lemma \ref{lemeq413} is at most of order $n^{-\hat{\delta}}$, which vanishes as $n \rightarrow \infty$.  

Given $\hat{\delta} >0$, we will make use of a sequence $(\lambda_j)_{j \geq 0}$ such that $\lambda_0=0$ and 
\begin{align} \label{reclaw}
\forall j \geq 0, \quad \hat{\delta} + \lambda_{j+1} (\gamma+1) = (2 \gamma - 1) \lambda_j + 2 - \gamma.
\end{align}
In particular, choosing $j=0$ in last expression, from $\lambda_0=0$ we get $\lambda_1=(2-\gamma-\hat{\delta}) (\gamma + 1)^{-1}$. Since we want to obtain $\lambda_1 > \lambda_0 =0$, we will impose $\hat{\delta} < 2 - \gamma$. Now from \eqref{reclaw}, $\forall j \geq 0$ it holds
$ 2 - \gamma - \hat{\delta} = \lambda_{j+1} (1+\gamma) - \lambda_j (2 \gamma - 1) = \lambda_{j+2} (1+\gamma) - \lambda_{j+1} (2 \gamma - 1), 
$ and this gives  $\forall j \geq 0$ that
$ (1+\gamma) \lambda_{j+2} - 3 \gamma \lambda_{j+1} + (2 \gamma - 1) \lambda_j =0.
$
From linear recurrence tools, there exist real constants $A_1$, $A_2$  such that
\begin{align*}
\forall j \geq 0, \quad \lambda_j = A_1 + A_2 \Big( \frac{2 \gamma - 1}{\gamma + 1} \Big)^j.
\end{align*}
Since  $\lambda_0=0$, we get $A_2=-A_1$. Combining this with $\lambda_1=(2-\gamma-\hat{\delta}) (\gamma + 1)^{-1}$, we have
\begin{equation} \label{reclambdaj}
\forall j \geq 0, \quad \lambda_j = \frac{2-\gamma-\hat{\delta}}{2-\gamma} \Big[ 1 - \Big( \frac{2 \gamma - 1}{\gamma +1} \Big)^j \Big] < 1 - \frac{\hat{\delta}}{2-\gamma}. 
\end{equation}
From $\gamma < 2$ we get $2 \gamma - 1 < \gamma +1$, which leads to $0 \leq \lambda_j < \lambda_{j+1}$ for any $j \geq 0$ and
\begin{align*}
\lim_{j \rightarrow \infty} \lambda_j = \frac{2-\gamma-\hat{\delta}}{2-\gamma} = 1 - \frac{\hat{\delta}}{2-\gamma}.
\end{align*}
Recalling that $0 < \delta < 3 - 2 r_a \leq 1$ is fixed, we will fix $\hat{\delta}=(2-\gamma) \delta/2$, which gives $(\lambda_j)_{j \geq 0}$, due to \eqref{reclambdaj}. Moreover, we can fix $N \in \mathbb{N}$ such that $1-\delta < \lambda_N < \frac{2-\delta}{2}$. Furthermore, we define $K_j:=n^{\lambda_j}$ for every $j \in \{0, 1, \ldots, N\}$.

Next, we state some auxiliary results that are useful to obtain Proposition \ref{propgammales32}. 
\begin{lem} \label{lem414}
Assume that $G \in C^{\infty}([0,1])$ and $\gamma > 1$. Then
\begin{align*}
& \mathbb{E}_{\nu_b^n} \Big[ \sup_{t \in [0,T]} \Big( \int_{0}^t b_{n} \sum_{x=1}^{n/2} \nabla G \left(  \tfrac{x}{n}  \right) \sum_{j=1}^{N} \sum_{z=K_{j-1}}^{K_j-1}  z a(z) \big[  \bar{\eta}_s^n(x) \bar{\eta}_s^n(x+z) - \overrightarrow{\psi}_x^{K_j}(\eta_s^n) \big] ds  \Big)^2  \Big] \lesssim n^{-\hat{\delta}}  ; \\
& \mathbb{E}_{\nu_b^n} \Big[ \sup_{t \in [0,T]} \Big(  \int_{0}^t b_{n} \sum_{x=K_N+n/2}^{n-1} \nabla G \left(  \tfrac{x}{n}  \right)\sum_{j=1}^{N} \sum_{z=K_{j-1}}^{K_j-1}     z a(z) \big[ \bar{\eta}_s^n(x) \bar{\eta}_s^n(x-z) - \overleftarrow{\psi}_x^{K_j}(\eta_s^n) \big] ds  \Big)^2  \Big] \lesssim  n^{-\hat{\delta}}  ; \\
& \mathbb{E}_{\nu_b^n} \Big[ \sup_{t \in [0,T]} \Big(  \int_{0}^t b_{n} \sum_{i=1}^N \sum_{x=K_{i-1}+n/2}^{K_i-1+n/2} \nabla G \left(  \tfrac{x}{n}  \right) \sum_{j=1}^{i-1} \sum_{z=K_{j-1}}^{K_j-1}     z a(z) \big[ \bar{\eta}_s^n(x) \bar{\eta}_s^n(x-z) - \overleftarrow{\psi}_x^{K_j}(\eta_s^n) \big] ds  \Big)^2  \Big] \lesssim  n^{-\hat{\delta}}  ; \\
& \mathbb{E}_{\nu_b^n} \Big[ \sup_{t \in [0,T]} \Big(  \int_{0}^t b_{n} \sum_{i=1}^N \sum_{x=K_{i-1}+n/2}^{K_i-1+n/2} \nabla G \left(  \tfrac{x}{n}  \right) \sum_{z=K_{i-1}}^{x-1-n/2}     z a(z) \big[ \bar{\eta}_s^n(x) \bar{\eta}_s^n(x-z) - \overleftarrow{\psi}_x^{K_i}(\eta_s^n) \big] ds  \Big)^2  \Big] \lesssim  n^{-\hat{\delta}} .
\end{align*}
\end{lem}
\begin{proof}
Applying \eqref{CSdiscexp} for $m=N$, Lemma \ref{lemeq413} and \eqref{reclaw}, we obtain the first two upper bounds.
On the other hand, the third and fourth expectations in the statement of the lemma can be treated by combining \eqref{CSdisc} with arguments analogous to the ones used to obtain Lemma \ref{lemeq413}. 
\end{proof}
Now, we observe that applying Minkowski's inequality  $\forall t \in [0, T]$ it holds
\begin{equation} \label{mink}
 \mathbb{E}_{\nu_b^n} \Big[ \sup_{s \in [0,t]} \Big( \int_{0}^s \sum_{j=1}^m g_j(\eta_r^n)\; dr  \Big)^2  \Big] \leq \Big( \sum_{j=1}^m  \mathbb{E}_{\nu_b^n} \Big[ \sup_{s \in [0,t]} \Big( \int_{0}^s  g_j(\eta_r^n)\; dr  \Big)^2  \Big]^{1/2}  \Big)^2,
\end{equation}
for every $m \in \mathbb{N}$ and every $g_1, \ldots, g_m: \Omega_n \rightarrow \mathbb{R}$.

The next ingredient of our multiscale argument is the following lemma.
\begin{lem} \label{lemmult0}
If $G \in C^{\infty}([0,1])$, $n \geq 2$, $h_n: \Lambda_n \rightarrow \mathbb{R}$, $J_n \subset \{1+n/2, \ldots, n-1\}$, $2 \leq D_1 \leq D_2 \leq  n/3$ and $\gamma \in (0,2)$, there exists a constant $C_0>0$ independent of $n$, $t$ and $G$ such that for any $t \in [0, T]$, 
\begin{align*}
& \mathbb{E}_{\nu_b^n} \Big[ \sup_{s \in [0,t]} \Big( \int_{0}^s  \sum_{x=1}^{n/2} \nabla G \left(  \tfrac{x}{n}  \right) h_n(x) \big[  \overrightarrow{\psi}_x^{D_1}(\eta_r^n) - \overrightarrow{\psi}_x^{D_2}(\eta_r^n) \big] dr  \Big)^2  \Big] \leq  C_0 \frac{(D_2)^{\gamma-1} }{\Theta(n) } t  \sum_{x=1}^{n/2} [ \nabla G \left(  \tfrac{x}{n}  \right) h_n(x)]^2, \\
& \mathbb{E}_{\nu_b^n} \Big[ \sup_{s \in [0,t]} \Big( \int_{0}^s  \sum_{x \in J_{n}} \nabla G \left(  \tfrac{x}{n}  \right) h_n(x) \big[  \overleftarrow{\psi}_x^{D_1}(\eta_r^n) - \overleftarrow{\psi}_x^{D_2}(\eta_r^n) \big] dr  \Big)^2  \Big] \leq  C_0 \frac{(D_2)^{\gamma-1} }{\Theta(n) } t \sum_{x \in J_{n}} [ \nabla G \left(  \tfrac{x}{n}  \right) h_n(x)]^2.
\end{align*}

\end{lem}
\begin{proof}
In what follows, we detail only the arguments for the first bound in last display, the proof for the second one being analogous. Observe that there exists $J \in \mathbb{N}$ such that $2^J D_1 \leq D_{2} <  2^{J+1} D_1$. Next, for every $i \in \{0, 1, \ldots, J\}$, we define $B_{i} = 2^{i} D_1$. Finally, we define $B_{J+1} =  D_{2}$. With this notation, the first expectation in the statement of Lemma \ref{lemmult0} can be rewritten as 
\begin{align} \label{explemmult0}
 \mathbb{E}_{\nu_b^n} \Big [     \sup_{s \in [0,t]} \Big( \sum_{i=0}^{J} \int_{0}^t  \sum_{x=1}^{n/2} \nabla G \left(  \tfrac{x}{n}  \right)  h_n(x)  \big[  \overrightarrow{\psi}_x^{B_{i}}(\eta_r^n) - \overrightarrow{\psi}_x^{B_{i+1}}(\eta_r^n) \big] dr  \Big)^2  \Big].
\end{align}
By applying arguments analogous to the ones used in the proof of Lemma \ref{lemeq413}, it is not hard to prove that for every $i \in \{0, 1, \ldots, J\}$, it holds 
\begin{equation} \label{explemmult0b}
\begin{split}
& \mathbb{E}_{\nu_b^n} \Big[ \sup_{s \in [0,t]} \Big( \int_{0}^s   \sum_{x=1}^{n/2} \nabla G \left(  \tfrac{x}{n}  \right) h_n(x) \big[  \overrightarrow{\psi}_x^{B_i}(\eta_r^n) - \overrightarrow{\psi}_x^{B_{i+1}}(\eta_r^n) \big] dr  \Big)^2  \Big] \leq   60 \kappa t \frac{(B_{i+1})^{\gamma-1} }{\Theta(n) } \sum_{x=1}^{n/2} [ \nabla G \left(  \tfrac{x}{n}  \right)  h_n(x) ]^2
, 
\end{split}
\end{equation}
where in last line $\kappa$ is exactly the constant given in the statement of Proposition \ref{prop48}.
The proof ends by combining \eqref{explemmult0} with \eqref{mink} and \eqref{explemmult0b}.
\end{proof}
For $j \in \{1, \ldots, N\}$ let $m_j:=\sum_{z=K_{j-1}}^{K_j-1}z a(z) $,  In particular, from \eqref{defsa} and \eqref{defma} we get
\begin{equation} \label{summj}
\sum_{j=1}^{n} |m_j| = \sum_{j=1}^{n} \Big| \sum_{z=K_{j-1}}^{K_j-1}z a(z) \Big| = \Big| \sum_{j=1}^{n} \sum_{z=K_{j-1}}^{K_j-1} z a(z)  \Big| \leq ´\Big| \sum_{z=1}^{\infty} z a(z) \Big| = |m_a| < \infty
\end{equation}
The last bound holds since $\gamma > 1$. Next we state an important lemma in order to obtain Proposition \ref{propgammales32}.
\begin{lem} \label{lem415b}
Assume that $G \in C^{\infty}([0,1])$ and $\gamma > 1$. Moreover, assume that either $\beta \geq 0$ and $2 r_a \leq 3$; or $\beta < 0$ and $2 r_a \leq 3 + \beta$. Then, if $A(\delta,n):= n^{(1-\gamma) \delta/2}$, it holds
\begin{align*}
& \mathbb{E}_{\nu_b^n} \Big [ \sup_{t \in [0,T]} \Big(  \int_{0}^t b_{n} \sum_{x=1}^{n/2} \nabla G \left(  \tfrac{x}{n}  \right) \sum_{j=1}^{N} m_j \sum_{r=j}^{N-1} \big[  \overrightarrow{\psi}_x^{K_r}(\eta_s^n) - \overrightarrow{\psi}_x^{K_{r+1}}(\eta_s^n) \big] ds  \Big)^2  \Big] \lesssim  A(\delta,n)   ; \\
& \mathbb{E}_{\nu_b^n} \Big[ \sup_{t \in [0,T]} \Big(  \int_{0}^t b_{ n} \sum_{x=K_N+n/2}^{n-1} \nabla G \left(  \tfrac{x}{n}  \right) \sum_{j=1}^{N} m_j \sum_{r=j}^{N-1} \big[  \overleftarrow{\psi}_x^{K_r}(\eta_s^n) - \overleftarrow{\psi}_x^{K_{r+1}}(\eta_s^n) \big] ds  \Big)^2  \Big] \lesssim  A(\delta,n)  ; \\
& \mathbb{E}_{\nu_b^n} \Big[ \sup_{t \in [0,T]} \Big(  \int_{0}^t b_{ n} \sum_{i=1}^N \sum_{x=K_{i-1}+n/2}^{K_i-1+n/2} \nabla G \left(  \tfrac{x}{n}  \right) \sum_{j=1}^{i-1} m_j \sum_{r=j}^{N-1} \big[  \overleftarrow{\psi}_x^{K_r}(\eta_s^n) - \overleftarrow{\psi}_x^{K_{r+1}}(\eta_s^n) \big] ds  \Big)^2  \Big] \lesssim  A(\delta,n) ; \\
& \mathbb{E}_{\nu_b^n} \Big[ \sup_{t \in [0,T]} \Big(  \int_{0}^t b_{n} \sum_{i=1}^N \sum_{x=K_{i-1}+n/2}^{K_i-1+n/2} \nabla G \left(  \tfrac{x}{n}  \right) \sum_{z=K_{i-1}}^{x-1-n/2}     z a(z) \sum_{r=i}^{N-1} \big[  \overleftarrow{\psi}_x^{K_r}(\eta_s^n) - \overleftarrow{\psi}_x^{K_{r+1}}(\eta_s^n) \big] ds  \Big)^2  \Big] \lesssim  A(\delta,n).
\end{align*}
\end{lem}
\begin{proof}
We prove here only the first upper bound, the proofs for the remaining ones being analogous. Applying \eqref{mink} for $m=N$, \eqref{summj}, \eqref{CSdiscexp} for $m=N$ and \eqref{CSdiscexp} for $m=N-j$, it is not very hard to prove that the first expectation in last display is bounded from above by
\begin{equation} \label{boundexpmult}
N^3 |m_a|^2  \sum_{r=1}^{N-1} \mathbb{E}_{\nu_b^n} \Big [       \sup_{t \in [0,T]} \Big(  \int_{0}^t  \sum_{x=1}^{n/2} \nabla G \left(  \tfrac{x}{n}  \right)  b_{n}  \big[  \overrightarrow{\psi}_x^{K_r}(\eta_s^n) - \overrightarrow{\psi}_x^{K_{r+1}}(\eta_s^n) \big] ds  \Big)^2  \Big].
\end{equation} 
Combining \eqref{boundexpmult} with an application of Lemma \ref{lemmult0} for $h_n \equiv b_n$, we get that the first expectation in the statement of the  lemma  is bounded from a constant, times
\begin{align*}
N^3 |m_a|^2  \sum_{r=1}^{N-1} \frac{(K_{r+1})^{\gamma-1} n^{2r_a-2}}{\Theta(n)} \lesssim \frac{(K_{r+1})^{\gamma-1} n^{2r_a-2}}{\Theta(n)} \lesssim
\begin{cases}
n^{(\gamma-1) \lambda_N + \gamma - 2 - 2 \beta_a }, \quad & \beta \geq 0; \\
n^{(\gamma-1) \lambda_N + \gamma - 2 - 2 \beta_a + \beta}, \quad & \beta \leq 0. 
\end{cases}
\end{align*}
In last bound we applied \eqref{timescale} and \eqref{defrs}. Thus, from \eqref{reclambdaj} we get
\begin{align*}
\begin{cases}
(\gamma-1) \lambda_N + \gamma - 2 - 2 \beta_a  < \gamma - 1 - \frac{\gamma-1}{2-\gamma} \hat{\delta} + \gamma - 2 - 2 \beta_a  = 2 r_a - 3 - \frac{\gamma-1}{2-\gamma} \hat{\delta}, \quad & \beta \geq 0; \\
(\gamma-1) \lambda_N + \gamma - 2 - 2 \beta_a + \beta < \gamma - 1 - \frac{\gamma-1}{2-\gamma} \hat{\delta} + \gamma - 2 - 2 \beta_a + \beta = 2 r_a - 3 - \beta - \frac{\gamma-1}{2-\gamma} \hat{\delta}, \quad & \beta \leq 0. 
\end{cases}
\end{align*}
Combining last display with the assumption that either $\beta \geq 0$ and $2 r_a \leq 3$; or $\beta < 0$ and $2 r_a \leq 3 + \beta$, we get that the first  expectation in the statement of Lemma \ref{lem415b} is bounded from a constant, times
$
n^{-\frac{\gamma-1}{2-\gamma} \hat{\delta}} =   n^{(1-\gamma) \delta/2}.
$
\end{proof}
We proceed by stating a final lemma that leads to the proof of Proposition \ref{propgammales32}.
\begin{lem} \label{lem416}
Let $G \in C^{\infty}([0,1])$ and $\gamma >1$. Assume that either $\beta \geq 0$ and $2 r_a \leq 3$; or $\beta < 0$ and $2 r_a \leq 3 + \beta$. If $A(\delta,n):= \max \big\{ n^{(1-\gamma) \delta/2}, \; n^{(\gamma-2) \delta/2} \big\} = \max \big\{ n^{(1-\gamma) \delta/2}, \; n^{- \hat{\delta}} \big\}  $, it holds
\begin{align*}
& \mathbb{E}_{\nu_b^n} \Big[ \Big( \sup_{t \in [0,T]}\int_{0}^t b_{n} \sum_{x=1}^{n/2} \nabla G \left(  \tfrac{x}{n}  \right) \sum_{z=1}^{K_N-1}       z a(z) \big[  \bar{\eta}_s^n(x) \bar{\eta}_s^n(x+z) - \overrightarrow{\psi}_x^{K_N}(\eta_s^n) \big] ds  \Big)^2  \Big] \lesssim  A(\delta,n) ; \\
& \mathbb{E}_{\nu_b^n} \Big[ \Big( \sup_{t \in [0,T]} \int_{0}^t b_{n} \sum_{x=K_n+1+n/2}^{n-1} \nabla G \left(  \tfrac{x}{n}  \right) \sum_{z=1}^{K_N-1}     z a(z) \big[ \bar{\eta}_s^n(x) \bar{\eta}_s^n(x-z) - \overleftarrow{\psi}_x^{K_N}(\eta_s^n) \big] ds  \Big)^2  \Big]\lesssim    A(\delta,n); \\
& \mathbb{E}_{\nu_b^n} \Big[ \Big( \sup_{t \in [0,T]} \int_{0}^t b_{n} \sum_{x=1+n/2}^{K_n+n/2} \nabla G \left(  \tfrac{x}{n}  \right) \sum_{z=1}^{x-1-n/2}       z a(z) \big[ \bar{\eta}_s^n(z) \bar{\eta}_s^n(x-z) - \overleftarrow{\psi}_x^{K_N}(\eta_s^n) \big] ds \Big)^2  \Big] \lesssim    A(\delta, n).
\end{align*}
\end{lem}
\begin{proof}
The proof ends by combining \eqref{CSdiscexp} for $m=2$ with Lemmas \ref{lem414} and \ref{lem415b}.
\end{proof}
We end this section by applying Lemma \ref{lem416} to obtain Proposition \ref{propgammales32}.  
\begin{proof} [Proof of Proposition \ref{propgammales32}]
Recall that $\delta$ is such that $0 < \delta < 3 - 2 r_a$, $K_N=n^{\lambda_N}$ and $1-\delta < \lambda_N < \frac{2-\delta}{2}$. Since $A_t^n(G)$ is equal to the sum of the terms in \eqref{Atngrad0a}, \eqref{Atngrad0b}, \eqref{Atngrad1} and \eqref{Atngrad2}, we get from Lemmas \ref{lem46}, \ref{lem45}, \ref{lem410}, \ref{lem416} and an application of \eqref{CSdiscexp} for $m=4$ that the expectation in the statement of Proposition \ref{propgammales32} is bounded from above by a constant times
\begin{align*}
\frac{1}{n}+ \frac{n^{2 r_a - 2}}{K_N^{2 \gamma - 1}} + \frac{n^{2 r_a - 2}}{K_N} +   n^{(1-\gamma) \delta/2} + n^{(\gamma-2) \delta/2} \lesssim \frac{1}{n} + \frac{n^{1-\delta}}{n^{\lambda_N}} +  n^{(1-\gamma) \delta/2} + n^{(\gamma-2) \delta/2},
\end{align*} 
which goes to zero as $n \rightarrow \infty$, due to $1 < \gamma < 2$. This ends the proof.
\end{proof}

\subsection{Proof of Proposition \ref{tightAtnG} } \label{sectightAtnG}

In order to obtain Proposition \ref{tightAtnG}, we begin by stating the following result.
\begin{lem} \label{boundvaraux}
Assume that $G \in C^{\infty}([0,1])$, $\gamma \in (0,2)$ and $\beta \geq 0$. Let $\lambda \in (0,1)$, $n \geq 2$, $h_n: \Lambda_n \rightarrow \mathbb{R}$, $J_n \subset \{1+n/2, \ldots, n-1\}$ and $\lambda \in (0,1)$. Then, for every $ t \in [0, T]$ it holds
\begin{align}
 & \mathbb{E}_{\nu_b^n} \Big[ \sup_{s \in [0,t]} \Big( \int_{0}^s  \sum_{x=1}^{n/2} \nabla G \left(  \tfrac{x}{n}  \right) h_n(x)   \overrightarrow{\psi}_x^{n^{\lambda}}(\eta_r^n)  dr  \Big)^2  \Big] \leq  \frac{\tilde{C}_0  t^{\frac{2 \gamma-1}{\gamma}}}{n}  \sum_{x=1}^{n/2} [ \nabla G \left(  \tfrac{x}{n}  \right) h_n(x)]^2, \label{boundaux1} \\
 & \mathbb{E}_{\nu_b^n} \Big[ \sup_{s \in [0,t]} \Big( \int_{0}^s  \sum_{x \in J_{n}} \nabla G \left(  \tfrac{x}{n}  \right) h_n(x)   \overleftarrow{\psi}_x^{n^{\lambda}}(\eta_r^n)  dr  \Big)^2  \Big] \leq  \frac{\tilde{C}_0 t^{\frac{2 \gamma-1}{\gamma}}}{n} \sum_{x \in J_{n}} [ \nabla G \left(  \tfrac{x}{n}  \right) h_n(x)]^2, \nonumber
\end{align}
for some constant $\tilde{C}_0>0$ depending on $T$ and independent of $n$, $t$, $G$ and $\lambda$.
\end{lem}
\begin{proof}
We detail only the proof for \eqref{boundaux1}, since the remaining bound can be treated in the same way.  Denoting $\zeta:= t (9 T)^{-1}$, we get that $t = (9T  ) \zeta$. Therefore, the expression in \eqref{boundaux1} is equivalent to
\begin{equation} \label{boundaux2}
\mathbb{E}_{\nu_b^n} \Big[ \sup_{s \in [0, \zeta]} \Big( \int_{0}^s  \sum_{x=1}^{n/2} \nabla G \left(  \tfrac{x}{n}  \right) h_n(x)   \overrightarrow{\psi}_x^{n^{\lambda}}(\eta_r^n)  dr  \Big)^2  \Big] \leq  \tilde{C}_0 (9 T )^{\frac{2 \gamma-1}{\gamma}} \frac{\zeta^{\frac{2 \gamma-1}{\gamma}}}{n}     \sum_{x=1}^{n/2} [ \nabla G \left(  \tfrac{x}{n}  \right) h_n(x)]^2,
\end{equation}
for every $\zeta \in [0, 1/9]$. An application of Lemma \ref{lemauxtight} for $\hat{K}=n^{\lambda}$ leads to
\begin{align*}
\forall t \in [0, T], \quad \mathbb{E}_{\nu_b^n} \Big[ \sup_{s \in [0,t]} \Big(  \int_{0}^s
\sum_{x=1}^{n/2}  \nabla G \left(  \tfrac{x}{n}  \right) h_{n}(x)  \overrightarrow{\psi}_x^{n^{\lambda}}(\eta_r^n)  dr  \Big)^2  \Big] \leq \frac{6}{n^{\lambda}} t^2 \sum_{x=1}^{n/2}[\nabla G \left(  \tfrac{x}{n}  \right) h_{n}(x)]^2.
\end{align*}
Therefore, for every $\zeta \in [0, 1/9]$, it holds
\begin{equation} \label{boundaux3}
\mathbb{E}_{\nu_b^n} \Big[ \sup_{s \in [0,  \zeta]} \Big( \int_{0}^s  \sum_{x=1}^{n/2} \nabla G \left(  \tfrac{x}{n}  \right) h_n(x)   \overrightarrow{\psi}_x^{n^{\lambda}}(\eta_r^n)  dr  \Big)^2  \Big] \leq 6 (9T)^2 \frac{\zeta^2}{n^{\lambda}}    \sum_{x=1}^{n/2} [ \nabla G \left(  \tfrac{x}{n}  \right) h_n(x)]^2.
\end{equation}
Next, observe that $\zeta^2 / n^{\lambda} \leq \zeta^{\frac{2 \gamma-1}{\gamma}}/n$ iff $n^{1-\lambda} \leq \zeta^{-1/\gamma}$, which is equivalent to $0 \leq \zeta \leq n^{\frac{\lambda-1}{\gamma}}$. Thus, \eqref{boundaux2} can be obtained directly from an application of \eqref{boundaux3} for $\hat{K}=n^{\lambda}$, when $0 \leq \zeta \leq n^{\frac{\lambda-1}{\gamma}}$. In this case, \eqref{boundaux2} holds for $\tilde{C}_0=6 (9T)^2 [ (9 T )^{\frac{2 \gamma-1}{\gamma}} ]^{-1}=6 (9T)^{1/\gamma}$.

On the other hand, for every $\varepsilon \in (0, 1/3)$, an application of Lemma \ref{lemauxtight} for $\hat{K}= \varepsilon n$ leads to
\begin{align*}
\forall t \in [0, T], \quad \mathbb{E}_{\nu_b^n} \Big[ \sup_{s \in [0,t]} \Big(  \int_{0}^s
\sum_{x=1}^{n/2}  \nabla G \left(  \tfrac{x}{n}  \right) h_{n}(x)  \overrightarrow{\psi}_x^{\varepsilon n}(\eta_r^n)  dr  \Big)^2  \Big] \leq \frac{6}{\varepsilon n} t^2 \sum_{x=1}^{n/2}[\nabla G \left(  \tfrac{x}{n}  \right) h_{n}(x)]^2.
\end{align*}
Therefore, for every $\zeta \in [0, 1/9]$, it holds
\begin{align} \label{boundaux4}
 \mathbb{E}_{\nu_b^n} \Big[ \sup_{s \in [0, \zeta]} \Big(  \int_{0}^s
\sum_{x=1}^{n/2}  \nabla G \left(  \tfrac{x}{n}  \right) h_{n}(x)  \overrightarrow{\psi}_x^{\varepsilon n}(\eta_r^n)  dr  \Big)^2  \Big] \leq 6 (9 T)^2 \frac{\zeta^2}{\varepsilon n}   \sum_{x=1}^{n/2}[\nabla G \left(  \tfrac{x}{n}  \right) h_{n}(x)]^2.
\end{align}
If $n^{\lambda} \leq \varepsilon n$, an application of Lemma \ref{lemmult0} for $D_1=n^{\lambda}$ and $D_2=\varepsilon n$ leads to 
\begin{align*}
\mathbb{E}_{\nu_b^n} \Big[ \sup_{s \in [0,t]} \Big( \int_{0}^s  \sum_{x=1}^{n/2} \nabla G \left(  \tfrac{x}{n}  \right) h_n(x) \big[  \overrightarrow{\psi}_x^{n^{\lambda}}(\eta_r^n) - \overrightarrow{\psi}_x^{\varepsilon n}(\eta_r^n) \big] dr  \Big)^2  \Big] \leq C_0 \frac{(\varepsilon n)^{\gamma-1} }{\Theta(n) }    t  \sum_{x=1}^{n/2} [ \nabla G \left(  \tfrac{x}{n}  \right) h_n(x)]^2,
\end{align*}
for every $t \in [0, T]$. In last line, $C_0$ is the same constant given in the statement of Lemma \ref{lemmult0}. Since $\beta \geq 0$, we get from \eqref{timescale} that $\Theta(n)=n^{\gamma}$. Thus, for any $t \in [0, T]$, it holds
\begin{align*}
\mathbb{E}_{\nu_b^n} \Big[ \sup_{s \in [0,t]} \Big( \int_{0}^s  \sum_{x=1}^{n/2} \nabla G \left(  \tfrac{x}{n}  \right) h_n(x) \big[  \overrightarrow{\psi}_x^{n^{\lambda}}(\eta_r^n) - \overrightarrow{\psi}_x^{\varepsilon n}(\eta_r^n) \big] dr  \Big)^2  \Big] \leq C_0 \frac{\varepsilon^{\gamma}  }{ \varepsilon n }    t  \sum_{x=1}^{n/2} [ \nabla G \left(  \tfrac{x}{n}  \right) h_n(x)]^2.
\end{align*}
From last display, we get that for every $\zeta \in [0, 1/9]$, it holds
\begin{equation*} 
\mathbb{E}_{\nu_b^n} \Big[ \sup_{s \in [0, \zeta]} \Big( \int_{0}^s  \sum_{x=1}^{n/2} \nabla G \left(  \tfrac{x}{n}  \right) h_n(x) \big[  \overrightarrow{\psi}_x^{n^{\lambda}}(\eta_r^n) - \overrightarrow{\psi}_x^{\varepsilon n}(\eta_r^n) \big] dr  \Big)^2  \Big] \leq C_0 (9T ) \frac{\varepsilon^{\gamma} \zeta }{ \varepsilon n }  \sum_{x=1}^{n/2} [ \nabla G \left(  \tfrac{x}{n}  \right) h_n(x)]^2. 
\end{equation*}
Combining last display with \eqref{boundaux4} and an application of \eqref{CSdiscexp} for $m=2$, we conclude that 
\begin{align*}
\mathbb{E}_{\nu_b^n} \Big[ \sup_{s \in [0,t]} \Big(  \int_{0}^s
\sum_{x=1}^{n/2}  \nabla G \left(  \tfrac{x}{n}  \right) h_n(x)  \overrightarrow{\psi}_x^{n^{\lambda}}(\eta_r^n)  dr  \Big)^2  \Big] \leq 18 T (54T + C_0)  \frac{\varepsilon^{\gamma} \zeta + \zeta^2}{\varepsilon n} \sum_{x=1}^{n/2}[\nabla G \left(  \tfrac{x}{n}  \right) h_n(x)]^2,
\end{align*}
if $n^{\lambda} \leq \varepsilon n$. Next, we observe that $\zeta^{1/\gamma} \leq 9^{-1/\gamma} \leq 9^{-1/2} =1/3$, therefore it is always possible to choose $\varepsilon=\zeta^{1/\gamma}$. For this choice, we get $n^{\lambda} \leq \varepsilon n$ iff $\zeta^{1/\gamma} \geq n^{\lambda-1}$, which is equivalent to $  n^{\frac{\lambda-1}{\gamma}} < \zeta \leq 1/9$. In this case, $(\varepsilon^{\gamma} \zeta + \zeta^2)(\varepsilon n)^{-1}= (2\zeta^2)(\varepsilon n)^{-1} =  2 \zeta^{\frac{2 \gamma-1}{\gamma}} /n$. In particular, \eqref{boundaux2} holds for $\tilde{C}_0=36 T (54 T + C_0) [ (9 T )^{\frac{2 \gamma-1}{\gamma}} ]^{-1}$. This ends the proof.
\end{proof}
For $r_a=3/2$, we get from \eqref{comprsgam} that $\gamma \geq 3/2$, which leads to $\gamma + 1 > 2 \gamma - 1 \geq 2$. Moreover, keeping \eqref{reclambdaj} and Lemma \ref{lem416} in mind, given any $\delta \in (0, 1/2)$, we can define $N=N(\delta) \in \mathbb{N}$ by
\begin{align*}
N=N(\delta):= \min \Big\{ j \in \mathbb{N}: \Big( \frac{2 \gamma - 1}{\gamma +1} \Big)^j < \frac{\delta}{2 - \delta} \Big\}, 
\end{align*}
Applying \eqref{reclambdaj} with $\hat{\delta}=(2-\gamma) \delta /2$, we have that $1/2 < 1-\delta < \lambda_N < \frac{2-\delta}{2}$, since $\lambda_N$ is given by
\begin{equation} \label{deflambdaN}
\lambda_N := \Big( 1 - \frac{\delta}{2} \Big) \Big[ 1 - \Big( \frac{2 \gamma - 1}{\gamma+1} \Big)^N \Big].
\end{equation}
Without loss of generality, we fix $\delta = 1/4$ in the remainder of this section. Thus, $\lambda_N$ is also fixed.

In order to obtain Proposition \ref{tightAtnG}, we introduce the auxiliary process $\tilde{A}^{n}(G)$: 
\begin{align} \label{deftilAtnf}
\tilde{A}_t^{n}(G):=  \frac{\sqrt{n}}{\sqrt{n-1}} \int_0^t  \Big[ \sum_{x=1}^{n/2} \nabla G \left(  \tfrac{x}{n}  \right)  \overrightarrow{\psi}_x^{n^{\lambda_N}}(\eta_s^n) + \sum_{x=n/2+1}^{n-1} \nabla G \left(  \tfrac{x}{n}  \right) \overleftarrow{\psi}_x^{n^{\lambda_N}}(\eta_s^n)  \Big]ds, \quad t \in [0,T],
\end{align}
for any $n \geq 2$ and $G \in C^{\infty}([0,1])$. The previous display is motivated by the following couple of results.
\begin{prop} \label{propaprox}
Let $G \in C^{\infty}([0,1])$. Assume that $2 r_a =3$ and $\beta \geq 0$. Then 
\begin{equation} \label{claimtight}
\lim_{n \rightarrow \infty} \mathbb{E}_{\nu_b^n} \big[ \sup_{t \in [0,T]} \big(  A_t^n(G) - 2 \alpha_a m_a  \tilde{A}_t^{n}(G) \big)^2 \big] =0.
\end{equation}
\end{prop}
\begin{prop} \label{proptightaux}
Let $G \in C^{\infty}([0,1])$. Assume that $2 r_a =3$ and $\beta \geq 0$. Then the family of processes $\{\tilde{A}_t^{n}(G): n \geq 2 \}$ is tight.
\end{prop}
In particular, Proposition \ref{tightAtnG} is a direct consequence of Propositions \ref{propaprox} and \ref{proptightaux}. Now we present the proof of Proposition \ref{propaprox}.
\begin{proof} [Proof of Proposition \ref{propaprox}] 
Recall from Section \ref{secconvAtnG} that $A_t^n(G)$ is equal to the sum of the terms in \eqref{Atngrad0a}, \eqref{Atngrad0b}, \eqref{Atngrad1} and \eqref{Atngrad2}, thus $A_t^n(G)$ can be estimated by applying Lemmas \ref{lem46}, \ref{lem45}, \ref{lem410} and \ref{lem416}. Due to our choice for $\lambda_N$, we get $2r_a-2 = 3 - 2= 1 < 2 \lambda_N =(3-1) \lambda_N =[2(3/2) -1] \lambda_N \leq (2\gamma-1) \lambda_N$ and the terms corresponding to Lemmas \ref{lem46}, \ref{lem45} and \ref{lem416} all vanish in $L^2(\mathbb{P}_{\nu_b^n})$, as $n \rightarrow \infty$. 

Thus, the proof is a direct consequence of Lemmas \ref{lem46}, \ref{lem45}, \ref{lem416}, \eqref{CSdiscexp} and Lemma \ref{boundvaraux}.
\end{proof} 
We end this section with the following observation: by proceeding in the same way as it was done in the end of Section 4.4.4 of \cite{jarafluc}, we can conclude that Proposition \ref{proptightaux} is a direct consequence of Lemma \ref{boundvaraux}.

\section{Characterization of the limit points} \label{seccharac}

We apply \eqref{decomp} to characterize the limit points of $( \mathcal{Y}_t^n)_{n \geq 1}$ as random elements satisfying the  conditions stated in Definition \ref{spde} or Definition \ref{spdefbe}, depending on  whether (H1) or (H2) is satisfied.
More exactly, let $(\alpha, \beta, \gamma) \in (0, \infty) \times \mathbb{R} \times (0,  2)$. From Propositions \ref{propmit},  \ref{tightfluc1}, \ref{tightmartterm}, \ref{tightfluc2a}, \ref{convL2error} \ref{conv0tight}, \ref{convAtnG} and \ref{tightAtnG}, there exists a subsequence $(n_j)_{j \geq 1}$ such that all the sequences $(\mathcal{Y}_0^{n_j})_{j \geq 1}$, $(\mathcal{M}_t^{n_j})_{j \geq 1}$, $(\mathcal{I}_t^{n_j})_{j \geq 1}$, $(\mathcal{E}_t^{n_j})_{j \geq 1}$ and $(A_t^{n_j})_{j \geq 1}$  converge, as $j\to+\infty$, in distribution, with respect to the Skorohod topology of $ \mcb{D} ( [0,T],  \mcb {S}_{\beta,\gamma}' )$. Denote the respective limit points by $\mathcal{Y}_0$, $\mathcal{M}_t$, $\mathcal{I}_t$, $\mathcal{E}_t$ and $A_t$. From \eqref{decomp}, $(\mathcal{Y}_t^{n_j})_{j \geq 1}$ converges in distribution to some limit point $\mathcal{Y}_t$ which is stationary, due to Remark \ref{remstat}. 
Thus, for every $G \in \mcb {S}_{\beta,\gamma}$, from \eqref{decomp} it holds
\begin{align} \label{decomp2}
	\mathcal{M}_t(G) = \mathcal{Y}_t(G) - \mathcal{Y}_0(G) - \mathcal{I}_t(G) +  A_t(G).
\end{align}
In last line, we applied \eqref{intprocess}, Corollary \ref{corL2ab} and Proposition \ref{aproxL2gen}. Also, we note that $\mathcal{E}_t(G)=0$, due to Propositions \ref{convL2error} and \ref{conv0tight}. From \eqref{MthnL2}, $\mathcal{M}_t(G)$ is the limit in distribution of the  sequence $\big(\mathcal{M}_t^{n_j}(G) \big)_{j \geq 1}$ of martingales. In the next subsection we prove that this sequence is uniformly integrable, and as a consequence the limit  $\mathcal{M}_t(G)$ is also a martingale.
It also remains to characterize the limit   $A_t(G)$ under (H1) and  (H2).

\subsection{Characterization of limiting martingale}

\begin{lem} \label{lemconvmartterm2}
Let $(\beta, \gamma) \in R_0$ and $G \in \mcb {S}_{\beta,\gamma}$. Then
\begin{align} \label{expconvmartterm2}
 \lim_{n \rightarrow \infty}  \mathbb{E}_{\nu_b^n} \big[\sup_{t \in [0,T]} \big(  \langle \mathcal M^n (G)  \rangle_t - \mathbb{E}_{\nu_b^n} [   \langle \mathcal M^n (G)  \rangle_t ] \big)^2 \big] =0. 
\end{align}
\end{lem} 
\begin{proof}
Combining Fubini's Theorem with \eqref{quadvar1} it holds
\begin{equation} \label{quadvarcent}
\begin{split}
\langle \mathcal M^n (G)  \rangle_t - \mathbb{E}_{\nu^n_b} [  \langle \mathcal M^n (G)  \rangle_t ]=& \frac{\Theta(n)}{2(n-1)} \int_0^t \sum_{x, y } [ G\left( \tfrac{y}{n}\right) -  G\left( \tfrac{x}{n}\right) ]^2  Z^n_{x,y}(\eta_s^n) ds \\
+ & \frac{\alpha \Theta(n)}{n^{\beta} (n-1)} (1 - 2b) \int_0^t \sum_{x } G^2\left( \tfrac{x}{n}\right) \bar{\eta}_{s}^{n}(x)   [ r_n^{+} \left( \tfrac{x}{n}\right) + r_n^{-} \left( \tfrac{x}{n}\right)] ds \\
+ & \frac{\alpha_a \Theta(n)}{n^{\beta_a} (n-1)} \frac{c_{+} - c_{-}}{c_{+} + c_{-}} \int_0^t  \sum_{x } G^2 \left( \tfrac{x}{n}\right) \bar{\eta}_s^n(x)  [ r_n^{+} \left( \tfrac{x}{n}\right) - r_n^{-} \left( \tfrac{x}{n}\right)] ds,
\end{split}
\end{equation}
where for every $n \geq 2$ and every $x,y \in \mathbb{Z}$, $Z^n_{x,y}(\eta_s^n)$ is given by
\begin{align*}
Z^n_{x,y}(\eta_s^n):= s(y-x) \big\{ [ \eta_s^n(y) - \eta_s^n(x)]^2 - 2 \chi(b) \big\} - \alpha_a n^{-\beta_a} a(y-x) [\bar{\eta}_{s}^{n}(y) - \bar{\eta}_{s}^{n}(x)].
\end{align*}
In particular, from \eqref{defsa} and \eqref{condrate}, for every $n \geq 2$ and every $x,y,w \in \Lambda_n$, it holds
\begin{equation} \label{boundexpZn}
\mathbb{E}_{\nu_b^n} \big[ | Z^n_{x,y}(\eta_s^n) Z^n_{x,w}(\eta_s^n) |  \big] \leq (3+\alpha_a)^2 s(y-x) s(w-x).
\end{equation}
Combining \eqref{quadvarcent} with an application of \eqref{CSdiscexp} for $m=3$ and the fact that $n \leq 2(n-1)$ for every $n \geq 2$, we have that the expectation in \eqref{expconvmartterm2} is bounded from above by a constant times
\begin{align}
& [ \Theta(n) ]^2 n^{-2} \mathbb{E}_{\nu_b^n} \Big[ \sup_{t \in [0,T]} \Big(  \int_0^t  \sum_{ x,y } [ G( \tfrac{y}{n} ) - G( \tfrac{x}{n} )]^2 Z^n_{x,y}(\eta_s^n) ds \Big)^2 \Big] \label{expconvmartterm2a} \\
+& \mathbb{E}_{\nu_b^n} \Big[ \sup_{t \in [0,T]} \Big( \Theta(n)  n^{-1- \beta} \int_0^t   \sum_{ x } [r_n^{+}(\tfrac{x}{n})+r_n^{-}(\tfrac{x}{n})]  G^2( \tfrac{x}{n})  \bar{\eta}_s^n(x) ds \Big)^2 \Big] \label{expconvmartterm2b} \\
+ & \mathbb{E}_{\nu_b^n} \Big[ \sup_{t \in [0,T]} \Big( \Theta(n)  n^{-1- \beta_a} \int_0^t   \sum_{ x } [r_n^{+}(\tfrac{x}{n}) - r_n^{-}(\tfrac{x}{n})]  G^2( \tfrac{x}{n})  \bar{\eta}_s^n(x) ds \Big)^2 \Big]. \label{expconvmartterm2c}
\end{align}

Combining \eqref{CSdiscexp} and \eqref{timescale} with an application of Proposition \ref{lem61stefano} for $q=2$ and $c_n=\Theta(n)  n^{-1- \beta}$, resp. \eqref{boundrnpm} and an application of Proposition \ref{lem61stefano} for $q=2$ and $c_n=\Theta(n)  n^{-1- \beta_a}$,  the term in \eqref{expconvmartterm2b}, resp. \eqref{expconvmartterm2c} vanishes, as $n \rightarrow \infty$. 
Combining \eqref{CSFub} with \eqref{boundexpZn} and the fact that $\mathbb{E}_{\nu_b^n} [  Z^n_{x,y}(\eta_s^n)]=0=\mathbb{E}_{\nu_b^n} [  Z^n_{x,y}(\eta_s^n) Z_{w,r}(\eta_s^n)]=0$ for disjoint sets $\{x,y\}$, $\{w,r\}$, we conclude that the expectation in \eqref{expconvmartterm2a} is bounded from above by
\begin{equation} \label{boundexpconvmartterm2a}
[T(3+\alpha_a)]^2 \Big\{ \sum_{ x,y } [ G( \tfrac{y}{n} ) - G( \tfrac{x}{n} )]^4 s^2(y-x) +  \sum_{ x } \Big[ \sum_{y } s(y-x) [ G( \tfrac{y}{n} ) - G( \tfrac{x}{n} )]^2 \Big]^2 \Big\}.
\end{equation}
Now we extend $G$ to some function $\hat{G} \in C_c^{\infty}(\mathbb{R})$. Moreover, it holds $[ \Theta(n) ]^2 n^{-2} \leq n^{2\gamma-2}$ for any $\beta \in \mathbb{R}$, due to \eqref{timescale}. Thus, the term in \eqref{expconvmartterm2a} is bounded from above by $[T(3+\alpha_a)]^2 B_n(\hat{G})$, where
\begin{align*}
B_n(H):= n^{2\gamma-2} \Big\{ \sum_{ x,y \in \mathbb{Z} } [ H( \tfrac{y}{n} ) - H( \tfrac{x}{n} )]^4 s^2(y-x) +  \sum_{ x \in \mathbb{Z} } \Big[ \sum_{y \in \mathbb{Z}} s(y-x) [ H( \tfrac{y}{n} ) - H( \tfrac{x}{n} )]^2 \Big]^2 \Big\},
\end{align*} 
for every $H \in C_c^{\infty}(\mathbb{R})$. Last expression is exactly equal to (4.1) of \cite{jarafluc}. Doing similar computations to those of  Appendix D of \cite{jarafluc}, we conclude  $\forall \gamma \in (0,2)$ and $\forall H \in C_c^{\infty}(\mathbb{R})$, that
$
 \quad \lim_{n \rightarrow \infty} B_n(H)=0,
$
and the proof ends.
\end{proof}
Now we can state a corollary which is useful in the remainder of this section.
\begin{cor} \label{corquadvar}
Let $(\alpha, \beta, \gamma) \in (0, \infty) \times \mathbb{R} \times (0,  2)$ and $G \in \mcb {S}_{\beta,\gamma}$. Then 
\begin{enumerate}
	\item 
$\big( \langle \mathcal M^{n} (G)  \rangle_t \big)_{n \geq 2}$ converges in distribution to $2 \chi(b) t \mathcal{P}^{\gamma}_{\alpha, \beta} G$, as $n \rightarrow \infty$, for any $t \in [0,T]$; 
\item 
there exists $C_1(G,T) >0$ depending only on $G$ and $T$ such that $\mathbb{E}_{\nu_b^n} \big[ \big(   \langle  \mathcal{M}^n(G) \rangle_t  \big)^2 \big] \leq C_1(G,T)$, for any $t \in [0,T]$ and any $n \geq 2$.
\end{enumerate}
\end{cor}
\begin{proof}
\begin{enumerate}
	
	The proof of item \textit{(1)} is a direct consequence of Chebyshev's and Markov's inequalities, Lemma \ref{expconvmartterm2} and \eqref{limexpquadvar}. 
For \textit{(2)} we note that from Lemma \ref{lemconvmartterm2}, there exists a constant $\tilde{C}_0(G) >0$ depending only on $G$ such that
	\begin{align} \label{bndvarquadvar}
		\forall n \geq 2, \quad \mathbb{E}_{\nu_b^n} \big[\sup_{r \in [0,T]} \big(  \langle \mathcal M^n (G)  \rangle_r - \mathbb{E}_{\nu_b^n} [   \langle \mathcal M^n (G)  \rangle_r ] \big)^2 \big] \leq \tilde{C}_0(G).
	\end{align}
The proof ends by combining \eqref{CSdisc}, \eqref{MthnL2} and \eqref{bndvarquadvar}.
\end{enumerate}	
\end{proof}
Now we present another useful lemma.
\begin{lem} \label{lem4mom}
Let $(\beta, \gamma) \in R_0$ and $G \in \mcb {S}_{\beta,\gamma}$. Then there exists $C_2(G,T) >0$ depending only on $G$ and $T$ such that $\mathbb{E}_{\nu_b^n} \big[ \big(   \langle  \mathcal{M}^n(G) \rangle_t  \big)^4 \big] \leq C_2(G,T)$, for any $t \in [0,T]$ and any $n \geq 2$. 
\end{lem}
\begin{proof}
The proof is a direct consequence of Lemma 3 in \cite{Dittrich91}, \eqref{MthnL2} and \eqref{defMtngfluc}.
\end{proof}

Let $\mathcal{N}_t^{n_j}(G):=[\mathcal{M}_t^{n_j}(G)]^2 - \langle \mathcal M^{n_j} (G)  \rangle_t$ for every $t \in [0,T]$ and $j \geq 1$, from Corollary \ref{corquadvar} we get that $\mathcal{N}_t^{n_j}(G)$ converges in distribution to $\mathcal{N}_t(G)$ as $j \rightarrow \infty$, where $\mathcal{N}_t(G)$ is given by  
\begin{align*} 
	\mathcal{N}_t(G):= [ \mathcal{M}_t(G)]^2 - 2 \chi(b) t \mathcal{P}^{\gamma}_{\alpha, \beta} G.
\end{align*}
Combining an application of \eqref{CSdisc} for $m=2$ with Corollary \ref{corquadvar} and Lemma \ref{lem4mom}, we get that
\begin{align*}
	\sup_{t \in [0,T], j \geq 1} \mathbb{E}_{\nu_b^{n_j}} \big[ \big( \mathcal{N}_t^{n_j}(G) \big)^2 \big]  = \sup_{t \in [0,T], j \geq 1} \mathbb{E}_{\nu_b^{n_j}} \big[ \big( [ \mathcal{M}_t^{n_j}(G)]^2 - \langle  \mathcal{M}^{n_j}(G) \rangle_t  \big)^2 \big]  < \infty. 
\end{align*}
From last display, $\mathcal{N}_t(G)$ is the limit in distribution of $\big(\mathcal{N}_t^{n_j}(G) \big)_{j \geq 1}$, an uniformly integrable sequence  of martingales. Therefore $\mathcal{N}_t(G)$ is itself a martingale.

\subsection{Characterization of the limit $A_t(G)$}
{If $(\beta,\gamma)\in \mathbb R\times (0,2)$ are such that    (H1) holds, then from Propositions \ref{convAtnG} and \ref{conv0tight} we get that $A_t(G)=0$, for any $G\in C^\infty([0,1])$.} Therefore, we conclude that $\mathcal{Y}_\cdot$ satisfies the conditions stated in Definition \ref{defspde}, ending the proof of Theorem \ref{clt}. 

Now assume   (H2).
Keeping \eqref{defprocA} in mind, for every $n \geq 2$ and $\varepsilon \in (0,1/3)$, we define the auxiliary process $A^{n,\varepsilon}(G)$ by
\begin{align*}
A_t^{n,\varepsilon}(G):= \int_0^t \Big\{ \frac{1}{n} \sum_{x=1}^{n/2} \nabla G \left(  \tfrac{x}{n}  \right)  \big[ \mathcal{Y}_t^n * \iota_{\varepsilon}^{+} \left( \tfrac{x}{n} \right) \big]^2 + \frac{1}{n} \sum_{x=n/2+1}^{n-1} \nabla G \left(  \tfrac{x}{n}  \right) \big[ \mathcal{Y}_t^n * \iota_{\varepsilon}^{-} \left( \tfrac{x}{n} \right) \big]^2 \Big\}ds, \quad t \in [0,T],
\end{align*}
where $\iota_{\varepsilon}^{+}$ and $\iota_{\varepsilon}^{-}$ are given by \eqref{defiota}. Analogously to (5.5) in \cite{jarafluc}, we claim that
\begin{equation} \label{limsupcharac}
\varlimsup_{\varepsilon \rightarrow 0^+} \varlimsup_{n \rightarrow \infty} \mathbb{E}_{\nu_b^n} \big[ \sup_{t \in [0,T]} \big(     A_t^n(G) - 2 \alpha_a m_a  A_t^{n,\varepsilon}(G) \big)^2  \big] =0.
\end{equation}
Recalling \eqref{limsupcharac} and  (H2), we get that $A_t^n(G)$ is asymptotically equivalent to $2 \alpha_a m_a \tilde{A}_t^n(G)$, where $\tilde{A}_t^n(G)$ is given by \eqref{deftilAtnf}. 
Applying Lemma \ref{boundvaraux} for $t=T$ and $\lambda=\lambda_N$ (given by \eqref{deflambdaN}), we get from \eqref{deftilAtnf} and \eqref{CSdiscexp} 
\begin{equation} \label{bound1charac}
\forall n \geq 2, \quad \mathbb{E}_{\nu_b^n} \Big[ \sup_{t \in [0,T]} \Big(\tilde{A}_t^{n}(G) - \frac{\sqrt{n-1}}{\sqrt{n}} \tilde{A}_t^{n}(G)   \Big)^2 \Big] \leq   \frac{ \tilde{C}_0 T^{\frac{2 \gamma-1}{\gamma}}}{n^2} \sum_{x } [ \nabla G \left(  \tfrac{x}{n}  \right) ]^2.
\end{equation}
For any $\varepsilon \in (0, 1/3)$ and $n \geq 2$ such that $n^{\lambda_N} \leq \varepsilon n$, applying Lemma \ref{lemmult0} for $t=T$, $D_1 =n^{\lambda_N}$ and $D_2 = \varepsilon n$, we have from \eqref{CSdiscexp}, \eqref{timescale} and the fact that $\beta \geq 0$ under  (H2) that
\begin{equation} \label{bound2charac}
\mathbb{E}_{\nu_b^n} \Big[ \sup_{t \in [0,T]} \Big( \frac{\sqrt{n-1}}{\sqrt{n}} \tilde{A}_t^n(G) -  \tilde{A}_t^{n,\varepsilon}(G)  \Big)^2  \Big] \leq   \varepsilon^{\gamma-1} \frac{2 C_0 T}{n} \sum_{x } [ \nabla G \left(  \tfrac{x}{n}  \right) ]^2.
\end{equation}
In last line, we defined, for every $\varepsilon \in (0, 1/3)$ and $n \geq 2$, 
\begin{align*}
\tilde{A}_t^{n,\varepsilon}(G):= \int_{0}^t\Big[ \sum_{x=1}^{n/2} \nabla G \left(  \tfrac{x}{n}  \right)  \overrightarrow{\psi}_x^{\varepsilon n}(\eta_s^n) + \sum_{x=n/2+1}^{n-1} \nabla G \left(  \tfrac{x}{n}  \right) \overleftarrow{\psi}_x^{\varepsilon n}(\eta_s^n)  \Big]ds, \quad t \in [0,T].
\end{align*}
From \eqref{eqdensfie}, \eqref{defiota} and \eqref{medemp}, we get
\begin{equation} \label{replproc}
\frac{\varepsilon n - 1}{\varepsilon n - \varepsilon} \tilde{A}_t^{n, \varepsilon}(G) - A_t^{n,\varepsilon}(G) =  - \frac{b(1-b)t}{\varepsilon n- \varepsilon} \sum_{x=1}^{n-1}  \nabla G \left(  \tfrac{x}{n}  \right) .
\end{equation}
Combining \eqref{CSdiscexp} with an application of Lemma \ref{lemauxtight} for $\hat{K}=\varepsilon n$ and $t =T$, we get
\begin{equation} \label{boundvareps}
\forall n \geq 2, \; \forall \varepsilon \in (0, 1/3), \quad \mathbb{E}_{\nu_b^n} \Big[ \sup_{t \in [0,T]} \Big(      \tilde{A}_t^{n,\varepsilon}(G) - \frac{\varepsilon n - 1}{\varepsilon n - \varepsilon} \tilde{A}_t^{n, \varepsilon}(G) \Big)^2  \Big]    \leq  \frac{48 T^2}{(\varepsilon n)^3}  \sum_{x } [ \nabla G \left(  \tfrac{x}{n}  \right) ]^2.
\end{equation}
From the Mean Value Theorem,  for any $h \in C^1([0,1])$ such that $\int_0^1 h(u) du =0$, it holds 
\begin{align*}
\forall n \geq 2, \quad \Big| \frac{1}{n} \sum_{j=0}^{n-1} h  ( \tfrac{j}{n} ) \Big| \leq \frac{2}{n} \sup_{u \in [0,1]} | \nabla h(u)|.
\end{align*}
From Remark \ref{remH2}, we have that $G \in \mcb S_{Dir} \cap \mcb S_{Neu}$ under  (H2).
 In particular, $\nabla G(0)=0$ and $ \int_0^1 \nabla G(u) du =0$. Thus, from last display, we conclude that
\begin{equation*} \begin{split}
  \sup_{t \in [0,T]} \Big( -\frac{b(1-b) t}{\varepsilon n-\varepsilon}   \sum_{x=1}^{n-1}   \nabla G \left(  \tfrac{x}{n}  \right)   \Big)^2    \leq \Big(  \frac{4Tb(1-b)}{\varepsilon n}  \|\Delta G \|_{\infty} \Big)^2 = \Big(  \frac{T}{\varepsilon n}  \|\Delta G \|_{\infty} \Big)^2. 
\end{split}\end{equation*}
Combining the display in last line with \eqref{replproc}, we conclude that
\begin{equation*} 
\forall n \geq 2, \; \forall \varepsilon \in (0, 1/3), \quad \mathbb{E}_{\nu_b^n} \Big[ \sup_{t \in [0,T]} \Big(  \frac{\varepsilon n - 1}{\varepsilon n - \varepsilon} \tilde{A}_t^{n, \varepsilon}(G) - A_t^{n,\varepsilon}(G) \Big)^2  \Big]    \leq \Big(  \frac{T}{\varepsilon n}  \|\Delta G \|_{\infty} \Big)^2.
\end{equation*}
Combining the bound in last line with \eqref{boundvareps} and an application of \eqref{CSdiscexp} for $m=2$, we conclude that
\begin{equation} \label{bound4charac}
\forall \varepsilon \in (0, 1/3), \quad \mathbb{E}_{\nu_b^n} \big[ \sup_{t \in [0,T]} \big(     \tilde{A}_t^{n,\varepsilon}(G) - A_t^{n,\varepsilon}(G) \big)^2  \big]    \leq 2 \Big(  \frac{T}{\varepsilon n}   \Big)^2 \Big[ [\|\Delta G \|_{\infty}]^2 +  \frac{16}{\varepsilon n}  \sum_{x } [ \nabla G \left(  \tfrac{x}{n}  \right) ]^2 \Big],
\end{equation}
for any $n \geq 2$. Combining \eqref{bound4charac} with \eqref{bound1charac}, \eqref{bound2charac} and an application of \eqref{CSdiscexp} for $m=3$,  there exists a constant $C(G,T)$ independent of $n$ and $\varepsilon$ such that for every $n \geq 2$, it holds
\begin{equation*} 
\forall \varepsilon \in (0, 1/3), \quad \mathbb{E}_{\nu_b^n} \big[ \sup_{t \in [0,T]} \big(     \tilde{A}_t^n(G) - A_t^{n,\varepsilon}(G) \big)^2  \big]    \leq C(G,T) [n^{-1} + \varepsilon^{\gamma-1} + (\varepsilon n)^{-2} + (\varepsilon n)^{-2} (\varepsilon)^{-1}]. 
\end{equation*}
Combining last display with \eqref{claimtight} and \eqref{CSdiscexp}, we conclude that \eqref{limsupcharac} holds.

Next, applying Lemma \ref{lemmult0} for $D_1 =\delta n$ and $D_2 = \varepsilon n$, we get from \eqref{CSdiscexp} that
\begin{equation} \label{bound5charac}
\forall 0 < \delta \leq \varepsilon \in (0, 1/3), \quad  \mathbb{E}_{\nu_b^n} \Big[ \sup_{s \in [0,t]} \Big( \tilde{A}_s^{n,\delta}(G)  -  \tilde{A}_s^{n,\varepsilon}(G)  \Big)^2  \Big] \leq   t \varepsilon^{\gamma-1} \frac{2 C_0}{n} \sum_{x } [ \nabla G \left(  \tfrac{x}{n}  \right) ]^2,
\end{equation}
for any $n \geq 2$. Moreover, from \eqref{bound4charac}, for any fixed $0 < \delta < \varepsilon \in (0, 1/3)$ it holds 
\begin{align*}
& \lim_{n \rightarrow \infty} \mathbb{E}_{\nu_b^n} \big[ \sup_{t \in [0,T]} \big(     \tilde{A}_t^{n,\varepsilon}(G) - A_t^{n,\varepsilon}(G) \big)^2  \big] =  \lim_{n \rightarrow \infty} \mathbb{E}_{\nu_b^n} \big[ \sup_{t \in [0,T]} \big(     \tilde{A}_t^{n,\delta}(G) - A_t^{n,\delta}(G) \big)^2  \big]=0.
\end{align*}
Combining the expression in last line with \eqref{bound5charac} and \eqref{CSdiscexp}, we get
\begin{align*}
\forall 0 < \delta \leq \varepsilon \in (0, 1/3), \quad \lim_{n \rightarrow \infty} \mathbb{E}_{\nu_b^n} \Big[ \sup_{s \in [0,t]} \Big( \tilde{A}_s^{n,\delta}(G)  -  \tilde{A}_s^{n,\varepsilon}(G)  \Big)^2  \Big] \leq   t \varepsilon^{\gamma-1} 2 C_0 \| \nabla G \|^2_{L^2}.
\end{align*}
Recalling \eqref{defprocA}, last display is equivalent to
\begin{align*}
\forall 0 < \delta \leq \varepsilon \in (0, 1/3), \quad \mathbb{E} \big[ \big( \mathcal{A}_{0,t}^{\varepsilon}(G) - \mathcal{A}_{0,t}^{\delta}(G)  \big)^2 \big] \leq 2 C_0 t \varepsilon^{\gamma-1} \| \nabla G \|^2_{L^2}. 
\end{align*}
From Remark \ref{remstat}, the process $(\mathcal{Y}_t)_{0 \leq t \leq T}$ is stationary. Thus, the previous bound shows that $(\mathcal{Y}_t)_{0 \leq t \leq T}$ satisfies \eqref{defEE} for $\kappa=2 C_0>0$ and $\omega=\gamma-1 \in (0,1)$. Therefore, the process $(\mathcal{A}_t)_{0 \leq t \leq T}$ given by $\mathcal{A}_t(G):= \lim_{\varepsilon \rightarrow 0^+} \mathcal{A}_{0,t}^{\varepsilon}(G)$ is well-defined.
Moreover, from \eqref{limsupcharac} we have that the limit in distribution of $\big( A_t^{n_j}(G) \big)_{j \geq 1}$ is given by
$
 A_t(G) = 2 \alpha_a m_a \mathcal{A}_t(G).
$
The previous arguments show that $\mathcal{Y}$ satisfies the conditions (1) and (2 ) stated in Definition \ref{defspdefbe} with $\kappa_1=2 \alpha_a m_a$.

Finally we observe that condition (3) in Definition \ref{defspdefbe} can be obtained by repeating the computations above for the process reversed in time, we leave the details to the reader. 

This  end the proof of Theorem \ref{clt2}.

\appendix

\section{Properties of the regional fractional Laplacian} \label{propfrac}

From \eqref{defdeltaepsilon} and \eqref{deflapfracreg}, for every $G \in C^{2}([0,1])$, it holds
\begin{align} \label{equivlapfrac}
\frac{2}{c^{+}+c^{-}}	\big( \mathbb{L}^{\gamma/2} G \big)(u) = 
	\begin{cases}
		I_{1,G}(u) + I_{2,G}(u) , \quad & 0 < u \leq 1/2, \\
		I_{3,G}(u) + I_{4,G}(u) , \quad & 1/2 < u < 1.
	\end{cases}.
\end{align}
In last display, $I_{1,G}$ and $I_{2,G}$ (resp. $I_{3,G}$ and $I_{4,G}$) are defined for every $u \in (0, \, 1/2]$ (resp. for every $u \in (1/2, \, 2)]$ and are given by
\begin{equation} \label{defI1I2}
	I_{1,G}(u):= \int_0^u \frac{G(u+w) + G(u-w) - 2G(u)}{ w^{1+\gamma}} dw; \quad I_{2,G}(u):= \int_{ u}^{1-u}  \frac{G(u+w) - G(u)}{ w^{1+\gamma}} dw;
\end{equation} 
\begin{equation*} 
	I_{3,G}(u):=  \int_{0}^{1-u}  \frac{G(u+w) + G(u-w) - 2G(u)}{ w^{1+\gamma}} dw; \quad I_{4,G}(u):= \int_{ 1-u}^{u}  \frac{G(u-w) - G(u)}{ w^{1+\gamma}} dw.
\end{equation*}
The goal of this section is to prove the following proposition.
\begin{prop} \label{Lqreg}
	Let $\gamma \in (0, 2)$. If $\gamma \in (0, \; 3/2)$, then $\mathbb{L}^{\gamma/2} G \in  L^2([0,1])$ for every $G \in C^{\infty}([0,1])$ and the operator $ \mathbb{L}^{\gamma/2}: C^{\infty}([0,1]) \rightarrow L^2([0,1])$ is continuous. On the other hand, if $\gamma \in [3/2, \; 2)$, then $\mathbb{L}^{\gamma/2} G \in  L^2([0,1])$ for every $G$ in $\mcb{S}_{Neu}$ and the operator $ \mathbb{L}^{\gamma/2}: \mcb{S}_{Neu} \rightarrow L^2([0,1])$ is continuous.
\end{prop}
In order to obtain last result, by performing second order Taylor expansions on $G$ around $u$, we get
\begin{align} \label{Lqreg1}
	\forall q \geq 1, \quad \int_0^{1/2}  |I_{1,G}(u)|^q du + \int_{1/2}^1  |I_{3,G}(u)|^q du  \leq \Big(  \frac{\| G^{(2)} \|_{\infty}}{2 - \gamma} \Big)^q. 
\end{align}
Therefore, Proposition \ref{Lqreg} is a direct consequence of \eqref{equivlapfrac}, \eqref{Lqreg1} and the following result.

\begin{prop}
	Let $\gamma \in (0,2)$ and $q \geq 1$. Assume that either $\gamma \in (0, 1 + q^{-1})$ and $G \in C^{\infty}([0,1])$; or $\gamma \in [1 + q^{-1}, 2)$ and $G \in \mcb S_{Neu}$. Then there exists a constant $C(q,\gamma)>0$ depending only on $q$ and $\gamma$ such that
	\begin{equation*} 
		\int_0^{1/2} |I_{2,G}(u)|^q du + \int_{1/2}^1 |I_{4,G}(u)|^q du \leq C(q,\gamma) \| G^{(2)} \|^q_{\infty}. 
	\end{equation*}
\end{prop}
\begin{proof}
In the case, $\gamma \in (0, 1 + q^{-1})$,  we can always choose $\delta \in [0,1] \cap (\gamma - q^{-1}, \gamma)$, from the Mean Value Theorem for $G$, the proof ends, while for  $\gamma \in [1 + q^{-1}, 2)$ and $G \in \mcb S_{Neu}$, from the Mean Value Theorem for $G$ and $G^{(1)}$ and the fact that $G^{(1)}(0)=0=G^{(1)}(1)$, the proof ends. We leave the details to the reader.
		
\end{proof}

\section{Discrete convergences} \label{secdiscconv} 

The following result was useful to obtain Proposition \ref{convrem2}. Recall the definition of $\mcb {K}_{n}$ in \eqref{op_Kn}.
\begin{prop} \label{convfrac}
	Assume  $\gamma \in (0, \; 3/2)$ and $G \in C^{\infty}([0,1])$; or $\gamma \in [3/2, \; 2)$ and $G \in \mcb{S}_{Neu}$. Then
	\begin{align*}
		\lim_{n \rightarrow \infty} \frac{1}{n-1} \sum _{x } \big|  n^{\gamma}  \mcb {K}_{n} G \left( \tfrac{x}{n} \right) - \mathbb{L}^{ \gamma/2 } G]\left( \tfrac{x}{n} \right) \big|^2 = 0.
	\end{align*}
\end{prop}
Following the general strategy used to obtain Lemma 5.1 in \cite{HLstefano}, Proposition \ref{convfrac} is a direct consequence of Propositions \ref{convfrac1} and \ref{convfrac2} below.
\begin{prop} \label{convfrac1}
	Assume  $\gamma \in (0, \; 3/2)$ and $G \in C^{\infty}([0,1])$; or $\gamma \in [3/2, \; 2)$ and $G \in \mcb{S}_{Neu}$. Then
	\begin{equation} \label{limconvreg0}
		\begin{split}
			&\varlimsup_{\varepsilon \rightarrow 0^+} \varlimsup_{n \rightarrow \infty} \frac{1}{n-1} \sum _{x=1}^{\varepsilon n} \big|  n^{\gamma}  \mcb {K}_{n} G \left( \tfrac{x}{n} \right) - \mathbb{L}^{ \gamma/2 } G]\left( \tfrac{x}{n} \right) \big|^2 = 0;  \\
			& \varlimsup_{\varepsilon \rightarrow 0^+} \varlimsup_{n \rightarrow \infty} \frac{1}{n-1} \sum _{x=n-\varepsilon n-1}^{n-1} \big|  n^{\gamma}  \mcb {K}_{n} G \left( \tfrac{x}{n} \right) - \mathbb{L}^{ \gamma/2 } G]\left( \tfrac{x}{n} \right) \big|^2.
		\end{split}
	\end{equation}
\end{prop}
\begin{prop} \label{convfrac2}
	Assume  $\gamma \in (0,  2)$ and $G \in C^{\infty}([0,1])$, then
	\begin{equation} \label{limconvreg2}
		\begin{split}
			&\varlimsup_{\varepsilon \rightarrow 0^+} \varlimsup_{n \rightarrow \infty} \frac{1}{n-1} \sum _{x=\varepsilon n + 1}^{(n-1)/2} \big|  n^{\gamma}  \mcb {K}_{n} G \left( \tfrac{x}{n} \right) - \mathbb{L}^{ \gamma/2 } G]\left( \tfrac{x}{n} \right) \big|^2 = 0;  \\
			&\varlimsup_{\varepsilon \rightarrow 0^+} \varlimsup_{n \rightarrow \infty} \frac{1}{n-1} \sum _{x=(n-1)/2 +1}^{n-\varepsilon n-2} \big|  n^{\gamma}  \mcb {K}_{n} G \left( \tfrac{x}{n} \right) - \mathbb{L}^{ \gamma/2 } G]\left( \tfrac{x}{n} \right) \big|^2 = 0.
		\end{split}
	\end{equation}
\end{prop}
\begin{proof}[Proof of Proposition \ref{convfrac1}]
	We prove only the first double limit in \eqref{limconvreg0}, the proof of the other one being analogous. From Proposition \ref{convfrac1} and  Proposition \ref{Lqreg}, we conclude that $|\mathbb{L}^{\gamma/2} G|^{2} \in  L^1([0,1])$. Moreover, 
	\begin{equation} \label{contleb}
		\varlimsup_{\varepsilon \rightarrow 0^+} \varlimsup_{n \rightarrow \infty}\frac{1}{n-1}  \sum_{x=1}^{\varepsilon n - 1} \big| (  \mathbb{L}^{ \gamma/2 } G)\left( \tfrac{x}{n} \right)  \big|^2 \lesssim \lim_{\varepsilon \rightarrow 0^+} \int_0^{\varepsilon} \big| (  \mathbb{L}^{ \gamma/2 } G) (u) \big|^2 du =0.
	\end{equation}
	The proof ends by combining \eqref{op_Kn} with \eqref{defI1I2}, \eqref{equivlapfrac} and \eqref{contleb}.
\end{proof}

\begin{proof} [Proof of Proposition \ref{convfrac2}]
	We prove only the first double limit in \eqref{limconvreg2}, the proof of the other one being analogous. From the steps described in (5.14) of \cite{HLstefano}, we are done if we can show that
	\begin{align}
		&\varlimsup_{\varepsilon \rightarrow 0^+} \varlimsup_{n \rightarrow \infty} \frac{1}{n-1} \sum _{x=\varepsilon n + 1}^{(n-1)/2}  \big| \mathbb{L}_{\varepsilon}^{ \gamma/2 } G \left( \tfrac{x}{n} \right) - \mathbb{L}^{ \gamma/2 } G \left( \tfrac{x}{n} \right) \big|^2 = 0; \nonumber \\
		&\varlimsup_{\varepsilon \rightarrow 0^+} \varlimsup_{n \rightarrow \infty} \frac{1}{n-1} \sum _{x=\varepsilon n + 1}^{(n-1)/2}  \big|  n^{\gamma}  \mcb {K}_{n} G \left( \tfrac{x}{n} \right) - \mathbb{L}_{\varepsilon}^{ \gamma/2 } G \left( \tfrac{x}{n} \right) \big|^2 = 0. \label{limconvreg2b}.
	\end{align}
	Above, $\mathbb{L}_{\varepsilon}^{ \gamma/2 } G$ is given in \eqref{defdeltaepsilon}. Following the same arguments as in (5.14) of \cite{HLstefano}, the limit in first line of last display is bounded from above by a constant times
	\begin{align*}
		\varlimsup_{\varepsilon \rightarrow 0^+} \varlimsup_{n \rightarrow \infty} \frac{1}{n-1} \sum _{x=\varepsilon n + 1}^{(n-1)/2}  \Big( \int_0^{\varepsilon} u^{1 - \gamma} du \Big)^2 \lesssim \varlimsup_{\varepsilon \rightarrow 0^+} \varlimsup_{n \rightarrow \infty} \frac{1}{n} \sum _{x=1}^{n} \varepsilon^{2(2-\gamma)} = \lim_{\varepsilon \rightarrow 0^+} \varepsilon^{2(2-\gamma)} =0.
	\end{align*}
	Above we used the fact that $\gamma < 2$. 
	
	Following (5.5) in \cite{HLstefano}, for every fixed $u \in (0, \, 1/2]$, we define $\theta_u: [0,u] \rightarrow \mathbb{R}$ on $ v \in [0,u],$ by $\theta_u(v):= G(u+v) + G(u-v) - 2 G(u).$
	In order to prove \eqref{limconvreg2b}, following the arguments in (5.20) of \cite{HLstefano}, it is enough to prove that 
	\begin{equation} \label{eq520stef}
		\varlimsup_{\varepsilon \rightarrow 0^+} \varlimsup_{n \rightarrow \infty} \frac{1}{n-1} \sum _{x=\varepsilon n + 1}^{(n-1)/2} [(I)^2 + (II)^2 + (III)^2 ] =0, 
	\end{equation} 
	where 
	\begin{align*}
		& (I):= \Big| n^{\gamma} \sum_{y=\varepsilon n}^{x-1} s(y) \theta_{x/n} \left( \tfrac{y}{n} \right) - \frac{c^{+}+c^{-}}{2} \int_{\varepsilon}^{x/n} u^{-1-\gamma} \theta_{x/n}(u) \, du \Big|; \\
		& (II):= \Big| n^{\gamma} \sum_{y=1}^{\varepsilon n-1} s(y)\theta_{x/n} \left( \tfrac{y}{n} \right) \Big|; \\
		& (III):=\Big| n^{\gamma} \sum_{y=x}^{n-1-x} s(y) [G \left( \tfrac{x+y}{n} \right) - G \left( \tfrac{x}{n} \right) ] - \frac{c^{+}+c^{-}}{2} \int_{x/n}^{(n-x)/n} u^{-1-\gamma} [G \left( \tfrac{x}{n} +u \right) - G \left( \tfrac{x}{n} \right) ]  \, du \Big|.
	\end{align*}
The proof ends by combining the arguments in (5.21), (5.22), (5.23) and (5.27) of \cite{HLstefano}.
\end{proof}
From here on, it will be convenient to make use of further spaces of test functions. Motivated by \eqref{expGSd}, for every $d \geq 1$, we denote
\begin{align*}
	\mathcal{S}_d: = \{G \in C^{\infty}([0,1]): \forall j \in \{0, 1, \ldots, d-1\}, \quad G^{(j)}(0) = 0 = G^{(j)}(1) \}.  
\end{align*}
Moreover, we denote $C^{\infty}([0,1])$ by $\mathcal S_0$. Observe that $\mathcal S_1 = \mcb S_{Dir}$ and $\mathcal S_2 = \mcb S_{Dir} \cap \mcb S_{Neu}$.  

Next we state and prove a result which is analogous to Lemma A.2 in \cite{flucstefano}.
\begin{prop} \label{lemA2stefano}
	Recall the definitions of $r^{\pm}$ and $r_n^{\pm}$ in \eqref{defrpm} and \eqref{defrnpm}, respectively. Assume that $\gamma \in (0,2)$, $q >0$, $d \in \mathbb{N}$, $q\gamma <2d+1$ and $G \in \mathcal {S}_d$. Then, we get
	\begin{align*}
		\lim_{n \rightarrow \infty} \frac{1}{n-1} \sum \big| n^{\gamma} r_n^{-} (\tfrac{x}{n}) - r^{-} (\tfrac{x}{n}) \big|^q G^2(\tfrac{x}{n}) =0 = \lim_{n \rightarrow \infty} \frac{1}{n-1} \sum \big| n^{\gamma} r_n^{+} (\tfrac{x}{n}) - r^{+} (\tfrac{x}{n}) \big|^q G^2(\tfrac{x}{n}). 
	\end{align*}
\end{prop}
\begin{proof}
	We only prove the first limit above, but we observe that the strategy for the second one is analogous. From (B.1) in \cite{BJ}, we get
	\begin{align*}
		\forall x \in \Lambda_n, \quad \big| n^{\gamma} r_n^{+} (\tfrac{x}{n}) - r^{+} (\tfrac{x}{n}) \big| \leq \frac{c^{+}+c^{-}}{2n} \big( \tfrac{x}{n} \big)^{- \gamma - 1}.
	\end{align*}
	Since $G \in  \mathcal {S}_d$, we can combine last display with \eqref{expGSd}, which leads to
	\begin{align} 
		&\frac{1}{n-1} \sum \big| n^{\gamma} r_n^{-} (\tfrac{x}{n}) - r^{-} (\tfrac{x}{n}) \big|^q G^2(\tfrac{x}{n}) \lesssim  \frac{1}{n} \sum \big| \frac{1}{n} \big( \tfrac{x}{n} \big)^{- \gamma - 1} \big|^q (\tfrac{x}{n})^{2d} = n^{q \gamma - 2d - 1} \sum_x x^{2d - q - q \gamma}. \label{bndlemA2stefano}
	\end{align}
	Since $0 <q  \gamma< 2d +1$, there are three possibilities: $q\gamma \in (2d+1 - q, 2d+1)$; $q\gamma =2d+1-q$; and $q\gamma < 2d+1 - q$. The proof ends by treating those cases separately.
\end{proof}

Recall the definitions of $\mcb S_{\beta,\gamma}$ in \eqref{defsbetagamma}, $\Theta(n)$ in \eqref{timescale}, $\mathbb{L}_{\alpha,\beta}^{ \gamma / 2 }$ in \eqref{deltaab} and $\mcb{K}_{n, \alpha, \beta}$ in \eqref{op_Knb}. Choosing $q=2$ in Proposition \ref{lemA2stefano} and combining it with Proposition \ref{convfrac}, we get next result, which was crucial to obtain Proposition \ref{convrem1}. 
\begin{cor} \label{convknbeta}
	Assume $(\alpha, \beta, \gamma) \in (0, \infty) \times \mathbb{R} \times (0,  2)$ and $G \in \mcb S_{\beta,\gamma}$. Then
	\begin{align*}
		\lim_{n \rightarrow \infty} \frac{1}{n-1} \sum _{x } \big[  \Theta(n)  \mcb {K}_{n, \alpha, \beta} G \left( \tfrac{x}{n} \right) - \mathbb{L}_{\alpha,\beta}^{ \gamma / 2 } G\left( \tfrac{x}{n} \right) \big]^2 = 0. 
	\end{align*}
\end{cor}

\section{Useful computations} \label{useful} 
\begin{prop} \label{contlemA2stef}
	Recall the definition of $r^{-}$ and $r^{+}$ in \eqref{defrpm}. Assume that $\gamma \in (0,2)$, $q \in \{1,2\}$, $d \in \mathbb{N}$, $q \gamma < 2 d+1$ and $G \in \mathcal{S}_d$. Then, 
	\begin{align*}
		\int_{0}^1 [r^{-}(u) + r^{+}(u) ]q |G(u)|^{2} du \lesssim \| G^{(d)} \|^{2}_{\infty} < \infty.
	\end{align*}
\end{prop}
\begin{proof}
	The proof is a direct consequence of \eqref{CSdisc}, \eqref{expGSd} and \eqref{defrpm}.
\end{proof}

Recall  \eqref{defsbetagamma}. Next we present the proof of Proposition \ref{norcont}.
\begin{proof} [Proof for Proposition \ref{norcont}]
Recall  \eqref{seminorm}. From the Mean Value Theorem
\begin{equation} \label{semnormfrac}
	\forall G \in  C^{\infty}([0,1]),  \quad \| G \|^2_{\gamma/2} \leq \| G^{(1)} \|^2_{\infty}  \frac{c^{+}+c^{-}}{2(2-\gamma)(3-\gamma)}.
\end{equation}
The proof ends by combining \eqref{semnormfrac} with Proposition \ref{contlemA2stef}.
\end{proof}

Now we present a result which is useful to treat $\hat{\mathcal{A}}_{n,\beta} (G)$, given by \eqref{op_Anb}.

\begin{prop} \label{prop1lem1convmart}
	Let $(\beta,\gamma) \in  \mathbb{R} \times (0,2)$ and $G \in C^{\infty}([0,1])$. Then
	\begin{align*}
		\lim_{n \rightarrow \infty} \hat{\mathcal{A}}_{n,\beta} (G) =  \mathbbm{1}_{  \{ \beta \geq 0 \} } \frac{c^{+}+c^{-}}{2}  \iint_{(0,1)^2} \frac{[G(v)-G(u)]^2}{|v-u|^{1+\gamma}} \, du \, dv.
	\end{align*}
\end{prop}
\begin{proof}
The proof is a direct consequence of \eqref{op_Anb}, \eqref{defsa} and \eqref{timescale}.
\end{proof}
Next we present a result which is useful to treat $\mathcal{B}_{n,\beta} (G)$, given by \eqref{op_Bnb}.
\begin{prop} \label{prop2lem1convmart}
	Let $(\beta,\gamma) \in  \mathbb{R} \times (0,2)$. Then 
	\begin{align*}
		\lim_{n \rightarrow \infty} \mathcal{B}_{n,\beta} (G) =  
		\begin{cases}
			\int_0^1 [r^{-}(u) + r^{+}(u)] G^2(u) \, du, \quad &(\beta,\gamma) \in (- \infty,0] \times (0, 1), \; G \in C^{\infty}([0,1]); \\
			\int_0^1 [r^{-}(u) + r^{+}(u)] G^2(u) \, du, \quad &(\beta,\gamma) \in (- \infty,0] \times [1, 2), \; G \in \mcb S_{Dir}; \\
			0,  \quad  & \gamma \in (1,2), \; \beta > \gamma - 1>0, \;  G \in C^{\infty}([0,1]); \\
			0,  \quad  & (\beta,\gamma) \in  (0, \infty) \times (0, 1], \;  G \in C^{\infty}([0,1]); \\
			0,  \quad  & (\beta,\gamma) \in  (0, \infty) \times (1, 2), \;  G \in  \mcb S_{Dir}.
		\end{cases}
	\end{align*}
\end{prop}
\begin{proof}
By combining \eqref{timescale} and \eqref{op_Bnb} with an application of Propositions \ref{lemA2stefano} and \ref{contlemA2stef} for $q=1$ and $d=0$, resp. $q=1$ and $d=1$, we get the desired result for the first, resp. second case, in last display.
	Next we observe that from \eqref{defrnpm} and \eqref{defsa}, we have that	
\begin{equation} \label{boundrnpm}
 \forall n \geq 2, \forall x \in \Lambda_n, \quad r_n^{-}(\tfrac{x}{n}) \leq \frac{(c^{+}+c^{-})(\gamma+1)}{2 \gamma} x^{-\gamma}, \quad \; r_n^{+}(\tfrac{x}{n}) \leq \frac{(c^{+}+c^{-})(\gamma+1)}{2 \gamma} (n-x)^{-\gamma}.
	\end{equation} 		
In the remaining three cases in the statement of Proposition \ref{prop2lem1convmart}, we have $\beta > 0$. Thus, combining \eqref{timescale} with \eqref{op_Bnb} and \eqref{boundrnpm}, we get	
\begin{align} \label{boundB01}
		\mathcal{B}_{n,\beta} (G) \lesssim n^{\gamma - 1 - \beta} \sum_x \big[ x^{-\gamma} + (n-x)^{-\gamma} \big] \| G \|^2_{\infty} \lesssim n^{\gamma - 1 - \beta} \sum_x x^{-\gamma}.
	\end{align}	
From \eqref{boundB01}, the proof ends for the third and fourth cases in the statement of Proposition \ref{prop2lem1convmart}.	
Finally, by combining \eqref{timescale} with \eqref{op_Bnb}, \eqref{boundrnpm} and an application \eqref{expGSd} for $d=1$, we obtain the desired result for the case $(\beta,\gamma) \in  (0, \infty) \times (1, 2)$ and $G \in \mcb S_{Dir}=\mathcal{S}_1$.
	
\end{proof}
We end this section by presenting the proof of Proposition \ref{aproxL2gen}.

\begin{proof}[Proof for Proposition \ref{aproxL2gen}  ]
	Let $j \geq 1$. Combining the assumption that $\widetilde{G} \in L^2([0,1])$ with the definition of $A_{\widetilde{G}}$ in \eqref{extL1} and Proposition 4.20 of \cite{brezis2010functional}, we have that $\phi_{j}* A_{\widetilde{G}} \in C^{\infty}(\mathbb{R})$. Thus, from \eqref{defpsij}, we have that $\psi_{j} \cdot (\phi_{j}* A_{\widetilde{G}} ) \in C_c^{\infty}(\mathbb{R})$ and $\psi_{j} \cdot (\phi_{j}* A_{\widetilde{G}} ) \equiv 0$ on $(-\infty, \, 1/j]$ and on $[1 - 1/j, \, \infty)$, concluding that $H^{\widetilde{G}}_j \in \mcb S$, for any $j \geq 1$.
	
	Next we will prove that $H_j$ converges to $\widetilde{G}$ in $L^2([0,1])$ as $j \rightarrow \infty$. In order to do so, let $\varepsilon >0$. From Theorem 4.22 of \cite{brezis2010functional}, we can fix $j_0 \geq 1$ such that
	\begin{align} \label{convL2a}
		\forall j \geq j_0, \quad \int_{ 0 }^1 [(\phi_{j}* A_{\widetilde{G}} ) - A_{\widetilde{G}} ]^2 (u) du \leq  \int_{ \mathbb{R} } [(\phi_{j}* A_{\widetilde{G}} ) - A_{\widetilde{G}} ]^2 (u) du < \varepsilon.
	\end{align}
	From \eqref{defpsij}, we have that $\psi_j \equiv 1$ on $[2/j, \, (j-2)/j]$, which leads to
	\begin{equation} \label{convL2b}
		\forall j \geq 1, \quad \int_{ 2/j }^{(j-2)/j} [ \psi_{j} \cdot (\phi_{j}* A_{\widetilde{G}}) -  ( \phi_{j}* A_{\widetilde{G}}) ]^2 (u) du =0.
	\end{equation}
The proof ends by combining \eqref{defphij}, \eqref{extL1} and H\"older's inequality with \eqref{defpsij}, \eqref{convL2b} and \eqref{convL2a}.
\end{proof}

\subsection*{Acknowledgements}
P.C. was funded by the Deutsche Forschungsgemeinschaft (DFG, German Research Foundation) under Germany’s Excellence Strategy – EXC-2047/1 – 390685813. P.G. thanks Funda\c c\~ao para a Ci\^encia e Tecnologia FCT/Portugal for financial support through the projects UIDB/04459/2020 and UIDP/04459/2020.  The authors thank Nicolas Perkowski for the nice discussions related to the uniqueness of energy solutions.

\bibliographystyle{plain}
\bibliography{bibliografia}

\end{document}